\documentclass[onefignum,onetabnum]{siamonline250211}

\usepackage{tikz-cd}
\usepackage{lipsum}
\usepackage{amsfonts}
\usepackage{graphicx}
\usepackage{epstopdf}
\usepackage{algorithmic}
\ifpdf
  \DeclareGraphicsExtensions{.eps,.pdf,.png,.jpg}
\else
  \DeclareGraphicsExtensions{.eps}
\fi
\usetikzlibrary{arrows.meta}

\PassOptionsToPackage{hyphens}{url}
\usepackage{xypic}
\usepackage{tikz}
\usetikzlibrary{math}
\usepackage{amsmath, amssymb}
\usepackage{subcaption}

\DeclareMathOperator*{\argmin}{arg\,min}
\usetikzlibrary{arrows.meta,positioning,calc}

\newcommand{\rmH}{\mathrm{H}}
\newcommand{\rmC}{\mathrm{C}}
\newcommand{\rmZ}{\mathrm{Z}}
\newcommand{\rmB}{\mathrm{B}}

\newcommand{\proj}{\mathrm{proj}}
\newcommand{\Sch}{\mathrm{Sch}}
\newcommand{\birthcochain}{\eta}

\usepackage{enumitem}
\setlist[enumerate]{leftmargin=.5in}
\setlist[itemize]{leftmargin=.5in}

\newsiamremark{remark}{Remark}
\newsiamremark{hypothesis}{Hypothesis}
\newsiamremark{notation}{Notation}

\crefname{hypothesis}{Hypothesis}{Hypotheses}
\newsiamthm{claim}{Claim}
\newsiamremark{fact}{Fact}
\newsiamremark{observation}{Observation}

\crefname{fact}{Fact}{Facts}

\headers{Birth and death cochains}{T. Weighill and L. Zhou}

\title{Topological optimization with birth and death cochains\thanks{
\funding{L.Z. received support from an AMS--Simons Travel Grant.}}}

\author{
Thomas Weighill\thanks{Department of Mathematics and Statistics, University of North Carolina at Greensboro, Greensboro, NC 
  (\email{t\_weighill@uncg.edu}).}
  \and  Ling Zhou\thanks{Department of Mathematics, Duke University, Durham, NC 
  (\email{ling.zhou@duke.edu}).}
  }

\usepackage{amsopn}

\begin{document}

\maketitle
\begin{abstract}
We introduce the notion of birth and death cochains as generalized versions of birth and death simplices in persistent cohomology. We show that birth and death cochains (unlike birth and death simplices) are always unique for a given persistent cohomology class. We use birth and death cochains to define birth and death content as generalizations of birth and death times. We then demonstrate the advantages of using that birth and death content as loss functions on a variety of topological optimization tasks with point clouds, time series and scalar fields. We close with a novel application of topological optimization to a dataset of arctic ice images.
\end{abstract}



\section{Introduction}

Persistent homology and cohomology quantify the shape of data by tracking the topology of a datset at various thresholds (or scales) \cite{edelsbrunner2000topological,
zomorodian2005computing,
edelsbrunner2008persistent,
carlsson2009topology,
de_Silva_2009}. 
A topological feature such as  a connected component or hole is ``born'' at one threshold and ``dies'' at later threshold. The time between these two events is the persistence of the feature, and this value often quantifies its importance in the overall structure of the data. Birth and death times are recorded as pairs $[b, d)$ to create a persistence diagram, which is known to be stable with respect to perturbations in the underlying data~\cite{cohen2005stability}. 

The birth--death pair $[b, d)$ records only partial information about a topological feature: it specifies the parameter values at which the feature appears and disappears, but not the specific changes in the underlying filtered spaces responsible for these events. 
In practice, these events are realized by particular simplices in the filtration. We refer to the simplex whose addition creates (resp.\ destroys) a feature as its \emph{birth simplex} (resp.\ \emph{death simplex}); see \Cref{fig:introexample} for a toy example. 
These birth and death simplices, despite being absent from a persistence diagram, are important in many applications. For example, they can be used to localize or visualize topological features in the underlying data~\cite{duchin2022homological,hickok2024persistent}, and they play an important role in several optimization frameworks involving persistent (co)homology~\cite{carriere2021optimizing,bubenik2023topological}.  Despite their utility, birth and death simplices have some notable drawbacks:
\begin{itemize}
    \item [(a)] they are highly localized: modifying a birth or death simplex amounts to changing only a small portion of the data,
    \item [(b)] they can fail to be well-defined (if two simplices share a filtration value), and
    \item[(c)] they can be highly unstable: birth and death simplices can be altered by small perturbations in the data (unlike birth and death times).
\end{itemize}
For example, in the point cloud example in \Cref{fig:introexample}, trying to decrease the birth time of the hole by focusing on the birth edge alone amounts to changing just two points. Moreover, a tiny disturbance of the point could change the location of the green edge while still preserving the same ``hole''. 

\begin{figure}
    \centering
    \includegraphics[width=0.8\linewidth]{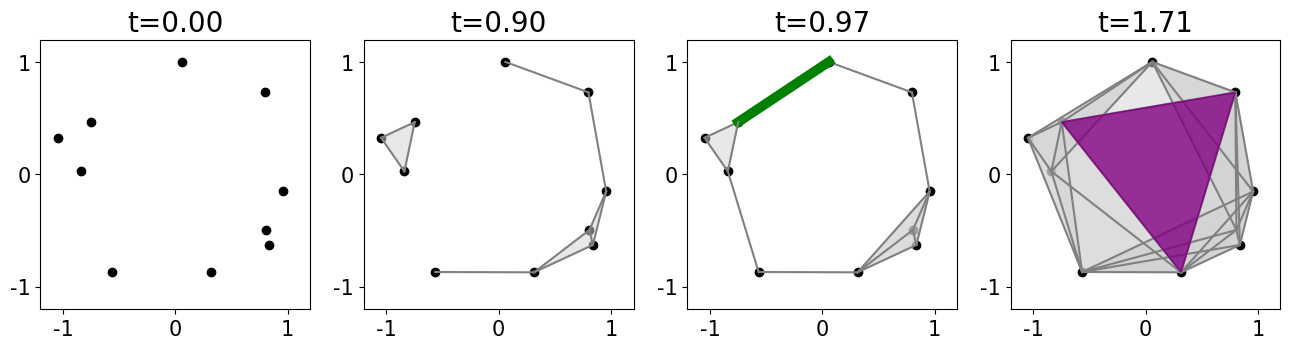}
    \caption{The classical notion of birth simplex (in green) and death simplex (purple) for degree-$1$ persistent homology.}
    \label{fig:introexample}
\end{figure}

In this paper, we attempt to address these shortcomings by relaxing the notion of birth and death simplices. 
We introduce \emph{birth and death cochains}, which are linear combinations of (dual) simplices relevant to the birth and death of a feature respectively. In the most general definition, birth and death cochains are defined for a simplicial pair $K\subseteq L$ as the solution to $\ell^2$ optimization problems which generalize those related to harmonic cochains. Among our theoretical results, we prove the uniqueness of these objects, and relate them to relative cohomology, Laplace learning, and persistent Laplacians (\Cref{sec:theory}). 

When dealing with persistent cohomology, $K$ and $L$ become snapshots of the filtration just before and after birth or death (see \Cref{fig:epsilon-timeline}), giving rise to \emph{$\varepsilon$-birth and $\varepsilon$-death cochains}. Drawbacks (a) and (b) listed above are avoided by $\varepsilon$-birth and $\varepsilon$-death cochains since they involve multiple simplices at once, and are always unique (see \Cref{rem:uniqueness}) and well-defined for a given persistent cohomology class. While we have good reason not to expect a theoretical guarantee of stability for $\varepsilon$-birth and $\varepsilon$-death cochains, we have a heuristic argument for why they may be more stable than birth and death simplices in practice (see \Cref{sec:instability}). Our hypothesis is supported by better convergence to regular configurations for point-clouds (\Cref{sec:pointclouds}), and more robust behavior on lower-star filtrations (\Cref{sec:lowerstar}).

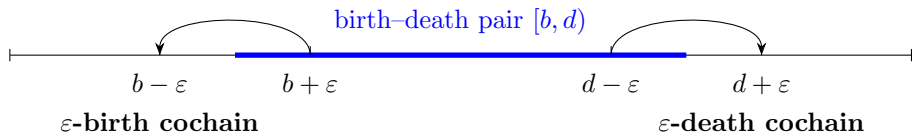
\begin{figure}[h]
\centering
\begin{tikzpicture}[font=\small, >=Stealth]

\draw[|-|] (0,0) -- (12,0);

\coordinate (bme) at (2,0);
\coordinate (bpe) at (4,0);
\coordinate (dme) at (8,0);
\coordinate (dpe) at (10,0);

\coordinate (b) at (3,0);
\coordinate (d) at (9,0);

\foreach \pt/\lbl in {bme/$b-\varepsilon$, bpe/$b+\varepsilon$, dme/$d-\varepsilon$, dpe/$d+\varepsilon$}{
  \draw (\pt) -- ++(0,0.12);
  \node[below=3pt] at (\pt) {\lbl};
}

\draw[line width=2pt, blue] ($(b)$) -- ($(d)$);
\node[above=4pt, blue] at ($(b)!0.5!(d)$) {birth--death pair $[b, d)$};

\draw[<-] (bme.north) .. controls +(0,0.6) and +(0,0.6) .. (bpe.north);
\draw[->] (dme.north) .. controls +(0,0.6) and +(0,0.6) .. (dpe.north);

\node[below=18pt] at ($(bme)!0.5!(bme)$) {\textbf{$\varepsilon$-birth cochain}};
\node[below=18pt] at ($(dpe)!0.5!(dpe)$) {\textbf{$\varepsilon$-death cochain}};

\end{tikzpicture}
\caption{The $\varepsilon$-birth cochain which lives at $b-\varepsilon$ is computed for a cochain at ${b+\varepsilon}$, and the $\varepsilon$-death cochain which lives at $d+\varepsilon$ is computed for a cochain at $d-\varepsilon$. See \Cref{sec:birth-death cochain}.}
\label{fig:epsilon-timeline}
\end{figure}

We use $\varepsilon$-birth and $\varepsilon$-death cochains to introduce a relaxation of the notion of birth time and death time, which we call \emph{$\varepsilon$-birth content} and \emph{$\varepsilon$-death content} respectively. These quantities approximate classical birth and death times for small $\varepsilon$, but depend on multiple simplices, making them perform better as objective functions in topological optimization tasks. \Cref{fig:smallcochains} shows a visualization of topological optimization on point clouds using birth and death cochains; note how multiple simplices are involved at each gradient ascent step. We prove that regular polygons are critical points of Vietoris-Rips $\varepsilon$-persistence content (i.e.~$\varepsilon$-death content minus $\varepsilon$-birth content) under suitable constraints (\Cref{thm:criticalmoment}). By contrast, birth and death simplices are not even well-defined for regular polygons, and are very unstable for almost-regular configurations.

\begin{figure}
    \centering
    \begin{tikzpicture}
        \node at (0,0) {\includegraphics[width=15cm]{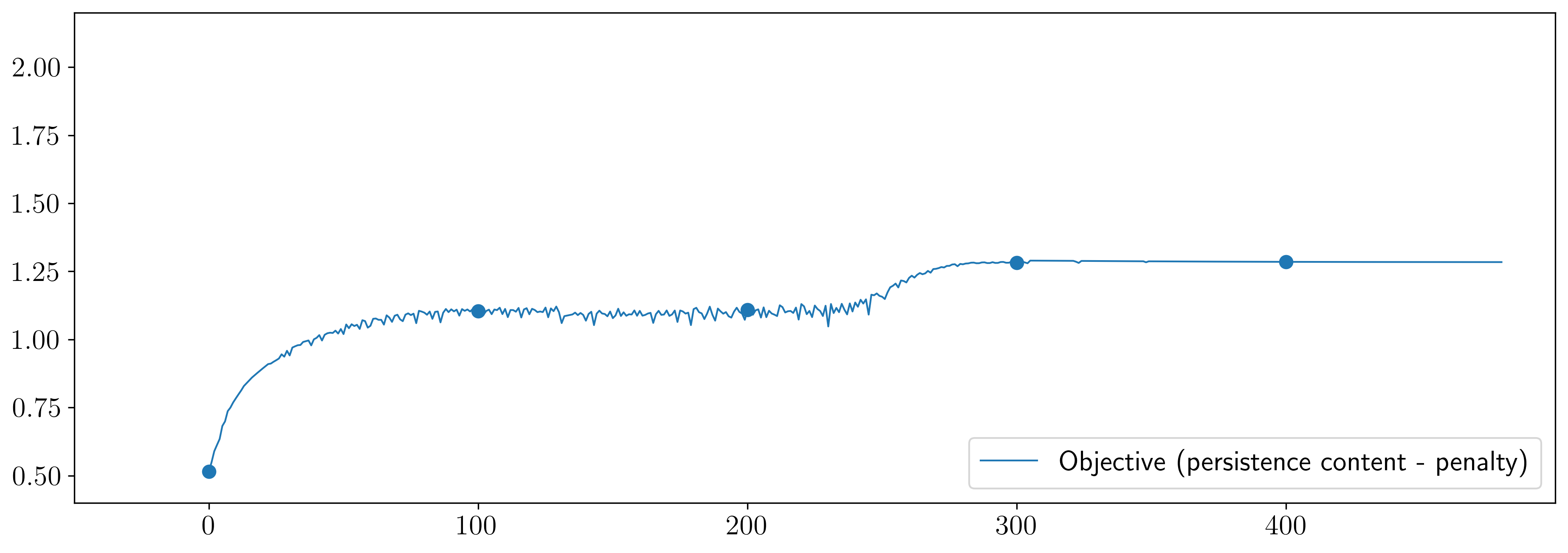}};

        \foreach \i in {0,100,200,300,400}{
        \node at (\i/38.3-5.64, 1.5) {\includegraphics[width=2cm]{pointmover/cochains_step\i.png}};
        }
    \end{tikzpicture}

    \caption{Optimizing a point cloud using birth and death cochains (see \Cref{sec:pointclouds}). We plot the objective function to be maximized, along with snapshots of the points cloud every 100 iterations. The birth cochain is shown using green edges and the death cochain is shown using purple triangles; these are the simplices that are adjusted in the gradient ascent step.}
    \label{fig:smallcochains}
\end{figure}

While birth and death cochains can be useful in any application involving birth and death simplices, our focus in this paper is on topological optimization, where a topological loss function is minimized either as a standalone objective or as part of a regularizer. To demonstrate the advantages of birth and death cochains, we deploy them in three model optimization tasks:
\begin{itemize}
    \item On point clouds, we show that using cochains instead of simplices often gives better long-term convergence and avoids local optima, leading to improved normalized persistence  (\Cref{sec:pointclouds}). 
    \item On image data, optimization using cochains avoids arbitrary choices and results in smoother, more realistic solutions (\Cref{sec:lowerstar}).
    \item For feature selection, we find that even a naive line-search method effectively finds features relevant to a topological feature when using cochains (\Cref{sec:features}). 
\end{itemize}

We conclude our demonstration of cochains in topological optimization with a novel application to Arctic ice extent data (\Cref{sec:arctic}). We deploy the naive line-search method for feature selection on satellite images of ice, revealing the specific assymetry in the melting and freezing cycle that causes the ``hole'' in the data observed by persistent homology, confirming the hypothesis in~\cite{ilagor2023visualizing}.

\section{Preliminaries}\label{sec:preliminaries}

Throughout this paper, all simplicial complexes are finite. In this section we fix notation and recall the basic notions from simplicial cohomology and persistent cohomology that will be used throughout.

\subsection{Cohomology and Harmonic Representatives}
All (co)homology groups are taken with coefficients in $\mathbb{R}$. 
For a simplicial complex $X$ and a degree $k$, the space of $k$-cochains $\rmC^k(X)$ consists of all real-valued functions on the set $X^k$ of $k$-simplices of $X$. 
We equip $\rmC^k(X)$ with the standard $\ell_2$ inner product defined by
\[
\langle \alpha, \beta \rangle
:= \sum_{\sigma \in X^k} \alpha(\sigma)\,\beta(\sigma),
\qquad \alpha, \beta \in \rmC^k(X),
\]
and denote the induced norm by $\|\alpha\|_2$.

We denote by $\delta^k : \rmC^k(X) \to \rmC^{k+1}(X)$ the $k$-th coboundary operator and by $(\delta^k)^*$ its adjoint with respect to the $\ell_2$ inner product. More generally, we use $(\cdot)^*$ to denote the adjoint of a linear operator between finite-dimensional inner product spaces.
A $k$-cochain $\alpha$ is called a \emph{cocycle} if $\delta^k \alpha = 0$, and a \emph{coboundary} if $\alpha \in \operatorname{im} \delta^{k-1}$. We denote the spaces of $k$-cocycles and $k$-coboundaries by $\rmZ^k(X) := \ker \delta^k$ and $\rmB^k(X) := \operatorname{im} \delta^{k-1}$, respectively.
The $k$-th \emph{cohomology group} is defined as
\(
\rmH^k(X) := \rmZ^k(X) \big/ \rmB^k(X).
\)
For a cocycle $\alpha$, we write $[\alpha]$ to denote its cohomology class.

Given a cohomology class $[\alpha] \in \rmH^k(X)$, its \emph{harmonic representative} is the unique cocycle $\hat{\alpha}$ in the class $[\alpha]$ that minimizes the $\ell_2$ norm. In other words, it is the unique element in
\[
\argmin \left\{ \|\alpha'\|_2 \; \middle|\; \alpha' \in \rmC^k(X),\; \alpha' \in [\alpha] \right\}.
\]
Equivalently, it is the unique cochain $\hat{\alpha} \in [\alpha]$ satisfying
\(
\Delta^k \hat{\alpha} = 0,
\)
where
\(
\Delta^k := (\delta^k)^*\delta^k + \delta^{k-1}(\delta^{k-1})^*
\)
denotes the $k$-th combinatorial Hodge Laplacian.

\subsection{Persistent Cohomology}
\label{subsec:persistent cohomology}
A \emph{filtration} of topological spaces is a collection $\{X_t\}_{t \in \mathbb{R}}$ such that $X_t \subseteq X_s$ whenever $t \leq s$. 
In this paper, we focus on filtrations of finite simplicial complexes, which can be represented as a finite nested sequence of subcomplexes:
\[
X_{t_0} \subseteq X_{t_1} 
\subseteq \cdots \subseteq X_{t_n} = X.
\]
That is, there exist $t_0 < t_1 < \cdots < t_n$ such that $X_t = X_{t_i}$ for all $t \in [t_i, t_{i+1})$.

We primarily consider filtrations of the following types:
\begin{itemize}
    \item The \emph{sublevel-set filtrations}, defined as 
    $\{X_t\}_{t\in\mathbb{R}}$ with 
    $X_t=\left\{\sigma\in X \;\middle|\; f(\sigma)\le t\right\}$ 
    for a simplicial complex $X$ and a real-valued function 
    $f$ on the simplices of $X$.
    \item The \emph{lower star filtration}, $\{\mathrm{LowSt}(K, t)\}_{t\in \mathbb{R}}$, for a simplicial complex $K$ with vertex set $V$ and $f: V \to \mathbb{R}$, where $\mathrm{LowSt}(K, t)$ consists of simplices whose vertices have function value at most $t$.
    \item The \emph{Vietoris–Rips filtrations}, defined as $\{\mathrm{VR}(M;t)\}_{t\in\mathbb{R}}$ for a metric space $(M,d)$, where $\mathrm{VR}(M;t)$ consists of all finite subsets $\sigma\subseteq M$ with $\max_{x,y\in\sigma} d(x,y)\le t$. Equivalently, this is the sublevel-set filtration of the diameter function $f(\sigma)=\max_{x,y\in\sigma} d(x,y)$ on the full simplex over $M$.
\end{itemize}

Throughout this paper, we work with cohomology rather than homology. While both theories provide the same information for finite simplicial complexes over a field (by the universal coefficient theorem), cohomology is more convenient for our purposes, especially when working with harmonic representatives and optimization problems.

Given a filtration $\{X_t\}_{t \in \mathbb{R}}$, the $k$-th cohomology groups $\rmH^k(X_t)$ assemble into a persistence module, with structure maps induced by the inclusions $X_t \hookrightarrow X_s$ for $t \leq s$:
\begin{equation*} \label{eq:module}
    \rmH^k(X_{t_0}) \leftarrow \rmH^k(X_{t_1}) \leftarrow \cdots \leftarrow \rmH^k(X_{t_n}).
\end{equation*} 
This persistence module is referred to as the \emph{$k$-th persistent cohomology} of the filtration.

By the structure theorem for persistence modules 
(see, e.g.,~\cite{crawley2015decomposition,oudot2015persistence}), 
the $k$-th persistent cohomology decomposes as a direct sum of interval modules (i.e., persistence modules supported on a single interval). 
This decomposition is encoded by a multiset of intervals,  
called the \emph{barcode} (or persistence diagram) in degree $k$. 
We adopt the convention that persistence intervals, also called \emph{bars}, are half-open of the form $[b,d)$, so that a class exists for $b \le t < d$ and disappears at $t=d$. 
For each bar $[b_i,d_i)$, we refer to $b_i$ and $d_i$ as its \emph{birth} and \emph{death} times.

\subsection{Relative Cohomology}
\label{sec:relative cohomology}
For clarity, when necessary, we add the underlying space as a subscript to the coboundary operator and cohomology class, e.g., $\delta^k_X$ and $[\cdot]_X$, to indicate that the coboundary or cohomology class is taken with respect to the complex $X$.

Consider a subcomplex $K$ of a complex $L$. The \emph{relative cochain group} $\rmC^k(L, K)$ consists of $k$-cochains on $L$ that vanish when restricted to $K$; that is,
\[
\rmC^k(L, K) := \left\{\, \varphi_L \in \rmC^k(L) \;\middle|\; \varphi_L|_K = 0 \, \right\},
\]
where $\varphi_L|_K \in \rmC^k(K)$ denotes the restriction of $\varphi_L$ to the $k$-simplices of $K$.
These groups assemble into the \emph{relative cochain complex}
\[
\cdots \to \rmC^{k-1}(L, K) \xrightarrow{\delta^{k-1}} \rmC^{k}(L, K) \xrightarrow{\delta^k} \rmC^{k+1}(L, K) \to \cdots,
\]
where the coboundary maps are induced from those on $L$. The relative cohomology groups $\rmH^k(L, K)$ are defined as the cohomology of this complex. The standard inner product on $\rmC^k(L)$ restricts to an inner product on $\rmC^k(L, K)$.

Let $\iota: K \hookrightarrow L$ be the inclusion of simplicial complexes. 
The \emph{restriction map} $\iota^\#: \rmC^k(L) \to \rmC^k(K)$ sends a cochain $\phi \in \rmC^k(L)$ to its restriction on the $k$-simplices of $K$. 
The \emph{extension by zero map} $\epsilon: \rmC^k(K) \hookrightarrow \rmC^k(L)$ embeds a cochain on $K$ into $L$ by assigning zero to all $k$-simplices in $L \setminus K$. 
By standard abuse of notation, we also write $\iota^\#: \rmH^k(L) \to \rmH^k(K)$ for the induced map on cohomology.
With respect to the standard inner product on cochains, the subspaces $\rmC^k(L, K)$ and $\epsilon(\rmC^k(K))$ are orthogonal complements in $\rmC^k(L)$.
Accordingly, we have the orthogonal decomposition
\begin{equation}\label{eq:orth-decomp}
\rmC^k(L)
=
\rmC^k(L,K)^\perp \oplus \rmC^k(L,K),
\qquad
\rmC^k(L,K)^\perp = \epsilon(\rmC^k(K)).
\end{equation}

\section{Birth and death cochains}\label{sec:theory}

\subsection{Birth and death cochains for simplicial pairs}
Let $K$ be a subcomplex of the simplicial complex $L$. We study the change in simplicial cohomology from $K$ to $L$. 
There are two primary scenarios, informally described as:
\begin{itemize}
    \item (Birth) A cohomology class which does not exist at $K$ appears at $L$, and
    \item (Death) A cohomology class which exists at $K$ ``dies'' (i.e.~is no longer present) in $L$.
\end{itemize}

More formally, we have the following definitions.

\begin{definition}\label{def:born-die}
    Let $K$ be a subcomplex of the simplicial complex $L$. We say a cohomology class $[\alpha]_L \in \rmH^k(L)$ is \emph{born between $K$ and $L$} if $[\alpha\vert_K] = 0 \in \rmH^k(K)$. We say that a cohomology class $[\beta]_K \in \rmH^k(K)$ \emph{dies between $K$ and $L$} if $\beta$ cannot be extended to a $k$-cocycle on $L$. 
\end{definition}

Suppose $[\alpha]_L \in \rmH^k(L)$ is born between $K$ and $L$. We begin by noting that there is a representative $\hat{\alpha} \in  [\alpha]_L$ whose restriction to $K$ is the zero cochain (not just zero in cohomology). Indeed, $\alpha\vert_{K}$ is a coboundary in $K$, so $\alpha\vert_{K} = \delta_{K}^k f$ for some $f \in \rmC^{k-1}(K)$. Extending $f$ by zero, we note that $\alpha - \delta_{L}^k f$ is cohomologous to $\alpha$ and restricts to zero on $K$ as required. 

\begin{definition}[Birth cochain]\label{def:birth-cochain}
    Let $K \subseteq L$ and let $[\alpha]_L \in \rmH^k(L)$ be born between $K$ and $L$. The \emph{birth cochain for $\alpha$ from $K$ to $L$} is the unique element $\hat{\alpha}$ of
    $$
        \argmin_{\alpha' \in \rmC^k(L)} \left\{ ||\alpha'||_2\;\middle|\;\alpha'\vert_K = 0,\ \alpha' \in [\alpha]_L \right\}.
    $$
\end{definition}

Uniqueness of $\hat{\alpha}$ is addressed in \Cref{rem:uniqueness} below. \Cref{fig:exampleb} illustrates a simple example of a birth cochain. 
In the special case of $K=\emptyset$, the condition $\hat{\alpha}|_K=0$ is vacuous, and the birth cochain reduces to the $\ell^2$-minimal (harmonic) representative of $[\alpha]_L$. Thus, birth cochains generalize harmonic representatives to the setting of simplicial pairs.
At the opposite extreme, if $K=L$, no nontrivial class can be born between $K$ and $L$.

We now consider death, which happens when there is an $[\beta]_K \in \rmH^k(K)$ which cannot be extended to a cocycle in $\rmH^k(L)$. 

\begin{definition}[Death cochain]\label{def:death cochain}
    Let $K \subseteq L$ and let $[\beta]_K \in \rmH^k(K)$ die between $K$ and $L$. The \emph{death cochain for $\beta$ from $K$ to $L$} is $\delta^k_{L} \hat{\beta}$, where $\hat{\beta}$ is any element in
    \begin{equation}\label{eq:deat-cochain-def}
        \argmin_{\beta' \in \rmC^{k}(L)} \left\{ \|\delta^k_{L} \beta'\|_2\;\middle|\; \beta'\vert_K = \beta \right\}.
    \end{equation}
    Any such $\hat{\beta}$ is called a \emph{death potential}.
\end{definition}

\Cref{fig:exampled} shows a simple example of a death cochain.
In \Cref{prop:death cochain}, we will interpret the death cochain as the $\ell_2$-minimal representative of the relative cohomology class $[\delta_L^k \epsilon(\beta)]_{(L,K)}$, where $\epsilon$ denotes the extension-by-zero map and the subscript $(L,K)$ indicates that the cohomology class is taken in $\rmH^{k+1}(L,K)$.
In the special case when $L$ and $K$ have the same $k$-skeleton (e.g. $L=K$), the death potential can only be $\beta$ itself. 

While the birth cochain clearly does not depend on the choice of the representative $\alpha$, the same is not immediately clear for the death cochain. While a direct proof is possible, the fact that the death cochain is independent of the choice of representative will actually follow from the formulation of birth and death cochains in terms of relative cohomology in the next section; see \Cref{cor:welldefined}. 
For now we note that with this in mind: we can reformulate the optimization problem in \Cref{def:death cochain} as 

$$
       \argmin_{\omega \in \rmC^{k+1}(L)} \left\{ ||\omega||_2\;\middle|\;\omega = \delta^k_{L} \beta',\, \beta'\vert_K \in [\beta]_K \right\}.
$$

\begin{remark}[Uniqueness of birth and death cochains]\label{rem:uniqueness}
The minimizers in \Cref{def:birth-cochain,def:death cochain} are unique.
In each case the feasible set is an affine subset of a finite-dimensional
cochain space defined by linear constraints, and the squared $\ell_2$ norm
is strictly convex. Hence the minimizer is uniquely determined.
\end{remark}

\begin{figure}[h]
    \centering
    \begin{subfigure}[t]{0.33\textwidth}
    \centering
        \includegraphics[height=2.5cm]{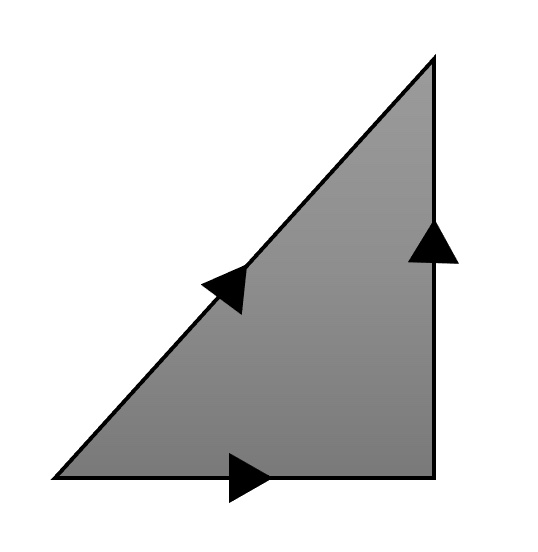}
        \caption{$K$}
    \end{subfigure}%
    \begin{subfigure}[t]{0.33\textwidth}
    \centering
        \includegraphics[height=2.5cm]{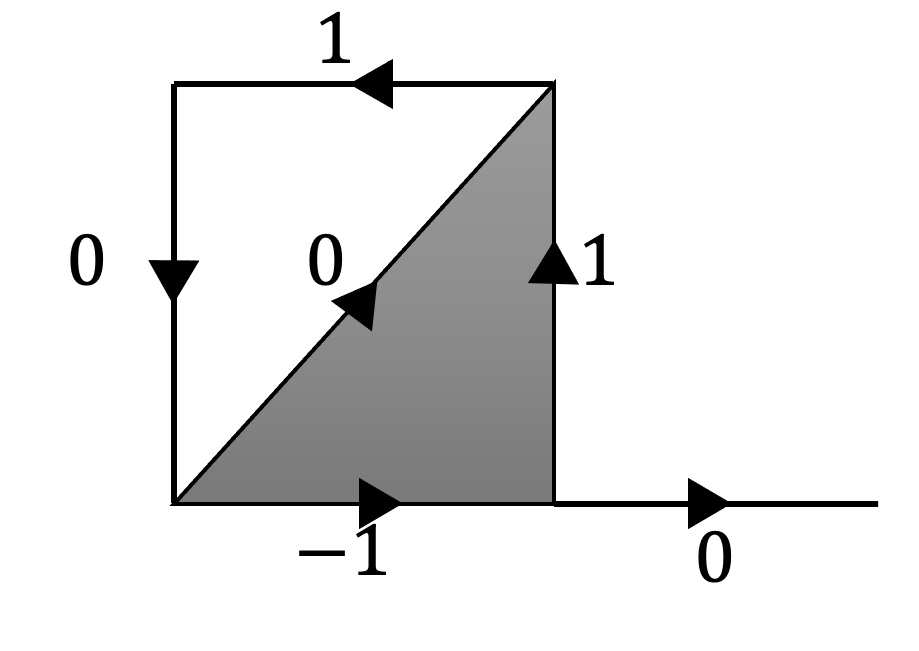}
        \caption{$L$ and $\alpha \in \rmC^1(L)$}
    \end{subfigure}%
    \begin{subfigure}[t]{0.33\textwidth}
        \centering
        \includegraphics[height=2.5cm]{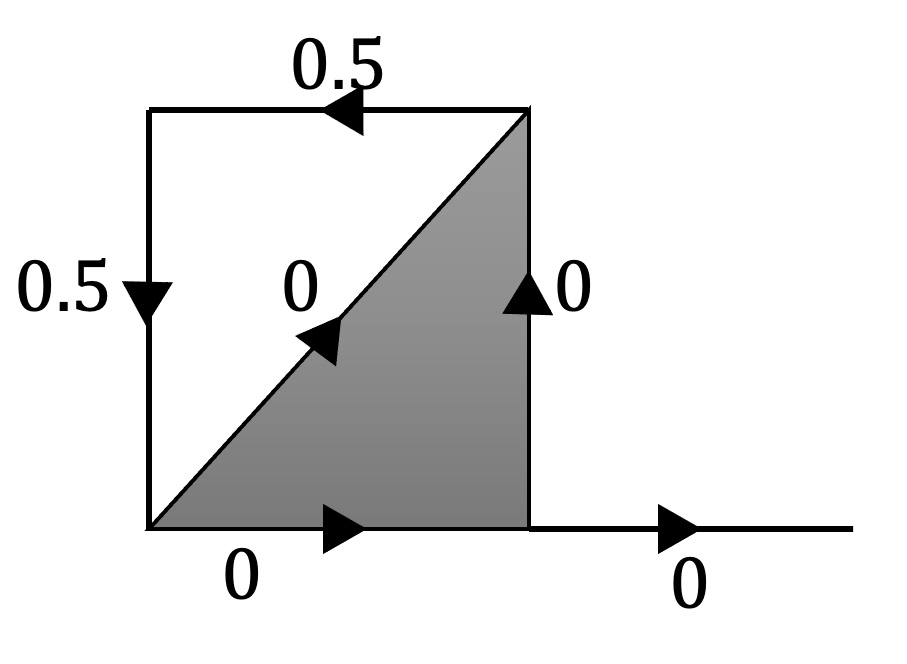}
        \caption{birth cochain}
    \end{subfigure}%
    \caption{The birth cochain for the cohomology class $\alpha$ born between $K$ and $L$, where the edge values show the coefficients in each $1$-chain.}
    \label{fig:exampleb}
\end{figure}

\begin{figure}[h]    
\centering
    \begin{subfigure}[t]{0.23\textwidth}
        \centering
        \includegraphics[height=2.5cm]{example/complex2.jpg}
        \caption{$K$ and $\beta \in \rmC^1 (K)$}
    \end{subfigure}%
    \begin{subfigure}[t]{0.26\textwidth}
        \centering
        \includegraphics[height=2.5cm]{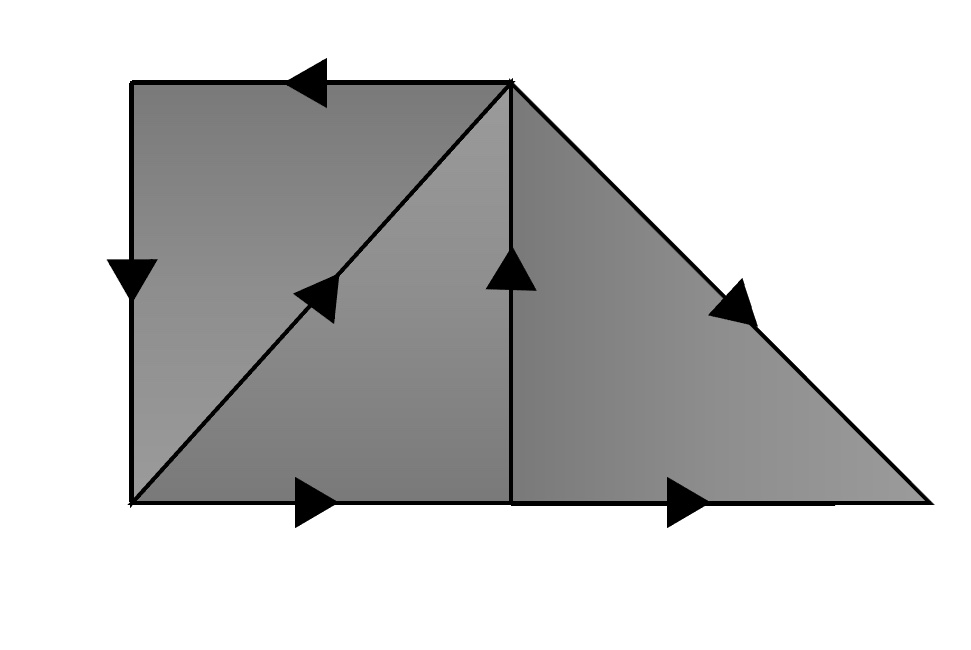}
        \caption{$L$}
    \end{subfigure}
    \begin{subfigure}[t]{0.42\textwidth}
        \centering
        \includegraphics[height=2.5cm]{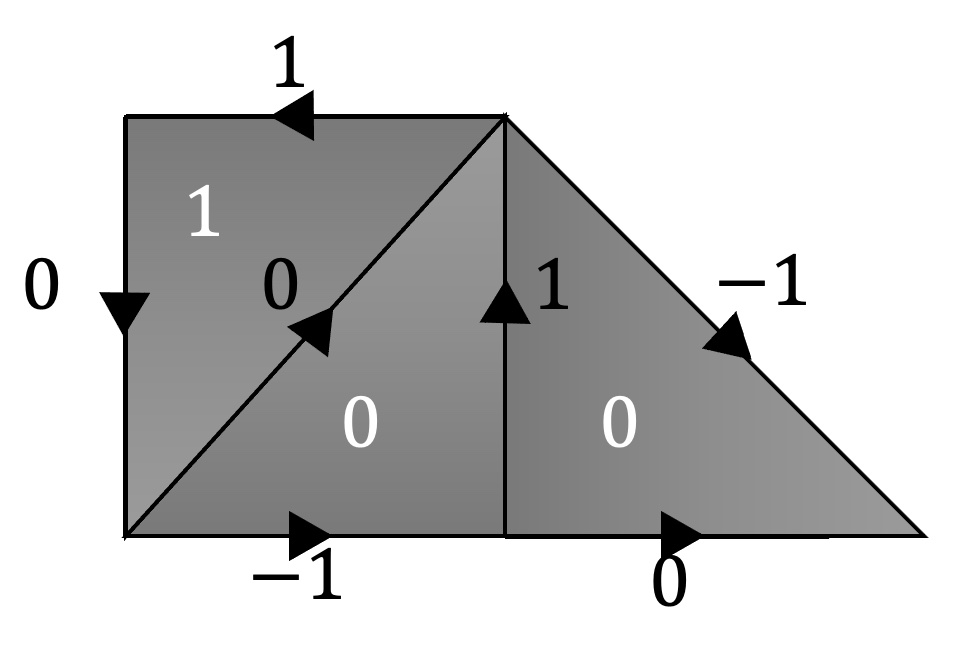}
        \caption{death cochain}
    \end{subfigure}
    \caption{The death cochain for the cohomology class $\beta$ which dies between $K$ and $L$. We show the coefficients of the death cochain on each $2$-simplex in white and the coefficients of the death potential on each edge in black. }
    \label{fig:exampled}
\end{figure}

\begin{remark}[Homogeneity of birth and death cochains]\label{rem:invariance}
Since persistent cohomology classes are typically only defined up to scalar multiplication, it is worthwhile observing the effects of scalar multiplication on birth and death cochains. In each of the optimization problems, multiplying the persistent cohomology class by a scalar multiplies the feasible regions by the same scalar. Since the objective function in each case is absolutely homogeneous, it follows that the birth and death cochains are also scaled by the same amount.
\end{remark}

\subsection{Birth and death cochains in persistence modules}
\label{sec:birth-death cochain}

We now place the preceding construction in the setting of persistent cohomology.  
Let $\{X_t\}_{t \in \mathbb{R}}$ be a filtration of simplicial complexes and fix a bar $[b, d)$ in degree $k$. Associated to the interval module for this bar is a family of non-zero cohomology classes defined up to scalar multiplication
\[
\{ c_t \in \rmH^k(X_t) \}_{t \in [b,d)}
\]
such that whenever $b \le s < t < d$, the class $c_t$ maps to $c_s$ under the homomorphism induced by the inclusion $X_s \subseteq X_t$.
We will refer to the collection of the $c_t$ as a \emph{persistent cohomology class} for the bar $[b, d)$. Birth and death events occur at the transition pairs
\[
X_{b-\varepsilon} \subseteq X_{b+\varepsilon}
\quad\text{and}\quad
X_{d-\varepsilon} \subseteq X_{d+\varepsilon},
\]
for sufficiently small $\varepsilon > 0$. Applying the definitions of \Cref{def:birth-cochain,def:death cochain} to these pairs yields the following.

\begin{definition}[$\varepsilon$-birth and death cochains]
\label{def:epsiloncochain}
    Let $c$ be the persistent cohomology class for a bar $[b, d)$ and let $0 < \varepsilon < d-b$.  
    The \emph{$\varepsilon$-birth cochain for $c$} is the birth cochain of $c_{b+\varepsilon}$ from $X_{b-\varepsilon}$ to $X_{b+\varepsilon}$. 
    The \emph{$\varepsilon$-death cochain for $c$} is the death cochain of $c_{d-\varepsilon}$ from $X_{d-\varepsilon}$ to $X_{d+\varepsilon}$.
\end{definition}

See \Cref{fig:epsilon-timeline} for a schematic visualization of the construction in \Cref{def:epsiloncochain}.
We remark that it is possible to choose a different $\varepsilon$ for each of the four occurrences in \Cref{def:epsiloncochain} above, but in practice it is convenient to have one parameter.

\begin{remark}
    Let $X_{t_0} \subseteq X_{t_1} 
\subseteq \cdots \subseteq X$ be a filtration such that $X_{t_i}$ and $X_{t_{i+1}}$ differ by a single simplex for each $i$. If $\varepsilon < \min_i \{t_{i+1} -t_i\}$, then the $\varepsilon$-birth cochain for a bar $[b,d)$ has support equal to the unique simplex in $X_{b+\varepsilon} \setminus X_b$. Similarly, any $\varepsilon$-death cochain has support equal to single simplex. In other words, for small $\varepsilon$, the birth and death cochains are dual to birth and death simplices respectively. We have thus defined a generalization of birth and death simplices as promised.  
\end{remark}

\section{Interpretations of birth and death cochains}
In this section we reinterpret birth and death cochains from several complementary perspectives. 

\subsection{Birth and death cochains via relative cohomology}

We give an alternative characterization of birth and death cochains using relative cohomology.
Throughout, let $K\subset L$ be simplicial complexes.

Recall from \Cref{sec:relative cohomology} that the relative cochain group $\rmC^k(L,K)$ consists of $k$--cochains on $L$
that vanish on $K$, and that the relative cohomology group $\rmH^k(L,K)$ is defined as
the cohomology of this complex.
The inclusion $\iota:K\hookrightarrow L$ induces the long exact sequence
\cite{hatcher2000}
\begin{equation}\label{eq:les}
\xymatrix{
\cdots \ar[r] &
\rmH^{k}(L,K) \ar[r]^{q^\#} &
\rmH^{k}(L) \ar[r]^{\iota^\#} &
\rmH^{k}(K) \ar[r]^{\lambda^{\#}} &
\rmH^{k+1}(L,K) \ar[r] & \cdots .
}
\end{equation}
Here $q^\#$ is induced by the inclusion of relative cochains into absolute cochains,
$\iota^\#$ is induced by restriction to $K$, and
$\lambda^{\#}$ is the connecting homomorphism, defined by
\begin{equation}\label{eq:connecting homomorphism}
\lambda^{\#}([\beta]_K)
:=
[\delta_L^k \beta']_{(L,K)},
\end{equation}
where $\beta'\in \rmC^k(L)$ is any extension of $\beta$ to $L$.
In particular, one may take $\beta'$ to be the extension of $\beta$ by zero outside $K$.
The connecting homomorphism is well-defined: if $\beta'$ extends $\beta$, then $(\delta_L^k \beta')|_K = \delta_K^k \beta = 0$ since $\beta$ is a cocycle on $K$, so $\delta_L^k \beta' \in \rmC^{k+1}(L,K)$. 

The proposition below on birth cochains is motivated by the following observation. If $[\alpha]_L$ is born between $K$ and $L$, then $\iota^\#([\alpha]_L) = 0 \in \rmH^k(K)$, so by the exactness of (\ref{eq:les}), $[\alpha]_{L}$ is the image of $q^\#$. 

\begin{proposition}\label{prop:birth-cochain}
    The birth cochain for $[\alpha]_L$ from $K$ to $L$ is the unique element of
    $$
    \argmin_{\alpha' \in \rmZ^k(L,K)} \left\{ \|\alpha'\|_2 \;\middle|\; q^\#([\alpha']_{(L,K)}) = [\alpha]_{L} \right\}.
    $$
\end{proposition}

\begin{proof}
Comparing with \Cref{def:birth-cochain}, we see that $\alpha' \in \rmZ^k(L, K)$ is precisely the condition that $\alpha' \vert_K = 0$, and $q^\#([\alpha']_{(L,K)}) = [\alpha]_{L}$ is precisely the condition that $\alpha'$ represents $[\alpha]$ as an element of $\rmH^k (L)$.
\end{proof}

The proposition below on death cochains is motivated by the following observation. Since $[\beta]_{K}$ cannot be extended to a cocycle on $L$, it is not in the image of $\iota^\#$. By the exactness of (\ref{eq:les}), the cohomology class $\lambda^{\#}([\beta]_K)$ is nonzero.

\begin{proposition}\label{prop:death cochain}
    The death cochain for $[\beta]_K$ from $K$ to $L$ is the unique element of
    $$
    \argmin_{\omega \in \rmZ^{k+1}(L, K)} \left\{ \|\omega\|_2 \;\middle|\; [\omega]_{(L,K)} = \lambda^{\#}([\beta]_{K}) \right\}.
    $$
\end{proposition}

\begin{proof}
    It suffices to show that the following equality of sets holds:
    \begin{equation}\label{eq:death cochain}
        \left\{ \omega\;\middle|\;\omega=\delta^k_L\beta' \text{ for }\beta' \in \rmC^k(L), \beta'\vert_{K} = \beta \right\} 
        = 
        \left \{ \omega \in \rmZ^{k+1}(L,K)\;\middle|\;[\omega]_{(L,K)} = \lambda^{\#}([\beta]_{K}) \right\}.
    \end{equation}

    Suppose $\omega=\delta_L^k\beta'$ for some $\beta'\in \rmC^k(L)$ with $\beta'|_K=\beta$. Since $\omega$ is a coboundary in $L$, it is a cocycle in $L$.
    Moreover, since the restriction of $\beta'$ to $K$ equals $\beta$, and $\beta$ is a cocycle on $K$, it follows that
    $\omega\vert_{K} = \delta^k_K(\beta'\vert_{K}) = \delta^k_K\beta = 0$. Thus $\omega$ vanishes on $K$, so $\omega\in \rmZ^{k+1}(L,K)$. 
    Since $\beta'$ extends $\beta$, the definition of the connecting homomorphism gives
    \[
    [\omega]_{(L,K)}
    = [\delta_L^k\beta']_{(L,K)}
    = \lambda^{\#}([\beta]_K).
    \]
    Hence the inclusion `$\subseteq$' in \eqref{eq:death cochain} holds.

    Conversely, suppose $\omega \in \rmZ^{k+1}(L,K)$ satisfies 
    $[\omega]_{(L,K)} = \lambda^{\#}([\beta]_K)$. 
    Let $\beta_1 \in \rmC^k(L)$ be any extension of $\beta$, so 
    $\beta_1|_K=\beta$. By definition of the connecting homomorphism,
    \(
    \omega - \delta_L^k \beta_1 \in \rmB^{k+1}(L,K).
    \)
    Thus there exists $\beta_2 \in \rmC^k(L,K)$ such that
    \(
    \omega - \delta_L^k \beta_1 = \delta_L^k \beta_2.
    \)
    Setting $\beta'=\beta_1+\beta_2$, we obtain 
    $\omega=\delta_L^k\beta'$. Moreover, since $\beta_2$ vanishes on $K$,
    \(
    \beta'|_K=\beta_1|_K+\beta_2|_K=\beta+0=\beta.
    \)
    Hence $\omega$ lies in the left-hand set of \eqref{eq:death cochain},
    which proves the inclusion `$\supseteq$'.
\end{proof}

\begin{corollary}\label{cor:welldefined}
    Suppose $[\beta]_K$ dies between $K$ and $L$. Then the death cochain does not depend on the representative for $[\beta]_K$.
\end{corollary}

\subsection{Birth cochains in degree $0$}
\label{subsec:birth-cochain-0}

Birth cochains admit a particularly simple description in degree $0$, 
where cohomology classes have canonical representatives. The situation for death cochains in degree $0$ turns out to be equivalent to so-called Laplace learning on graphs (see \Cref{sec:laplacelearning}). Since degree-$0$ cohomology depends only on the $1$-skeleton of the filtration, it suffices to work with graphs.  Throughout this subsection, let $G=(V,E)$ be a finite graph with vertex set $V$ and edge set $E$, equipped with a function $f:V\to\mathbb{R}$ inducing a sublevel-set filtration. We assume for simplicity that $f$ is injective; if not, then certain non-canonical choices might need to be made to define persistent cohomology classes (say by fixing an ordering of the vertices). 

For any $t\in\mathbb{R}$, write
\(
G_t := f^{-1}((-\infty,t])
\)
for the sublevel set at $t$, and let $V_t$ denote its vertex set. It is easy to see that $0$-cocycles for $G_t$ are functions $V_t \to \mathbb{R}$ which are constant on connected components, and that each cocycle is the unique representative for a cohomology class. For any vertex $v\in V_t$, denote by $G_t^{v} \subset G_t$ the connected component containing $v$, and by $V_t^{v}$ the vertex set of $G_t^{v}$.  

Let $[b,d)$ be a finite bar in the $0$-dimensional persistence barcode for the filtration $\{G_t\}_{t \in \mathbb{R}}$ and let $v$ be the vertex with $f(v)=b$. Let $d^- = \max\{ t \in \mathsf{Im}f \mid t < d\}$. Then a persistent cohomology class for this bar is given by 
$\{[\alpha_t]\}_{t \in [b,d)}$
where 
$$
\alpha_t(u) = \begin{cases}
    1 & u \in {V^v_{d^-} \cap G_t} \\
    0 & \text{ otherwise }
\end{cases}
$$
Indeed, the fact that a birth event takes place at $b$ implies that $\alpha_b$ has support $\{v\}$, and the rest is determined by the fact that $\alpha_t$ is constant on connected components. Since cohomology classes have unique representatives, the $\varepsilon$-birth cochain  is just $\alpha_{b+\varepsilon}$, that is, the indicator function on all vertices which will be connected to $v$ before the death time $d$ and whose filtration values are at most $b+\varepsilon$. 

\subsection{Death cochain via Laplace learning}\label{sec:laplacelearning}

Laplace learning is a classical method in semi-supervised learning that constructs a function on a graph by minimizing Dirichlet energy subject to fixed boundary values \cite{zhu2003semi}.
Let $L$ be a finite simplicial complex of dimension $1$, i.e., a graph, and $K$ a subcomplex of $L$.
Let $\Delta^{0,\mathrm{up}}_L := (\delta_L^0)^{*}\delta_L^0$ denote the (combinatorial) graph Laplacian.
Given vertex labels $g: K^0 \to \mathbb{R}$, Laplace learning solves
\begin{equation*}
    \argmin_{f \in \rmC^0(L)} \left\{ \|\delta^0_L f\|_2^2 \;\middle|\; \, f|_K = g \right\}.
\end{equation*}
Equivalently, the solutions $f$ are characterized by the following conditions \cite[Section~2]{zhu2003semi}:
\begin{equation}\label{eq:dirichlet}
f|_K=g 
\qquad\text{and}\qquad
(\Delta^{0,\mathrm{up}}_L f)\big|_{L\setminus K}=0.
\end{equation}

Let $K \subseteq L$ be simplicial complexes and $[\beta]_K \in \rmH^k(K)$ die between $K$ and $L$.
Recall from \Cref{def:death cochain} that the death cochain of $[\beta]_K$ is defined as $\delta^k_L \hat{\beta}$, where $\hat{\beta}$ is a death potential that solves the optimization problem
\begin{equation*}
    \argmin_{\beta' \in \rmC^{k}(L)} \left\{ \|\delta^k_{L} \beta'\|_2 \;\middle|\; \beta'\vert_K = \beta \right\}.
\end{equation*}
Thus, death cochains may be viewed as higher-order generalizations of Laplace learning. In the graph setting, one prescribes arbitrary values on $K$, whereas here the prescribed cochain $\beta$ is required to be a cocycle, since it represents a cohomology class that becomes trivial in $L$. One could just as easily formulate the same variational problem (\Cref{eq:deat-cochain-def}) for an arbitrary cochain $\beta$ on $K$; however, without the cocycle condition some of the topological interpretations (such as the relative cohomology viewpoint) would be lost.

By analogy with the graph case (see \Cref{eq:dirichlet}), the death cochain problem can be viewed as a higher–order Dirichlet problem. In particular, the minimizer satisfies a relative harmonicity condition analogous to Laplace learning. We state this precisely in \Cref{prop:death-cochain-laplace-learning} and include a brief proof for completeness.
Let
\[
\Delta^{k,\mathrm{up}}_{L} := (\delta_L^{k})^* \delta_L^k : \rmC^k(L) \to \rmC^k(L)
\]
denote the $k$-th up Laplacian on $L$.

\begin{proposition}\label{prop:death-cochain-laplace-learning}
Let $[\beta]_K\in \rmH^k(K)$ die between $K$ and $L$. 
The death potentials $\hat{\beta}$ are characterized by the following conditions:
\[
\hat{\beta}|_K=\beta 
\qquad\text{and}\qquad
(\Delta^{k,\mathrm{up}}_L \hat{\beta})\big|_{L\setminus K}=0.
\]
\end{proposition}

\begin{proof}
Let $\epsilon : \rmC^k(K) \to \rmC^k(L)$ denote the extension-by-zero map, and write 
$\hat{\beta}=\epsilon(\beta)+\gamma$ with $\gamma\in \rmC^k(L,K)$, with respect to the decomposition in \eqref{eq:orth-decomp}. Then the objective function in the death cochain problem becomes a quadratic polynomial in $\gamma$. 
At a minimum, its gradient with respect to $\gamma$ must vanish in all
admissible directions $\eta\in \rmC^k(L,K)$, i.e.
\[
\frac{d}{dt}\Big|_{t=0}\|\delta_L^k(\hat{\beta}+t\eta)\|_2^2
=\frac{d}{dt}\Big|_{t=0}\left(\|\delta_L^k\hat{\beta}\|_2^2+ 2t\langle \Delta_L^{k,\mathrm{up}}\hat{\beta},\eta\rangle + t^2 \|\delta_L^k \eta\|_2^2\right)
=2\langle \Delta_L^{k,\mathrm{up}}\hat{\beta},\eta\rangle=0.
\]
Thus $\Delta_L^{k,\mathrm{up}}\hat{\beta}$ is orthogonal to $\rmC^k(L,K)$,
which is equivalent to
\(
(\Delta_L^{k,\mathrm{up}}\hat{\beta})|_{L\setminus K}=0.
\)
\end{proof}

\subsection{Death cochain via persistent Laplacian}
\label{sec:p laplacian}
The persistent Laplacian extends the combinatorial Laplacian to pairs of simplicial complexes
\cite{lieutier2014persistentharmonic,wang2020persistent,memoli2022persistent}. 
In the chain-complex (i.e.~homology) setting, a characterization of the persistent Laplacian in terms of Schur complements was established in \cite{memoli2022persistent}. 
The corresponding formulation in the cochain setting was developed in \cite{gulen2023generalization}. 
We briefly recall the relevant constructions.

Let $\Delta:V\to V$ be a self-adjoint positive semidefinite operator on a
finite-dimensional inner product space, and let
$V = W \oplus W^\perp$ be an orthogonal decomposition.
With respect to this decomposition (and any choice of bases adapted to it),
the operator $\Delta$ admits a block matrix representation
\begin{equation} \label{eq:block matrix}
    \Delta =
    \begin{pmatrix}
    \Delta_{WW} & \Delta_{W\perp} \\
    \Delta_{\perp W} & \Delta_{\perp\perp}
    \end{pmatrix}.
\end{equation}
The \emph{(generalized) Schur complement} of $\Delta_{\perp\perp}$ in $\Delta$ is defined as 
\(
\Delta_{WW} - \Delta_{W \perp} \Delta_{\perp\perp}^+ \Delta_{\perp W},
\)
where $(\cdot)^+$ denotes the Moore–Penrose pseudoinverse.
As noted in \cite[Proposition 9]{gulen2023generalization}, viewed as a linear operator, the Schur complement is independent of the choice of bases and therefore is canonically determined by the pair $(\Delta,W)$. This operator is called the \emph{Schur restriction} of $\Delta$ onto $W$ and denoted as \[\Sch(\Delta,W):W \to W.\]

As before, let $\epsilon: \rmC^k(K) \hookrightarrow \rmC^k(L)$ be the extension by zero map.
Define the subspace
\[
    \rmC^{k+1}_{L,K} := \left\{ \xi \in \rmC^{k+1}(L) \;\middle|\; (\delta^k_L)^*(\xi) \in \epsilon(\ker((\delta^{k-1}_{K})^*)) \right\} \subset \rmC^{k+1}(L),
\]
and let $\partial^{L,K}_{k+1}$ denote the restriction of $(\delta^{k}_{L})^*$ to $\rmC^{k+1}_{L,K}$.
The \emph{degree-$k$ persistent up Laplacian} is 
\[
\Delta^{k,\mathrm{up}}_{L,K} := \partial^{L,K}_{k+1}\circ (\partial^{L,K}_{k+1})^* \colon\rmC^{k}(L, K)^\perp \to\rmC^{k}(L, K)^\perp = \epsilon(\rmC^k(K)).
\]

\begin{remark}\label{rmk:schur and laplacian}
    Let 
    \(
    W := \epsilon\big(\ker((\delta^{k-1}_{K})^*)\big)
    \subseteq \epsilon(\rmC^k(K))
    \).
    It follows from \cite[Proposition~10]{gulen2023generalization} that the Schur restriction of $\delta_L^k$ onto $W$ and the degree-\(k\) persistent up Laplacian satisfies 
    \[
    \Delta^{k,\mathrm{up}}_{L,K}
    \;=\;
    \iota_W\;\circ\;
    \Sch(\delta_L^k, W)
    \;\circ\;
    \proj_W,
    \]
    where \(\proj_W:\epsilon(\rmC^k(K))\to W\) is the orthogonal projection and
    \(\iota_W:W\hookrightarrow \epsilon(\rmC^k(K))\) is the inclusion.
    We give a detailed explanation of the above equality in
    \Cref{app:psistentLaplacianproofs}.
\end{remark}

\begin{theorem} 
\label{thm:relation to p laplacian}
Let $K\subset L$ be simplicial complexes and let $[\beta]_K\in \rmH^k(K)$ die between $K$ and $L$.
Then, the squared norm of the death cochain $\delta_L^k \hat{\beta}$ satisfies
\begin{equation}\label{eq:death-schur}
    \|\delta_L^k \hat{\beta}\|_2^2
    =
    \left\langle
    \Sch\left(\Delta_L^{k,\mathrm{up}},\epsilon(\rmC^k(K))\right)\epsilon(\beta),\,
    \epsilon(\beta)
    \right\rangle.
\end{equation}
When $\beta$ is harmonic on $K$, we have
\begin{equation}\label{eq:death-cochain-norm}
    \|\delta_L^k \hat{\beta}\|_2^2
    =\left\langle \Delta^{k,\mathrm{up}}_{L,K} (\epsilon(\beta)), \epsilon (\beta)\right\rangle.
\end{equation}
\end{theorem}

\begin{proof}
    Let $U:= \epsilon(\rmC^k(K))$ and $U^\perp$ be its orthogonal complement in $\rmC^k(L)$. 
    Then, any death potential $\hat{\beta}$ can be written as $\hat{\beta} = \epsilon(\beta) + \gamma$ with $\gamma \in U^\perp$.
    With respect to the decomposition $\rmC^k(L) = U \oplus U^\perp$ (and any choice of bases adapted to it), write the block matrix representation of $\Delta_L^{k,\mathrm{up}}$ as in \cref{eq:block matrix}.
    Then the objective function of the death cochain problem is the quadratic form of the positive semidefinite operator \(
    \Delta_L^{k,\mathrm{up}} = (\delta_L^k)^*\delta_L^k,
    \) satisfying
    \begin{align}
        \|\delta_L^k \hat{\beta}\|_2^2
        = &
        \langle \Delta_L^{k,\mathrm{up}}(\epsilon(\beta)+\gamma),
        \epsilon(\beta)+\gamma\rangle \notag \\
    = &\langle \Delta_{UU}\epsilon(\beta),\epsilon(\beta)\rangle
    + 2\langle \Delta_{U\perp}\gamma,\epsilon(\beta)\rangle
    + \langle \Delta_{\perp\perp}\gamma,\gamma\rangle.  \label{eq:explicit death potential}
    \end{align}
    We now invoke a standard quadratic minimization result
    \cite[Example~3.15]{boyd2004convex}:  
    if
    \(
    f(x,y)=x^T A x + 2x^T B y + y^T C y
    \)
    is positive semidefinite with $A$ and $C$ symmetric,
    then
    \(
    \inf_y f(x,y)
    =
    x^T(A-BC^\dagger \rmB^T)x.
    \)
    Applying this with
    \(x=\epsilon(\beta)\), \(y=\gamma\), \(A=\Delta_{UU}\), \(B=\Delta_{U\perp}\), \(C=\Delta_{\perp\perp},\) and the definition of Schur restriction, we obtain \cref{eq:death-schur}.

    If $\beta$ is harmonic on $K$, then \(
    \beta \in \ker(\delta_K^k)\cap \ker((\delta_K^{k-1})^*),
    \)
    implying that $\epsilon(\beta) \in W := \epsilon\big(\ker((\delta^{k-1}_K)^*)\big)
    $.
    Via a similar argument as above, by replacing $U$ with $W$, we obtain
    \[
        \|\delta_L^k \hat{\beta}\|_2^2
    =
    \left\langle\Sch\left(\Delta_L^{k,\mathrm{up}},W\right)\epsilon(\beta),
    \epsilon(\beta)
    \right\rangle.
    \]
    Finally, using the relationship between the Schur restriction and the persistent Laplacian from \Cref{rmk:schur and laplacian}, along with $\epsilon(\beta) \in W$ and $\iota_{W} = (\proj_{W})^*$, we obtain
    \begin{align*}
        \left\langle
        \Sch\big(\Delta^{k,\mathrm{up}}_{L}, W\big)
        \epsilon(\beta),
        \epsilon(\beta)
        \right\rangle
        &= \left\langle \Delta^{k,\mathrm{up}}_{L,K} (\epsilon(\beta)), \epsilon (\beta)\right\rangle.
    \end{align*}
    This completes the proof.
\end{proof}  

\begin{remark}\label{rmk:death-cochain-explicit}
An explicit expression for the death cochain follows from \eqref{eq:explicit death potential}. 
Writing $\hat{\beta} = \epsilon(\beta) + \gamma$ with $\gamma \in U^\perp$, an optimal choice of $\gamma$ is
\[
\hat \gamma := -\Delta_{\perp\perp}^+\Delta_{\perp U}\epsilon(\beta)
\quad \text{(up to } \ker(\Delta_{\perp\perp})\text{)},
\]
which yields a death potential $\hat{\beta}$. While $\hat{\beta}$ is not unique, the resulting death cochain is uniquely determined and given by
\[
\delta_L^k \hat{\beta}
=
\delta_L^k \bigl(I - \Delta_{\perp\perp}^+\Delta_{\perp U}\bigr)\epsilon(\beta).
\]
\end{remark}

    At present, we are not aware of any direct analogues of the results in this section and \Cref{sec:laplacelearning} for birth cochains. This asymmetry is not surprising in the light of the assymetry between \Cref{prop:birth-cochain} and \Cref{prop:death cochain}, namely that the birth cochain is the $\ell_2$-minimal representative in the \emph{preimage} of $[\alpha]_L$ under the map $q^\#$, whereas the death cochain is the $\ell_2$-minimal representative of the relative class obtained as the \emph{image} of $[\beta]_K$ under the connecting homomorphism $\lambda^{\#}$.

\section{Generalizing birth and death time}\label{sec:content}

\subsection{Birth and death content}

Birth and death cochains naturally arise in problems where one seeks to increase or decrease the persistence of selected homological features by modifying the underlying data. 
Motivated by this perspective, we introduce relaxations of birth and death times.

\begin{definition}\label{def:content}
Let $c$ be a persistent cohomology class with associated bar $[b, d)$ and let $\varepsilon > 0$.
    Let $\birthcochain $ and $\omega$ denote the $\varepsilon$-birth and $\varepsilon$-death cochains for $c$ respectively. The \textbf{$\varepsilon$-birth content}, $B_\varepsilon(c)$ is defined as
    \begin{equation}\label{eq:birthcontent}
       B_\varepsilon(c) := \sum_{\sigma \in X^k_{b+\varepsilon}} f(\sigma)\frac{|\birthcochain (\sigma)|}{||\birthcochain ||_1} 
    \end{equation}
    where $X^k_{b}$ is the $k$-skeleton of $X_b$. The \textbf{$\varepsilon$-death content} denoted by $D_\varepsilon(c)$ is defined as
    \begin{equation}\label{eq:deathcontent}
    D_\varepsilon(c) := \sum_{\sigma \in X^{k+1}_{d+\varepsilon}} f(\sigma)\frac{|\omega(\sigma)|}{||\omega||_1}
    \end{equation}
\end{definition}

Note that as long as $\varepsilon < (d-b)/2$, we always have $B_\varepsilon(c) \leq D_\varepsilon(c)$ since the filtration values in \Cref{eq:birthcontent} are all less than the filtration values in \Cref{eq:deathcontent} and the coefficients sum to $1$. In practice, we often use $\varepsilon = \varepsilon_0(d-b)$ where $0 < \varepsilon_0 < 1/2$ is a parameter chosen by the user, so that $\varepsilon$ depends on the bar length. Notice also that by \Cref{rem:invariance}, $B_\varepsilon(c)$ and $D_\varepsilon(c)$ are invariant under multiplication of $c$ by scalars, which is helpful since the persistent cohomology class for a bar is often defined only up to scalar multiplication. 

\newsiamremark{example}{Example}

\begin{example}
    Recall from \Cref{subsec:birth-cochain-0} that in degree $0$, the
    $\varepsilon$-birth cochain for a bar $[b, d)$ is the indicator function
    of all vertices which will eventually merge with the birth vertex $v$ before it dies. Therefore the $\varepsilon$-birth content in the notation of \Cref{subsec:birth-cochain-0} is
    $$
    B_\varepsilon
    =
    \sum_{u \in V^v_{d^-} \cap G_{b+\varepsilon}}
    f(u)\frac{|\birthcochain (u)|}{\|\birthcochain \|_1}
    =
    \frac{1}{|V^v_{d^-} \cap G_{b+\varepsilon}|}
    \sum_{u \in V^v_{d^-} \cap G_{b+\varepsilon}} f(u).
    $$
    or the average filtration value of vertices present at $b+\varepsilon$ which will merge with $v$ before the death time $d$.
\end{example}

When the filtration in question is a Vietoris-Rips complex, it can be inconvenient to use the filtration value of simplices of dimension higher than $1$. The filtration value of a higher dimensional simplex $\sigma$ is by definition the maximum filtration value of its edges. But the specific edge that achieves this maximum value is not stable; indeed, if all edges have almost the same filtration value, then a small perturbation can easily change the maximum edge. This leads to instability when trying to optimize. For this reason, we want to consider a further relaxation of our birth and death content.

\begin{definition}
    The \emph{edge-relaxed $\varepsilon$-birth content}, denoted $\tilde{B}_\varepsilon(c)$, is defined as in~\Cref{eq:birthcontent} but with $f(\sigma)$ replaced by 
    $$
    \tilde{f}(\sigma) = \mathsf{mean} \left\{f(\tau) \;\middle|\; \tau \subseteq \sigma,\ \tau \in X^1_{b+\varepsilon} \setminus X^1_{b-\varepsilon} \right\} 
    $$
    Similarly, the \emph{edge-relaxed $\varepsilon$-death content}, denoted $\tilde{D}_\varepsilon(c)$, is defined as in~\Cref{eq:deathcontent} but with $f(\sigma)$ replaced by
    $$
    \tilde{f}(\sigma) = \mathsf{mean} \left\{f(\tau) \;\middle|\; \tau \subseteq \sigma,\ \tau \in X^1_{d+\varepsilon} \setminus X^1_{d-\varepsilon} \right\} 
    $$
\end{definition}

We define the \emph{(edge-relaxed) $\varepsilon$-persistence content} of $c$ to be the (edge-relaxed) $\varepsilon$-death content minus the (edge-relaxed) $\varepsilon$-birth content.

\begin{proposition}\label{prop:epsbounds}
    For any persistent cohomology class $c$ with associated bar $[b,d)$, we have
    $$
    D_\varepsilon(c) - B_\varepsilon(c) - \varepsilon \leq d-b \leq D_\varepsilon(c) - B_\varepsilon(c) + \varepsilon
    $$
    and the same for the edge-relaxed versions. In particular, as $\varepsilon \to 0$, persistence content converges to persistence.
\end{proposition}
\begin{proof}
    The coefficients in \Cref{eq:birthcontent} sum to $1$ by design, and the filtration values for simplices with non-zero coefficients are within $\varepsilon$ of $b$ and $d$ respectively.
\end{proof}

The choice of $\varepsilon$ has a significant effect on birth and death cochains, and thus also on birth and death content. As we will see in \Cref{sec:pointclouds}, there are trade-offs between small and large $\varepsilon$ values. For this reason, we propose \emph{multi-cochain} versions of birth and death content by defining, for a finite subset $\mathcal{E} \subseteq (0,\infty)$,
\begin{equation}\label{eq:multicochains}
    \mathcal{B}_{\mathcal{E}}(X) = \frac{1}{|\mathcal{E}|} \sum_{\varepsilon \in \mathcal{E}} \mathcal{B}_{\varepsilon}(X), \quad \mathcal{D}_{\mathcal{E}}(X) = \frac{1}{|\mathcal{E}|} \sum_{\varepsilon \in \mathcal{E}} \mathcal{D}_{\varepsilon}(X)
\end{equation}
together with their edge-relaxed versions. Note that since we are taking an average, \Cref{prop:epsbounds} still applies.

So far we have only defined the $\varepsilon$-birth and $\varepsilon$-death content for a given bar $[b, d)$. When doing topological optimization on real data, it is common to define a way to choose a bar(s) at each step of the optimization process, aligned with the overall objective. For example, one could focus on the longest bar to promote a feature, or the shortest bar to suppress noise. In each of our experiments below, we specify in advance how to choose the bar to optimize.

\begin{remark}[Gradients for birth and death content]
    Looking at \Cref{eq:birthcontent} and \Cref{eq:deathcontent}, we note that if the birth and death cochains ($\birthcochain $ and $\omega$ respectively) are treated as fixed, then birth and death content is a linear combination of filtration values. We can therefore easily determine the gradient with respect to a filtration value $f(\sigma)$. Typically, these filtration values are themselves differentiable functions of some other features (e.g.~distances between points), and so we can differentiate birth and death content with respect to those features for fixed birth and death cochains. Using these derivatives allows us to do gradient descent in a variety of settings, some of which we demonstrate in later sections. For birth and death content to be truly differentiable with respect to feature values, we require that sufficiently small changes in feature values do not change the simplicial complex at crucial thresholds, namely: $b-\varepsilon$, $b+\varepsilon$, $d-\varepsilon$ and $d+\varepsilon$. We expect this to be the case in many practical settings. For example, for Vietoris-Rips on point clouds, we need that all non-zero distances are distinct and that none of them are precisely equal to $b-\varepsilon$, $b+\varepsilon$, $d-\varepsilon$ or $d+\varepsilon$. 
\end{remark}

\subsection{Critical points of birth and death content for point clouds}

To illustrate the theoretical advantages of birth and death cochains, we will find a natural class of critical points for these functions on point clouds. Our focus is on the $\varepsilon$-birth and $\varepsilon$-death content of the longest bar in Vietoris-Rips degree-$1$ persistent cohomology, for which the critical points will be regular polygons. The main idea behind the results in this section is that in a regular polygon points are geometrically indistinguishable, from which it should follow that the gradient of birth and death content is highly symmetric and can be controlled by a penalty function with similar symmetries. \Cref{app:criticalproofs} contains some technical lemmas about degree-$1$ persistent cohomology for regular polygons, which we will need for the main theorem in this section. 

Note that scaling a point cloud $X$ also scales its persistence content. Thus, in order to state a stationarity result, we need to apply either some kind of regularization term or constraint. \Cref{thm:criticalmoment} covers the first case, and we make a remark on the second thereafter. Given an $\varepsilon$ value used to calculate the birth and death content for a bar $[b,d)$ associated to a filtration function $f$, we say that $\varepsilon$ is \emph{generic} if $\{b-\varepsilon, b+\varepsilon, d-\varepsilon, d+\varepsilon\} \cap \mathsf{Im}f = \varnothing$. That is, the filtration values involved in calculating the birth and death cochains are not critical values for the topology of the filtered simplicial complex. Note that in the case of a VR-complex, $\mathsf{Im}f$ is precisely the set of distances between points.

\begin{theorem}\label{thm:criticalmoment}

    Fix $\varepsilon_0 > 0$, and write $\mathbb{R}^{n \times 2}$ as the set of (ordered) $n$-tuples of points $(\mathbf{x}_1, \ldots, \mathbf{x}_n)$ in $\mathbb{R}^2$. For any $X \in \mathbb{R}^{n \times 2}$, let $\mathcal{B}_\varepsilon(X)$ and $\tilde{\mathcal{D}}_\varepsilon(X)$ be the $\varepsilon$-birth content and the edge-relaxed $\varepsilon$-death content of the most persistent degree-$1$ bar $[b, d)$ in the VR persistence diagram of $X$ (or zero if no such bar exists), where $\varepsilon := \varepsilon_0(d-b)$. Let $\mathcal{P}: \mathbb{R}^{n \times 2} \to [0, \infty)$ be any differentiable penalty function which is invariant under linear isometries of $\mathbb{R}^2$ and such that for any $X \in \mathbb{R}^{n \times 2}$ which is not all zero,
    $$
    \lim_{t_0 \to \infty} \frac{d}{dt} \mathcal{P}(t\cdot t_0 \cdot X) = \infty, \ \lim_{t_0 \to 0} \frac{d}{dt} \mathcal{P}(t\cdot t_0 \cdot X) = 0
    $$
    Consider the loss function $\mathcal{L}: \mathbb{R}^{n \times 2} \to \mathbb{R}$ given by
    $$
    \mathcal{L}(X) = - \left(\tilde{\mathcal{D}}_\varepsilon(X) -\mathcal{B}_\varepsilon (X) \right) + \mathcal{P}(X)
    $$
    Then there is a regular polygon centered at the origin whose vertices $\hat{X}$ are a critical point of $\mathcal{L}$, assuming that $\varepsilon$ is generic at $\hat{X}$.
\end{theorem}

\begin{proof}
Fix for now a regular polygon $\hat{X} \in \mathbb{R}^{n \times 2}$ centered at the origin. This choice gives an action $\rho$ of the dihedral group $D_n$ (the group of symmetries of a regular polygon) on $\mathbb{R}^{n \times 2}$ by linear isometries. We will need to consider two different actions of $D_n$ on $\mathbb{R}^{n \times 2}$: 
\begin{itemize}
    \item for $g \in D_n$, let $\pi_g$ be the permutation on $\{1,\ldots, n\}$ induced by $g$, and define 
    $$\phi(g)(\mathbf{x}_1, \ldots, \mathbf{x}_n) = (\mathbf{x}_{\pi_g^{-1}(1)}, \ldots, \mathbf{x}_{\pi_g^{-1}(n)})$$
    \item for $g \in D_n$, $\psi(g)$ is the automorphism of $\mathbb{R}^{n \times 2}$ given by applying $\rho(g)$ to each entry.
\end{itemize}
Note that in particular, $\phi(g)(\hat{X}) = \psi(g)(\hat{X})$ for any $g$. For any $g \in D_n$, $\phi(g)$ and $\psi(g)$ both give simplicial maps from the VR complex of $\hat{X}$ to itself and so it makes sense to ask if they preserve certain cochains. 
By \Cref{lem:symmetry}, 
birth and death cochains involved in  $\mathcal{B}_\varepsilon(\hat{X})$ and $\tilde{\mathcal{D}}_\varepsilon(\hat{X})$ are preserved up to sign. We can ask the same question of point clouds sufficiently close to $\hat{X}$ in $\mathbb{R}^{n \times 2}$, but only for $\phi(g)$ and not for $\psi(g)$ since the latter need not fix the point cloud set-wise. When $\varepsilon$ is generic, the simplicial complexes for nearby clouds at $b \pm \varepsilon$ and $d \pm \varepsilon$ are the same as for $\hat{X}$, so birth and death cochains are still preserved up to sign by $\phi(g)$. In particular, the coefficients in \Cref{eq:birthcontent} and \Cref{eq:deathcontent} of simplicies in the same $D_n$ orbit are the same. It follows that $\tilde{\mathcal{D}}_\varepsilon$ and $\mathcal{B}_\varepsilon$ are both invariant under the action $\phi$ for point clouds near $\hat{X}$. Observe that for all point clouds, $\tilde{\mathcal{D}}_\varepsilon$ and $\mathcal{B}_\varepsilon$ are invariant under the action $\psi$ simply because the birth and death content are preserved by isometries of the point cloud.

We have that $\mathcal{B}_\varepsilon$ and $\tilde{\mathcal{D}}_\varepsilon$ are differentiable at $\hat{X}$ because Euclidean distances are differentiable away from $0$, and the other factors are constant near $\hat{X}$. Let $\mathcal{T}(X) = \tilde{\mathcal{D}}_\varepsilon(X) -\mathcal{B}_\varepsilon (X)$, and note that due to $\mathcal{T}$ being invariant under both actions of $D_n$, $\mathcal{T} \circ \psi(g)\circ \phi(g^{-1}) = \mathcal{T}$ for all $g \in D_n$. A standard argument shows that the gradient of $\mathcal{T}$ must satisfy the equivariance
$$
\nabla \mathcal{T}(\psi(g) \phi(g^{-1})( X)) =  \psi(g) \phi(g^{-1}) ( \nabla \mathcal{T}(X))
$$
so that in particular, $\nabla \mathcal{T}(\hat{X}) = \psi(g) \phi(g^{-1}) ( \nabla \mathcal{T}(\hat{X}))$ for all $g \in D_n$.

For each vertex $\mathbf{x}_i$ of our regular polygon, we can pick a $g$ such that $\psi(g)$ reflects each point around the line from the origin through $\mathbf{x}_i$ and $\phi(g)$ fixes the $i^{th}$ entry in any tuple. Then the partial gradient $\mathbf{v}_i = \frac{\partial}{\partial \mathbf{x}_i} \mathcal{T}(\hat{X}) \in \mathbb{R}^2$ must be fixed by $\phi(g)$, and thus points in the direction of $\mathbf{x}_i$ (or in the opposite direction). Notice that $\mathcal{T}(\hat{X})$ strictly increases if we scale $\hat{X}$ by a scalar greater than $1$; indeed, we have already argued that the birth and death cochains remain the same close to $\hat{X}$, and the filtration values only increase. It follows that $\mathbf{v}_i$ must point from the origin towards $\mathbf{x}_i$ since this is a direction of increase in $\mathcal{T}$. Since $\psi$ acts by isometries and $\phi$ permutes the entries of $\hat{X}$ transitively, all the $\mathbf{v}_i$ have the same norm. We have thus characterized the gradient $\nabla \mathcal{T}(\hat{X})$ as $\gamma \hat{X} \in \mathbb{R}^{n \times 2}$ for some $\gamma > 0$.

The penalty term is invariant under all linear isometries, so by a similar argument, the gradient of $\mathcal{P}$ at $\hat{X}$ must have the form $\gamma' \hat{X}$. If we scale $\hat{X}$ by some $\lambda > 0$, $\varepsilon$ scales with $\lambda$, as do the distances between points, so the birth and death cochains are unaffected. The birth and death content scale by $\lambda$, but $\nabla \mathcal{T} (\hat{X})$ remains constant (since the gradients for the Euclidean distances remains unchanged). The condition on $\mathcal{P}$ in the statement of theorem guarantees that the norm of the gradient of $\mathcal{P}$ can be any real number, so that for some scaling $\lambda$, the gradient of $\mathcal{L} = -\mathcal{T} + \mathcal{P}$ must equal zero. When it does, we have found the required regular polygon. 
\end{proof}

We note that one can also replace the penalty term $\mathcal{P}$ by the constraint $
\frac{1}{n} \sum_{i} ||\mathbf{x}_i||^2 = 1 
$. The proof is mostly the same, one only need note that the gradient is normal to the feasible region at $\hat{X}$.

\section{(In)stability}\label{sec:instability}

One of the most prominent advantages of persistent (co)homology is that it is stable in a certain precise sense. Given two functions $f, g$ on a simplicial complex $X$, the bottleneck distance between the barcodes of the sublevel-set filtrations they induce is bounded above by $||f-g||_\infty$~\cite{cohen2005stability}. In particular, this implies that if $||f-g||_\infty < \delta$, then for every bar $[b,d)$ in the barcode for $f$, $[b,d)$ is short (length less than $2\delta$) or there is a bar $[b', d')$ in the barcode for $g$ which is $\delta$-close, i.e.~$\max (|b'-b| , |d'-d|) < \delta$. 
However, there need be no relation between the locations of the corresponding features, which may lie in entirely different parts of the dataset. 
For this reason, any topological optimization method which optimizes a function of the persistence barcode by modifying the underlying data is bound to suffer from instability in the worst case, since small changes in the data can significantly alter the gradient. This is true of optimization with birth and death simplices, and with birth and death cochains.

We nonetheless make a heuristic argument for why birth and death cochains provide more empirical stability in optimization tasks than birth and death simplices, as demonstrated for example in the tail-end convergence in \Cref{fig:tenpointruns} and the multiple trials in \Cref{fig:MNIST0}. Let $f$ be a function on a simplicial complex $X$, and consider a feature born at $b = f(\sigma)$. We now make some simple observations; in each case we only have to note that the simplicial complex does not change at the relevant values: $b$ for simplices, and $b \pm \varepsilon$ for cochains.
\begin{itemize}
    \item If $f$ is changed to $f'$ and for every simplex $\tau$ we have
    $f(\tau) < b \iff f'(\tau) < b$, then the simplex $\sigma$ still gives birth to a persistent cohomology class at time $b$.
    \item If $f$ is changed to $f'$ and for every simplex $\tau$ we have
    $f(\tau) < b-\varepsilon \iff f'(\tau) < b-\varepsilon$
    and
    $f(\tau) < b+\varepsilon \iff f'(\tau) < b+\varepsilon$,
    then a persistent cohomology class is still born between $b-\varepsilon$ and $b+\varepsilon$ and has the same birth cochain.
\end{itemize}

These observations make it clear that the stability of the birth simplex depends on the distance between $b$ and the filtration values, while the stability of the birth cochain depends on the distance betwen $b \pm \varepsilon$ and the filtration values. Similar observations hold for death simplices and death cochains. In many topological optimization tasks, filtration values cluster together over the runtime of the algorithm. For example, point clouds start to resemble regular polygons (see \Cref{thm:criticalmoment}), and images have fewer local maxima and minima. When this happens, the space between $b$ and other filtration values is often decreased, so that the stability of the birth simplex deteriorates. On the other hand, the space between $b \pm \varepsilon$ and the filtration values is often increased, bolstering birth cochain stability. The most extreme example of this is when $f$ is not injective (such as for a true regular polygon), in which case the birth simplex may not even be well-defined, while the birth cochain (for generic $\varepsilon$) is. 

To demonstrate this effect, we track the pairwise distances (which are the filtration values for Vietoris-Rips) and the important values $b$, $b \pm \varepsilon$, $d$ and $d \pm \varepsilon$ during the run in \Cref{fig:smallcochains}, with the results shown in \Cref{fig:heuristic}. Notice how as the values converge to the optimum, they cluster around $b$ and $d$, destabilizing birth and death simplices, but move steadily away from $b \pm \varepsilon$ and $d \pm \varepsilon$, stabilizing the birth and death cochains.

\begin{figure}[h]
    \centering
    \includegraphics[width=0.8\linewidth]{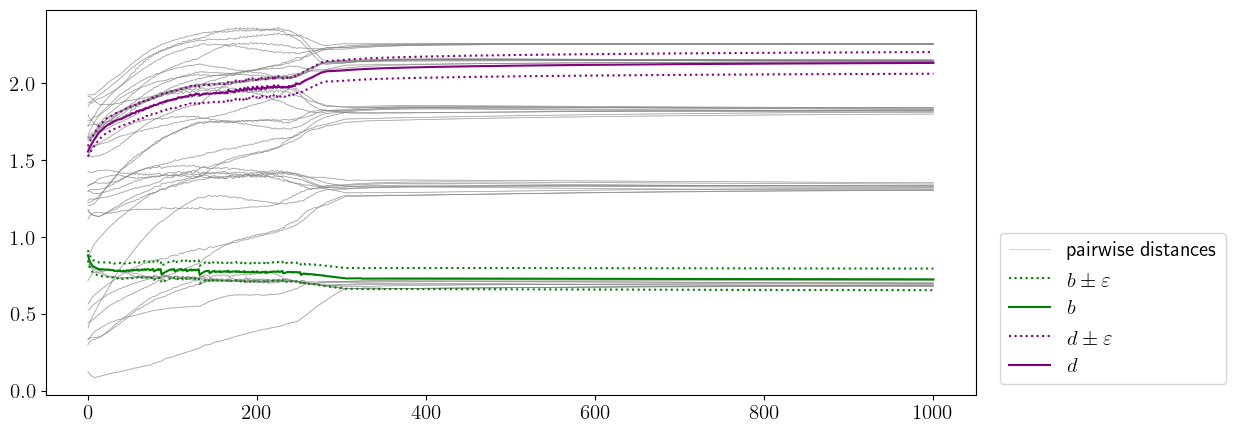}
    \caption{Tracking $b$, $b \pm \varepsilon$, $d$ and $d \pm \varepsilon$ for the run in \Cref{fig:smallcochains}, where $\varepsilon$ is defined as a relative value via $0.05(d-b)$ for this run. }
    \label{fig:heuristic}
\end{figure}

\section{Point cloud optimization}\label{sec:pointclouds}

In this section, we consider the model problem of increasing the persistence of an degree-$1$ bar in the Vietoris-Rips persistent homology of a point cloud. Specifically, for a point cloud $X = (\mathbf{x}_1, \mathbf{x}_2, \ldots, \mathbf{x}_n)$ of points in $\mathbb{R}^2$, let $[b, d)$ be the longest bar in the degree-$1$ persistence barcode for $X$. We will assume throughout that there is a unique longest bar, so that we have a well-defined topological feature to target at each step. To stop the points going off to infinity, we use a penalty term 
$$\mathcal{P}(X) = \sum_i d(\mathbf{x}_i, B)^2$$
where $B$ is the unit ball. Our penalty term is modeled on the one used in \cite{carriere2021optimizing} for point clouds. We will compare two different approaches, one involving optimizing by identifying and adjusting birth and death simplices alone, and one using cochains. In general, we will set $\varepsilon$ to be a fixed fraction of bar length. We now describe our experiments in more detail.

\subsection{Single point cloud experiment}

In our first experiment, we randomly generate a point cloud by sampling ten points from the unit circle and adding Gaussian noise with standard deviation $0.1$. We consider two different methods:
\begin{itemize}
    \item \emph{Cochain method.} Maximize $\mathcal{C}_{\varepsilon_0}(X) - \mathcal{P}(X)$ where $\mathcal{C}_{\varepsilon_0}(X) := \tilde{D}_\varepsilon(\alpha) - \mathcal{B}_\varepsilon(\alpha)$ is the edge-relaxed $\varepsilon$-persistence content of the highest persistence degree-$1$ bar and $\varepsilon = \varepsilon_0(d-b)$. 
    \item \emph{Simplices method.} Maximize $(d - b) - \mathcal{P}(X)$ directly by identifying birth and death simplices at each iteration. 
\end{itemize}
We use gradient ascent for each method with learning rate $\gamma = 0.02$ for $1,000$ iterations. \Cref{fig:smallcochains} shows snapshots of the cochains method for $\varepsilon_0 = 0.05$. We test a range of values $\varepsilon_0 \in \{0.03, 0.04, 0.05, 0.06\}$ for the cochains method. Since the objective functions for each method are different and thus may end up with final point clouds at different scales, we do not want to compare the persistence of the final clouds directly. Instead, we normalize the persistence by computing $(d-b)/||X||_2$ for the point cloud $X$ at each iteration.

\Cref{fig:tenpointruns} shows the results of each method. Additional runs for other learning rates can be found in Supplemental Material; while different learning rates do have an effect, $\gamma = 0.02$ represents the overall trend. We see that the behavior of the cochains method depends strongly on the choice of $\varepsilon_0$. For low $\varepsilon_0$, the method behaves like the simplices method, converging poorly to a irregular final configuration. For intermediate values, the method converges well to a regular $10$-gon with high persistence. Finally, high values of $\varepsilon_0$ result in the cochains method squashing nearby points together to create a (slightly irregular) $9$-gon.  

\begin{figure}
    \centering
    \begin{tikzpicture}
        \node at (-0.5,0) {\includegraphics[width=10cm]{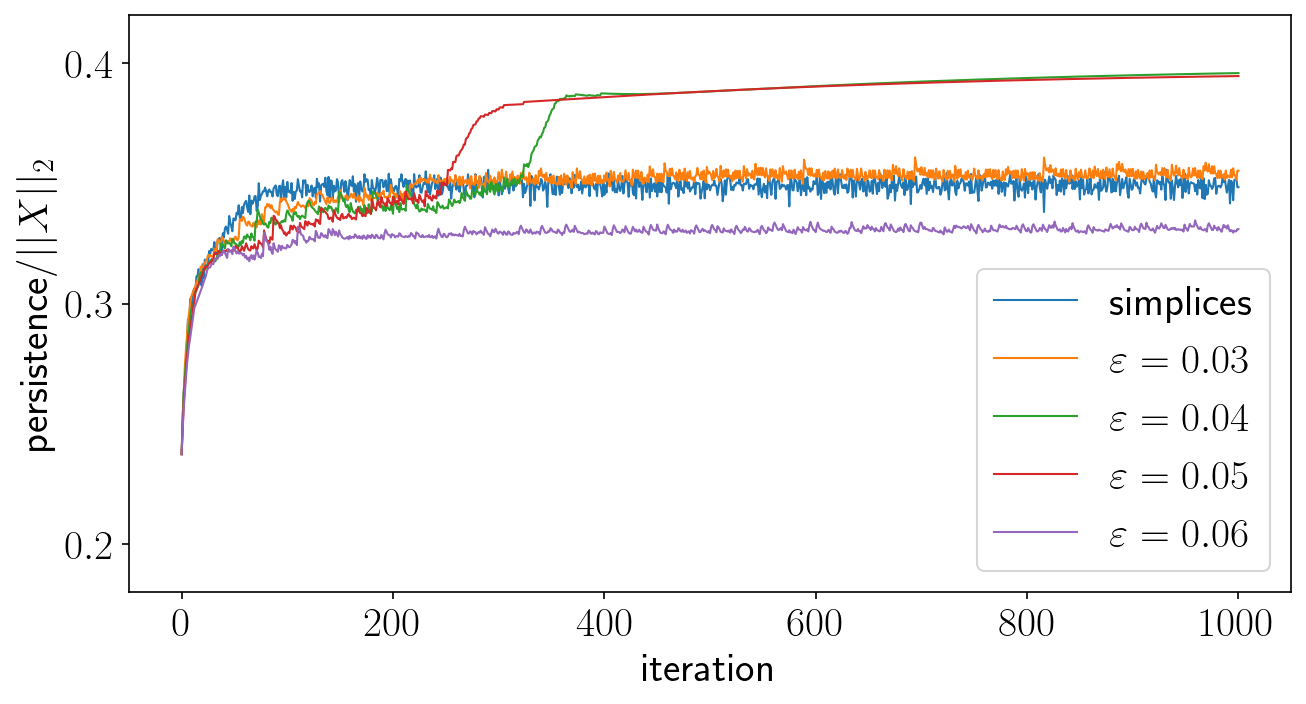}};

        \node at (6,2.5) {\includegraphics[width=2.5cm]{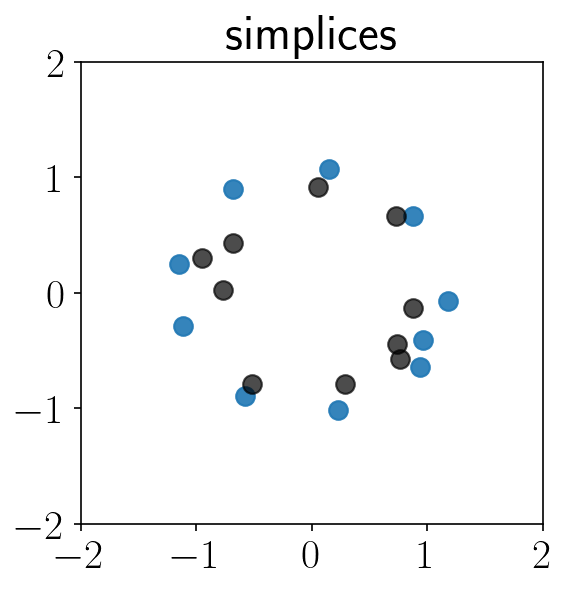}};
        \node at (8.5,2.5) {\includegraphics[width=2.5cm]{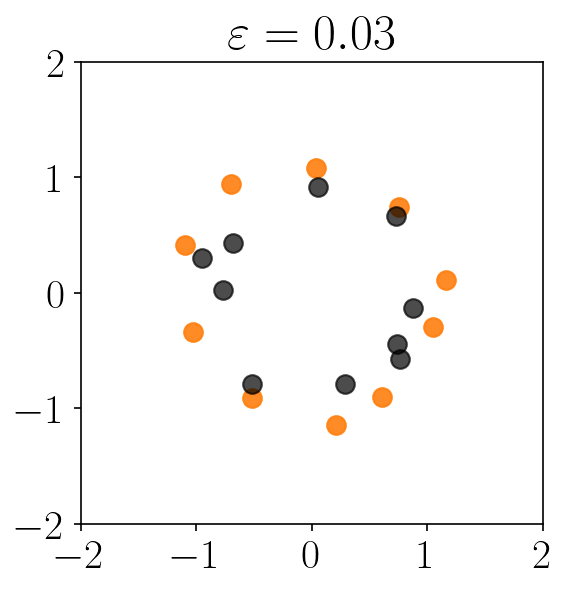}};
        \node at (6,0) {\includegraphics[width=2.5cm]{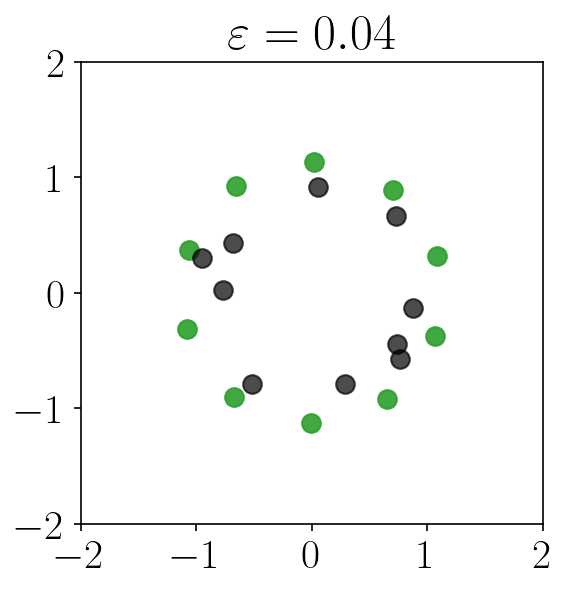}};
        \node at (8.5,0) {\includegraphics[width=2.5cm]{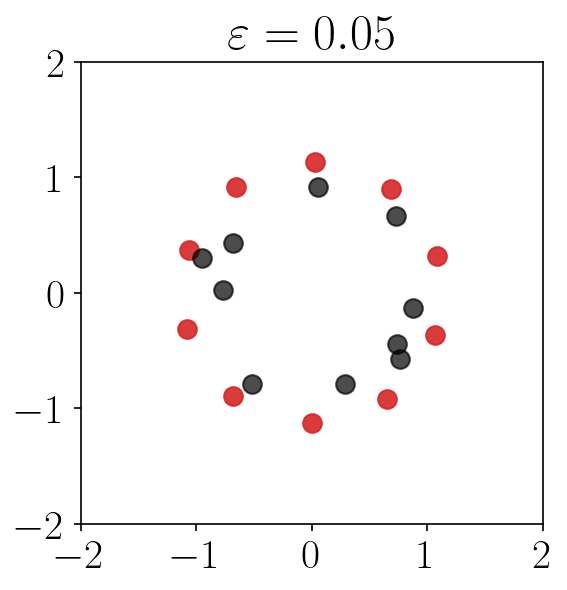}};
        \node at (7.125,-2.5) {\includegraphics[width=2.5cm]{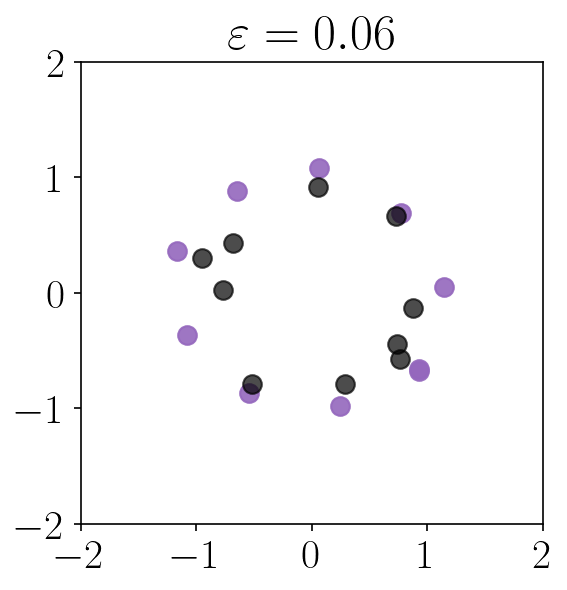}};
        
    \end{tikzpicture}
    \caption{Maximizing an degree-$1$ feature using either birth and death cochains, or birth and death simplices. We show how the normalized persistence changes over gradient ascent iterations (left), and the initial point cloud (in black) vs the final point cloud for each method (right).}
    \label{fig:tenpointruns}
\end{figure}

\subsection{Multi-cochain method}

\Cref{fig:tenpointruns} suggests that there is a trade-off between low and high $\varepsilon_0$ values. A high $\varepsilon_0$ value adjusts multiple simplices at once, possibly avoiding the local minima that the simplices method gets trapped in, but it can squash points together. Low $\varepsilon_0$ inherits the tendency of the simplices method to converge to non-regular configurations, but does not squash as many points. This is an ideal situation for using the multi-cochain version of birth and death content as described in \Cref{eq:multicochains}.

To evaluate the multi-cochains method, we repeat the experiment in the previous section on a set of $110$ randomly generated point clouds with sizes ranging from $10$ points to $20$ points (i.e.~$10$ clouds for each size). We fix the learning rate at $\gamma = 0.02$ throughout, and optimize the same loss functions as before but with birth and death content replaced by their multi-cochain versions (\Cref{eq:multicochains}) with $\mathcal{E} = \{0.01, 0.05, 0.1\}$. 
\Cref{fig:multicochains} shows the results. A direct comparison of normalized persistence shows that the multi-cochains method matches or outperforms the simplices method in almost all cases. The most extreme difference is shown on the right in \Cref{fig:multicochains}, where the multi-cochains finds a regular $10$-gon from a very irregular cloud.

\begin{figure}
    \centering
    \begin{subfigure}[t]{0.4\textwidth}
    \centering
    \includegraphics[height=5cm]{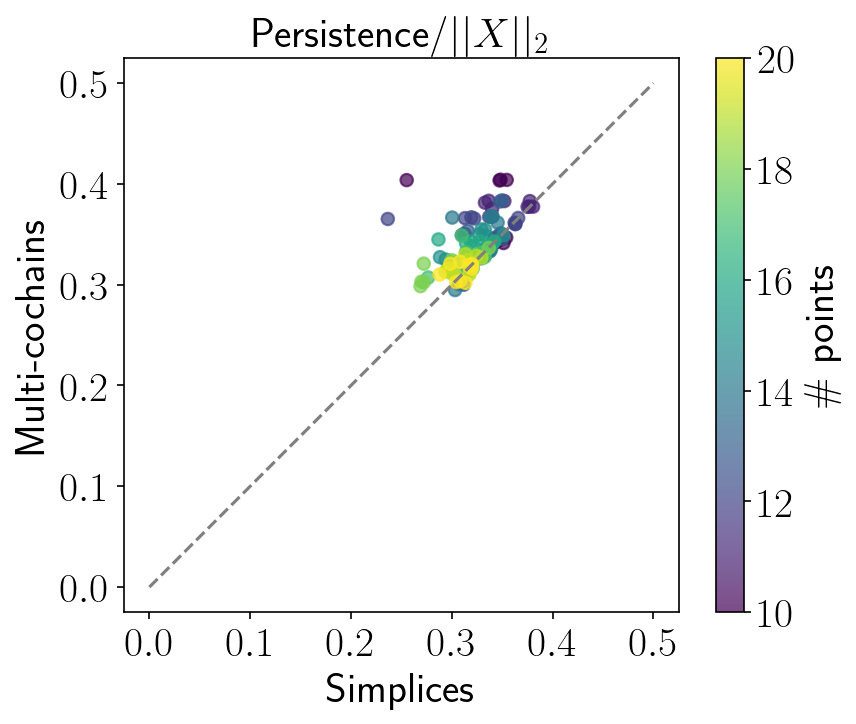}
    \end{subfigure}\hfill
    \begin{subfigure}[t]{0.6\textwidth}
    \centering
    \includegraphics[height=5cm]{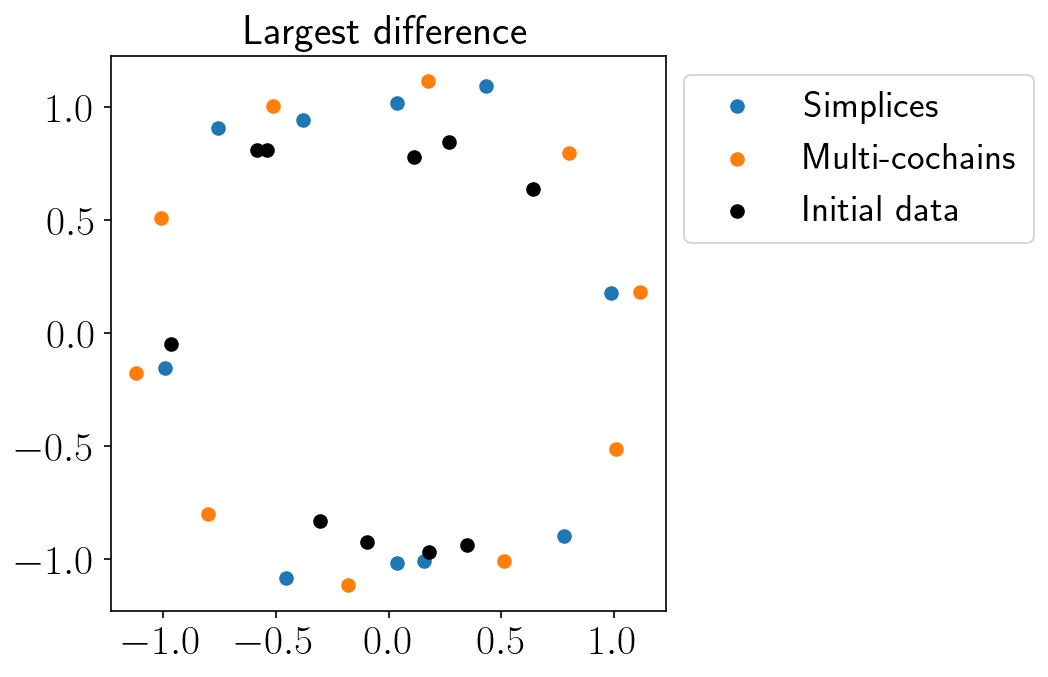}
    \end{subfigure}
    \caption{We compare the multi-cochains method to the simplices method on 110 random clouds of various sizes. We compare the normalized persistence as a scatter plot (left) where color shows the size of the point cloud. We also show the final configurations for the case where multi-cochains showed the most improvement over simplices (right).}
    \label{fig:multicochains}
\end{figure}

\section{Lower star filtration optimization}\label{sec:lowerstar}

We consider a model problem of reducing the number of points in the degree-$0$ persistence diagram of an image. We choose to target only death times of features so that local minima are preserved, but merged into one another. As before take two different approaches: the \emph{cochains method} which reduces the death content using gradient descent, and the \emph{simplices method} which reduces death time directly by identifying death simplices. Our model dataset consists of MNIST images of handwritten digits corrupted by a horizontal line of dark (but not black) pixels. We also add some small amount of noise so that birth and death pixels are always well-defined, and invert the image so that white pixels have value $0$ and black pixels have value $1$. 

We expect that reducing death content or death values will cause the pieces of the digit to merge together. We use $\varepsilon = 0.1$ for the cochains method and use a learning rate of $\gamma = 0.1$ for both methods. The end results are shown in \Cref{fig:MNIST}, and we see that each method makes some attempt to merge the disconnected pieces into a single symbol. When there are multiple natural ways to join two components together with a path, the simplices method makes a single choice (see the the $0$, $8$ and $9$). As shown in \Cref{fig:MNIST0}, this choice is highly sensitive to random noise. The cochains method, by contrast, can create two paths at the same time. The paths created are also closer in width to the rest of the digit (see digits $1$, $3$ and $5$), because the death cochain takes into account additional pixels with values near the death value, and tries to smoothly bridge between the connected components being merged. The paths created by the cochains method are more faithful to the original image (see digits $2$, $7$ and $9$), although for digit $6$ the interpolation is not accurate even though it appears realistic.

This experiment is not meant as a standalone application since the design is somewhat artificial. However, the benefits displayed by the cochains method suggest some general advantages in the lower-star setting, namely that the cochains method is able to handle situations where the death pixel location is ambiguous or unstable, and leads to more natural interpolations of the existing values when modifying an image.

\begin{figure}
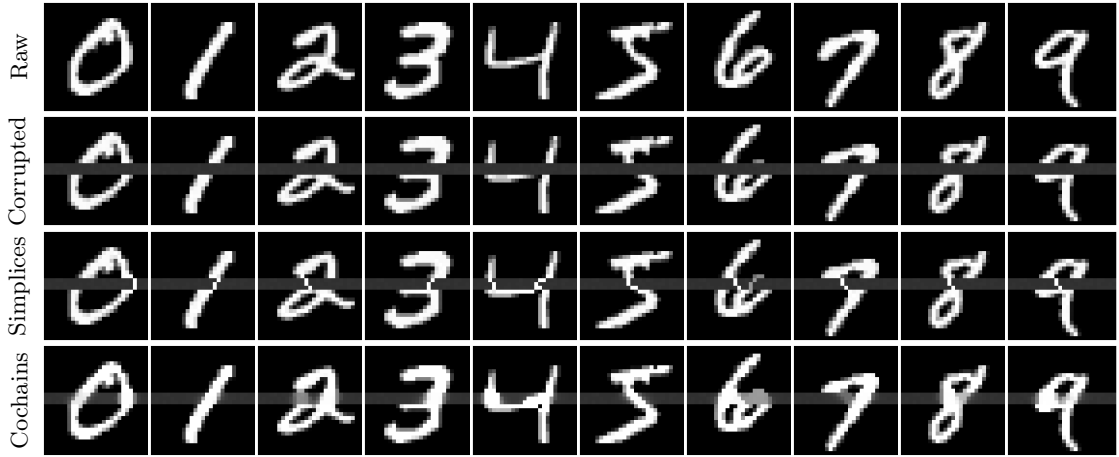

    \centering
    \begin{tikzpicture}[scale=0.95]
    \node[rotate=90] at (0.4, 1.6) {\footnotesize Raw};
    \node[rotate=90] at (0.4, 0) {\footnotesize Corrupted};
    \node[rotate=90] at (0.4, -1.6) {\footnotesize Simplices};
    \node[rotate=90] at (0.4, -3.2) {\footnotesize Cochains};
        \foreach \i in {1,...,10}{
        \node at (1.5*\i, 1.6) {
        \includegraphics[width=1.5cm]{imageshader/MNIST_sample\i _raw.png}
        };
        \node at (1.5*\i, 0) {
        \includegraphics[width=1.5cm]{imageshader/MNIST_sample\i _initial.png}
        };
        \node at (1.5*\i, -1.6) {
        \includegraphics[width=1.5cm]{imageshader/MNIST_simplices_sample\i _final.png}
        };
        \node at (1.5*\i, -3.2) {
        \includegraphics[width=1.5cm]{imageshader/MNIST_cochains_sample\i _final.png}
        };
        }
    \end{tikzpicture}
    \caption{MNIST images (top row) are corrupted by a horizontal band of dark gray pixels and slight noising (second row). We attempt to repair them by reducing the death time (third row) or death content (bottom row) of degree-$0$ features.}
    \label{fig:MNIST}
\end{figure}

\begin{figure}
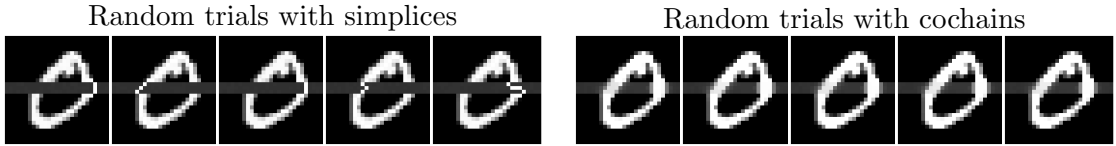

    \centering
    
    \begin{tikzpicture}[scale=0.95]
    \node at (4.5, 1) {Random trials with simplices};
        \node at (12.5, 1) {Random trials with cochains};
        \foreach \i in {1,2,3,4,5} {
        \node at (1.5*\i, 0) {
        \includegraphics[width=1.5cm]{imageshader/MNIST_simplices_sample1_trial\i _final.png}
        };
        }
        \foreach \i in {1,2,3,4,5} {
        \node at (1.5*\i+8, 0) {
        \includegraphics[width=1.5cm]{imageshader/MNIST_cochains_sample1_trial\i _final.png}
        };
        }
    \end{tikzpicture}
    \caption{Five random trials of the zero-digit repair task in \Cref{fig:MNIST}, with slightly different noise each time. Random trials for additional digits are in Supplemental Material.}
    \label{fig:MNIST0}
\end{figure}

\section{Feature weight optimization}\label{sec:features}

In this section, we use birth and death cochains to learn feature importance. Given an input dataset $X \in \mathbb{R}^{n \times d}$ consisting of $n$ observations with $d$ features, we consider the problem of reweighting the features to optimize some function of Vietoris-Rips persistence. This class of optimization tasks was first introduced in \cite{bubenik2023topological} in the context of feature selection for multivariate time series. To demonstrate birth and death cochains in this setting, we reproduce the experiment in Section 7.2 of that paper on finding periodicity in multivariate time series. 

We briefly recall the experimental setup. We consider a multivariate time series $f: \mathbb{R} \to \mathbb{R}^{10}$ with components
$$
f_i(t) = \sin\left(\frac{2\pi}{50}(t-K_i)\right) \quad i = 1,2,\ldots,10,
$$
where the $K_i$ are random phase shifts. We restrict to the domain $D = \{1,\ldots, 300\}$, giving 10 pure signals $f_1,\dots,f_{10}$ in \Cref{fig:puresines}. We shuffle the values of $f_4,\ldots,f_{10}$ to eliminate the periodicity and then add Gaussian noise with standard deviation $\sigma = 1.5$ (see \Cref{fig:sinewaves-noise}).\footnote{Note that our various random values (most importantly the phase shifts) are different to those in \cite{bubenik2023topological}, so our dataset (and results) will differ.} 

Let $\Sigma_{10} = \{(w_1, w_2, \ldots, w_{10}) \in \mathbb{R}^{10} \mid \forall_i w_i \geq 0,\ \sum_i w_i = 1\}$ be the standard $10$-simplex from which we draw our weights.
For weights $\mathbf{w} = (w_1,\ldots,w_{10}) \in \Sigma_{10}$, we reweight the time series to $w_f(\mathbf{w}):=(w_1 f_1,\ldots,w_{10} f_{10})$ and apply the sliding window embedding as in \cite{bubenik2023topological} with window length $L=250$. 
This yields $300-L+1=51$ points, where for $j=1,\dots,51$,
$$
X_j^{(\mathbf{w})}
:=
\big(
w_1 f_1(j), \ldots, w_1 f_1(j+L-1),\;
\ldots,\;
w_{10} f_{10}(j), \ldots, w_{10} f_{10}(j+L-1)
\big) \in \mathbb{R}^{10L} = \mathbb{R}^{2500}.
$$
Thus, each feature $f_i$ contributes a block of $L$ coordinates scaled by $w_i$, producing a pointcloud $\{X_1^{(\mathbf{w})},\ldots,X_{51}^{(\mathbf{w})}\} \subset \mathbb{R}^{2500}$, equipped with the $\ell_1$ norm \cite{bubenik2023topological}. 
We want to maximize the persistence or the persistence content of the longest degree-$1$ bar in the Vietoris-Rips persistence of this pointcloud. Since degree-$1$ features in sliding window embeddings are known to capture periodicity, this should result in the first three features being upweighted.

We run gradient ascent using the simplices method (maximizing persistence) for $1,000$ iterations and the cochains method (maximizing persistence content) for $100$ iterations, each with the learning rate $\gamma = 2^{-6}/10$ from \cite{bubenik2023topological}. The cochains method uses the multi-cochain method with $\mathcal{E} = \{0.01, 0.05, 0.1\}$. We also compare these two methods to a one-step approximation of gradient ascent using cochains. More precisely, we compute the gradient $\mathbf{w}'$ of persistence content at the uniform weights vector $\frac{1}{d}\mathbf{1}$, and project to the tangent space of the simplex $\Sigma_d$. Then, rather than do gradient ascent steps, we find the intersection of the line segment $\left\{\frac{1}{d}\mathbf{1} + t \mathbf{w}' \;\middle|\; t \in [0,\infty)\right\}$ and the boundary of $\Sigma_d$ and take the result as our estimate of the optimum. The practical benefits of this one step process are that we do not have to worry about convergence or dependence on step size, and we also reduce the number of persistent homology calculations to one. We call this last method the \emph{One-step (cochains)} method.

\begin{figure}
    \centering
    \begin{subfigure}{0.28\textwidth}
        \includegraphics[width=\textwidth]{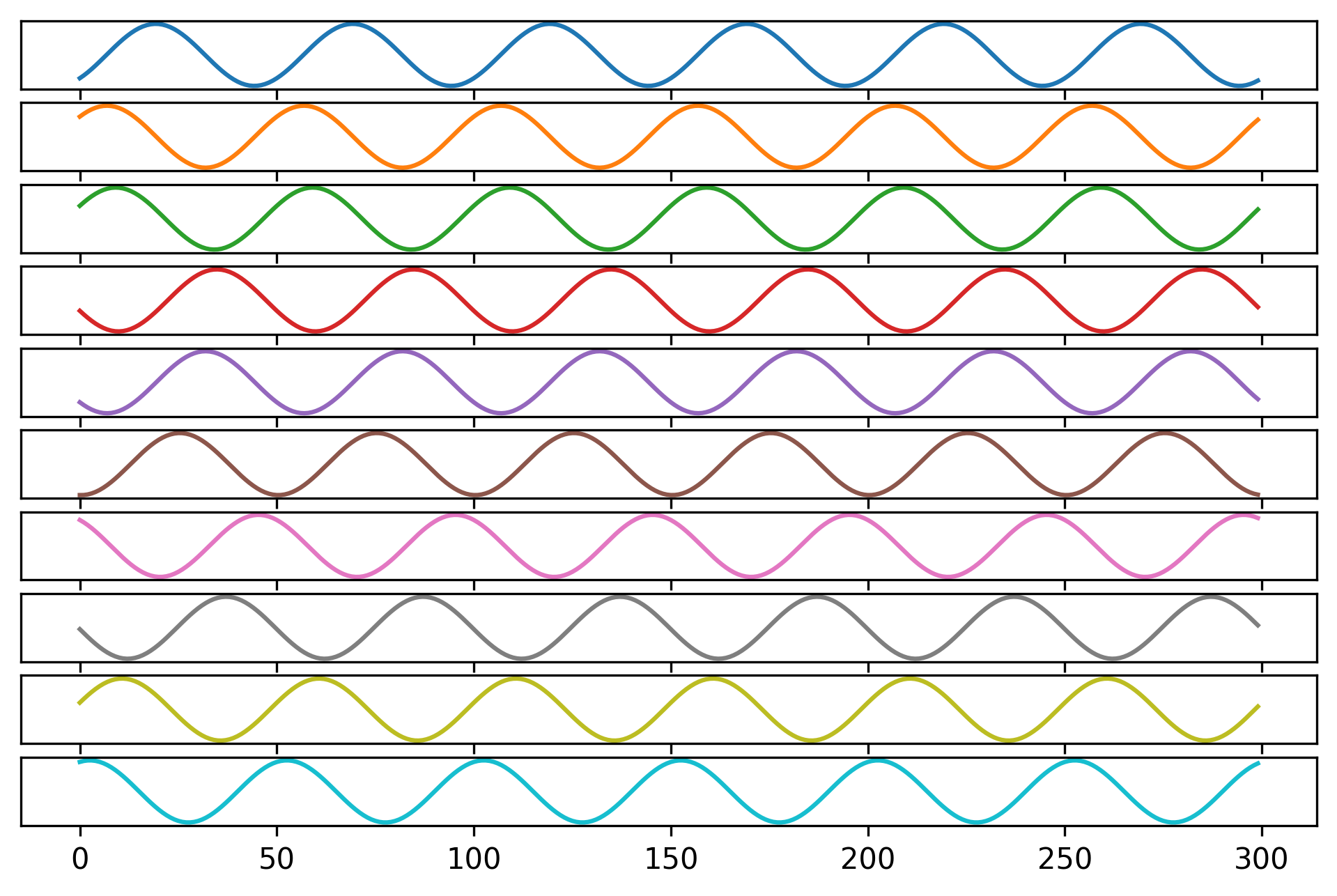}\caption{Pure sine waves}\label{fig:puresines}
    \end{subfigure}
    \begin{subfigure}{0.28\textwidth}
        \includegraphics[width=\textwidth]{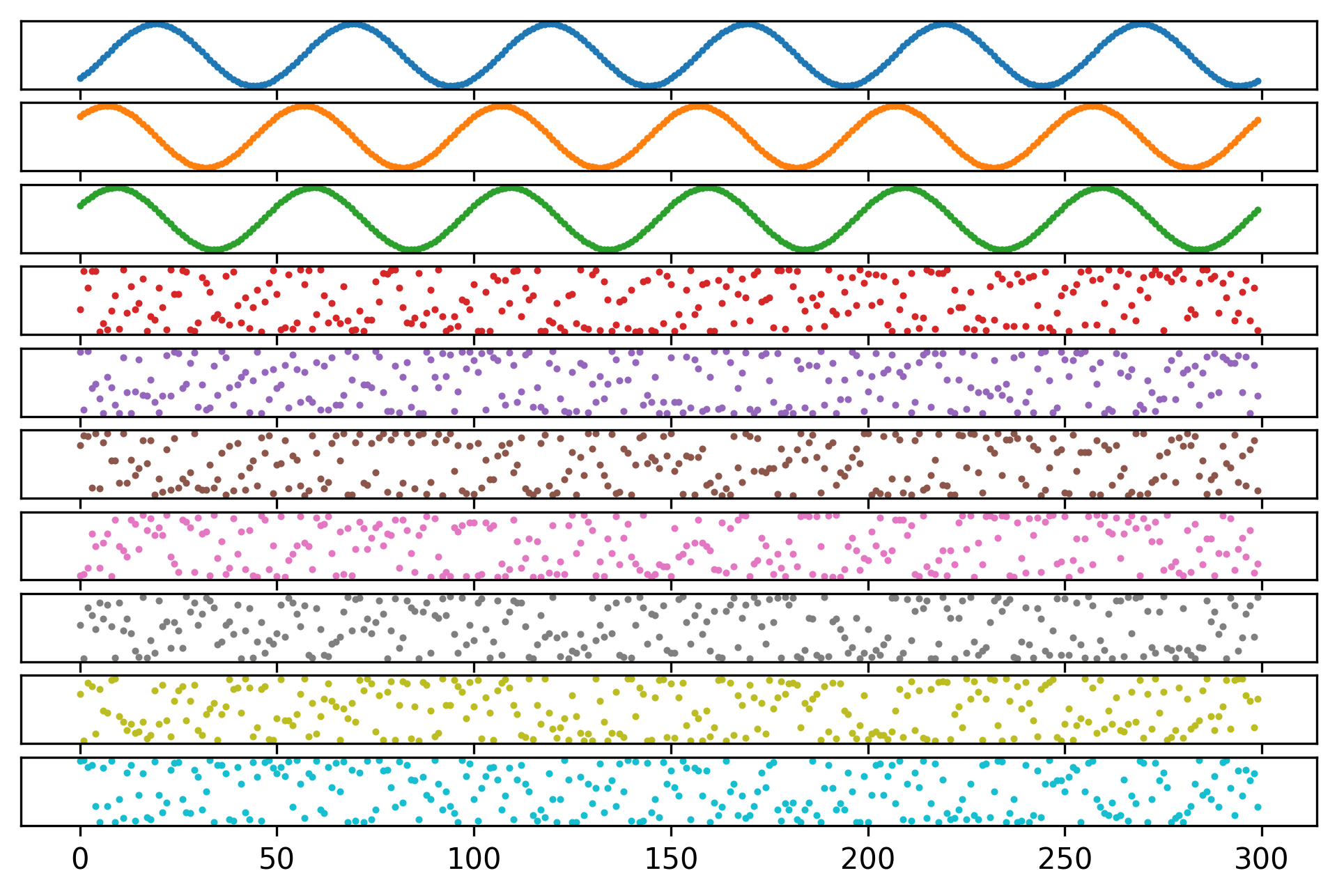}
        \caption{Shuffled}
        \label{fig:shuffled}
    \end{subfigure}
    \begin{subfigure}{0.28\textwidth}
        \includegraphics[width=\textwidth]{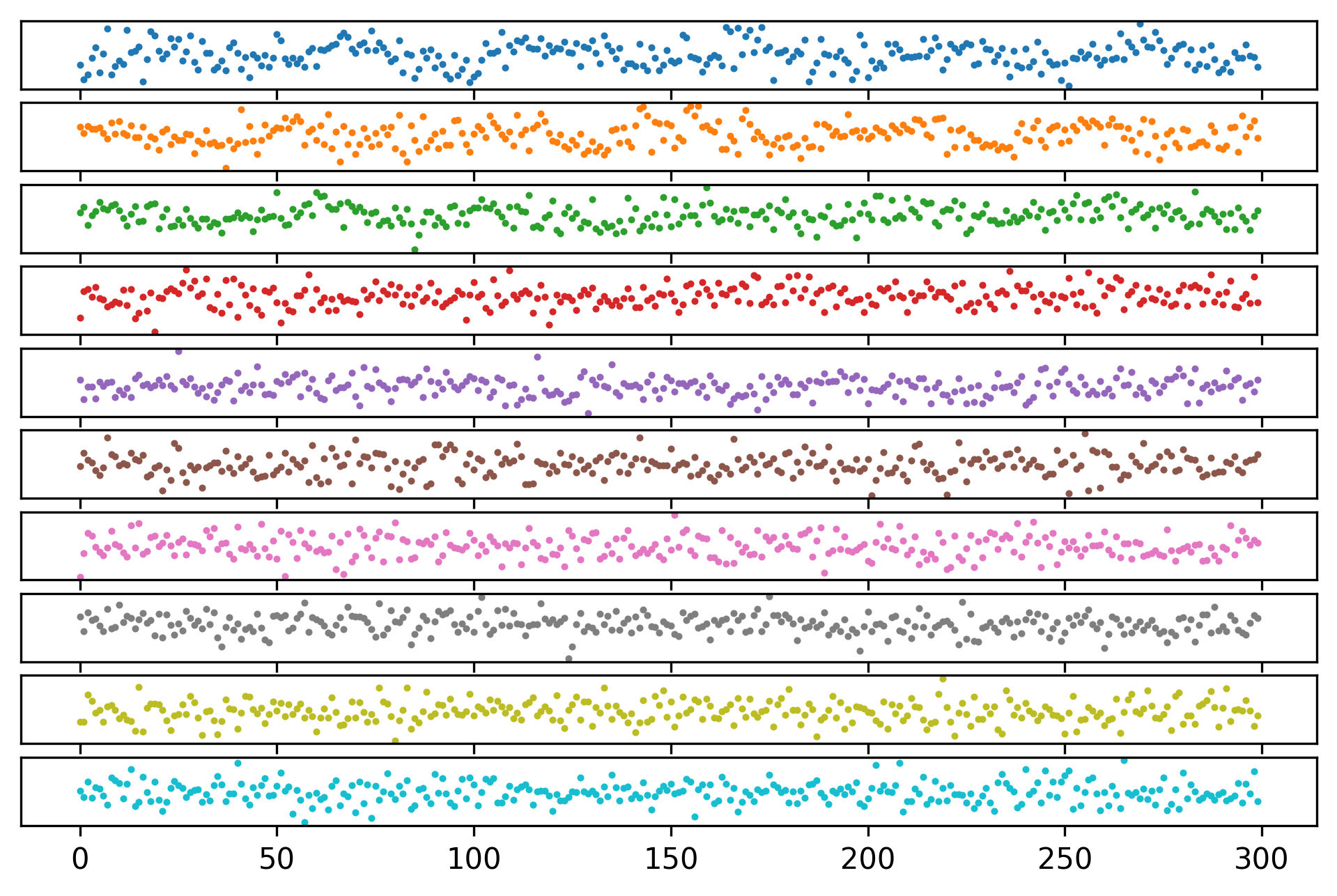}
        \caption{Noise added}
        \label{fig:sinewaves-noise}
    \end{subfigure}
    \caption{Data creation process for extracting periodicity, following \cite{bubenik2023topological}.}
    \label{fig:sinewaves}
\end{figure}

We repeat the experiment ten times, each time with new random phase shifts $K_i$ and added noise. We compare the persistence of the longest degree-$1$ bar achieved by each method in \Cref{fig:persistence_comparison}. We see that with the exception of Trial 8, the simplices and cochains methods have almost identical performance, while the one-step method lags behind but still yields higher persistence than the unweighted initial cloud. Trial 8 is a special case since the most prominent degree-$1$ cycle in the initial cloud does not accurately track the time dimension, but is instead an artifact of randomness. The extent to which the persistence of this feature can be increased is therefore highly limited since it cannot be promoted by weighting just a few features, and this is demonstrated in the low performance of the cochains method. The simplices method avoids this issue by random chance; adjusting weights based just on the birth edge happens to increase the correct weights to promote the ``true'' degree-$1$ feature into being the longest bar. After this, the method can easily maximize the persistence of the new feature as with other trials.  

\begin{figure}
    \centering
\begin{subfigure}{\textwidth}
    \centering
    \includegraphics[width=0.6\linewidth]{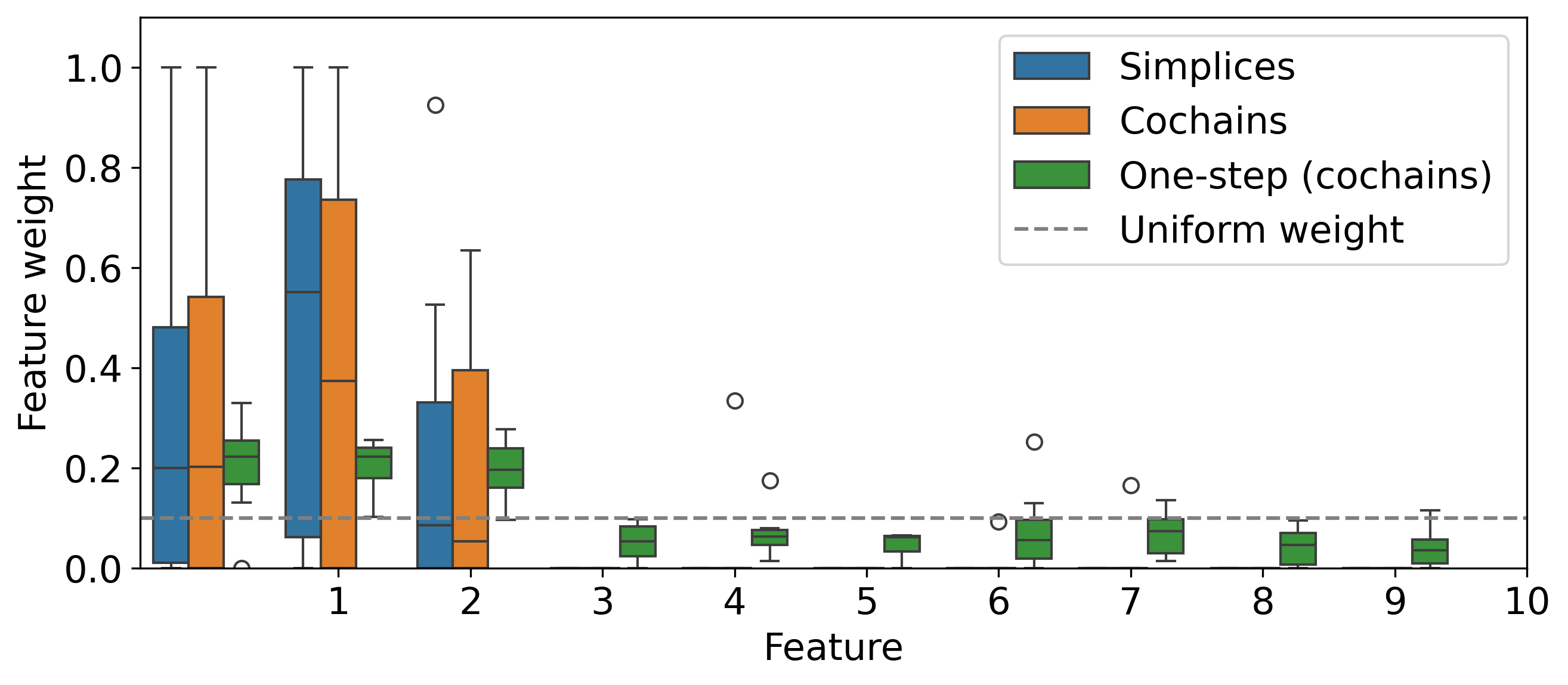}
    \caption{Distribution of learned feature weights over ten trials}
    \label{fig:feature_weights}
\end{subfigure}
    \begin{subfigure}{0.58\textwidth}
    \centering
        \includegraphics[height=4.5cm]{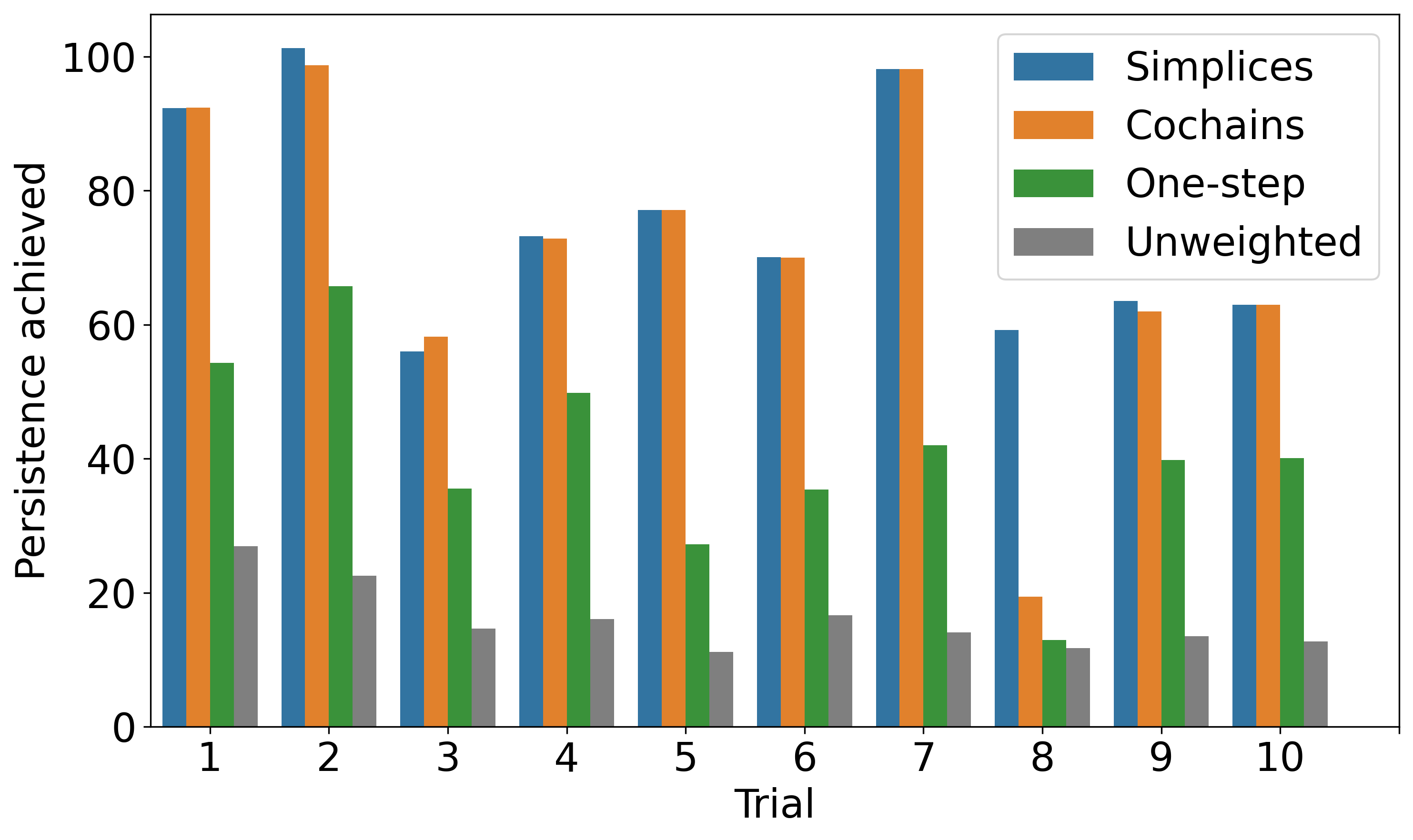} 
        \caption{Persistence of final weighted cloud}
        \label{fig:persistence_comparison}
    \end{subfigure}%
    \begin{subfigure}{0.4\textwidth}
    \centering
        \includegraphics[height=4.5cm]{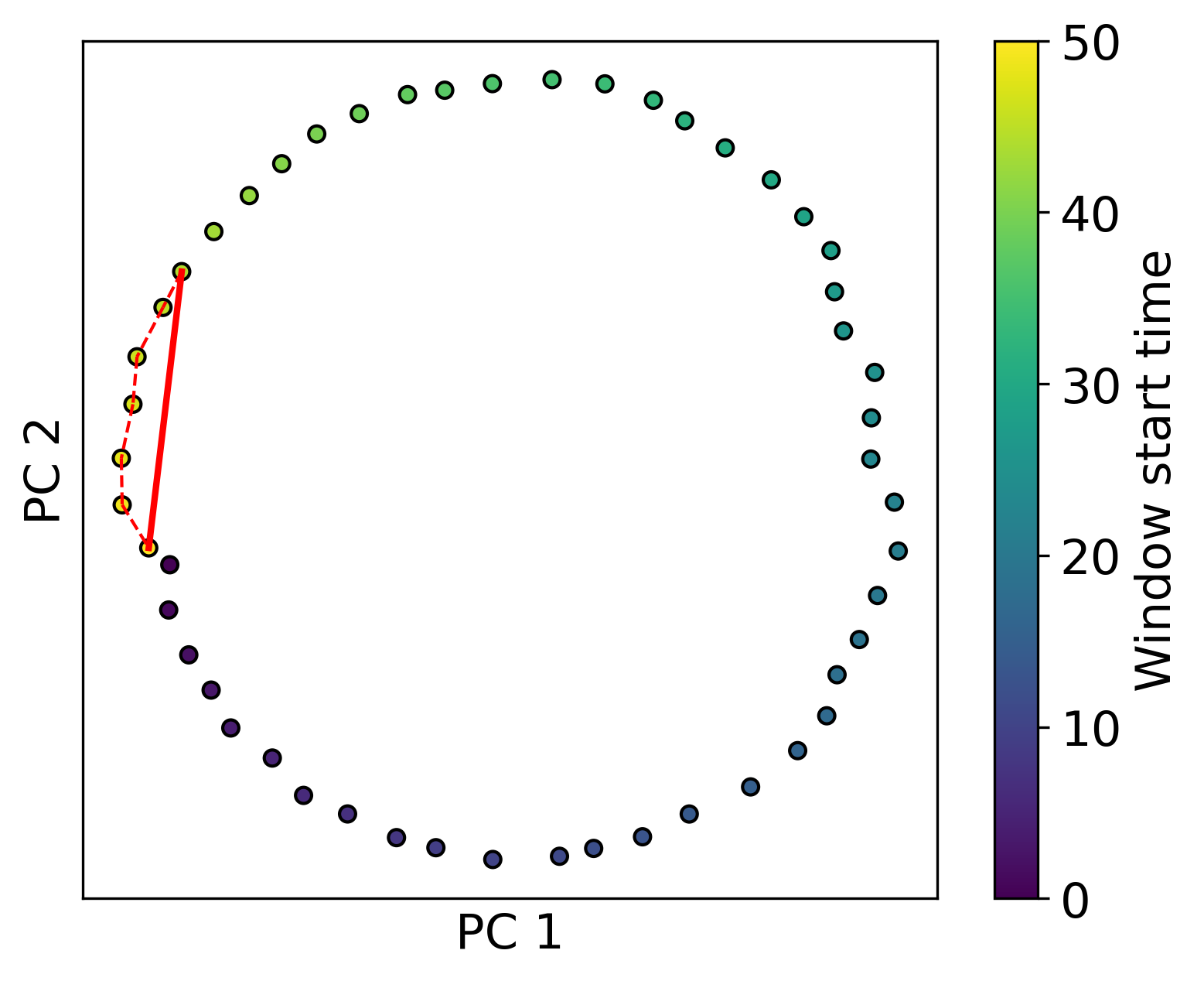}
        \caption{Starting cycle for Trial 8}
    \end{subfigure}
    \caption{Comparing three methods for optimizing feature weights to maximize the degree-$1$ persistence of a sliding window embedding using ten random trials. We show the distribution of learned feature weights (top) and the final persistence achieved by each method (bottom left). To explain the outlier behavior in Trial 8, we examine (bottom right) the most persistent degree-$1$ feature in the initial cloud via PCA, showing the cycle (dashed) and birth edge (solid).}
    \label{fig:feature_weighter_persistences}
\end{figure}

The final feature weights for each trial are shown in \Cref{fig:feature_weights}. Recall that the periodicity survives only in the first three features, which we should see upweighted. We see that while the simplices and cochains methods are more effective than the one-step (cochains) method at suppressing non-periodic features, their weighting of features 1--3 varies greatly. This is often because some of the first three signals are in phase and thus redundant. By contrast, the one-step method tends to overweight noise features but has lower variance in weights for the periodic features. As a result, the one-step method is a much better classifier in terms of non-periodic vs periodic features; the first three features are almost always upweighted relative to uniformity, and the others downweighted. We therefore propose that while the gradient ascent processes have better performance in terms of raw persistence, the one-step method provides a robust and efficient approximate method which can scale better to larger clouds since it requires fewer persistence computations.

\section{Real-world data application}\label{sec:arctic}

\subsection{Motivation}

We now use birth and death cochains together with feature selection to follow up on the investigation in \cite{ilagor2023visualizing} of Arctic ice extent images from 1999--2009 obtained from \cite{Long}. Each image is processed to be a binary $1530\times 1530$ image (with $1$ indicating ice), which we downsample to $51 \times 51$ and unravel to a vector in $\mathbb{R}^{2601}$. Each year's data is then a point cloud in $\mathbb{R}^{2601}$, equipped with the $\ell^1$ distance\footnote{We use $\ell_1$-distance rather than the OT distance in \cite{ilagor2023visualizing} to enable the reweighting. }, from which we can compute a Vietoris-Rips persistence diagram. With uniform weights, we can see a prominent degree-$1$ feature in all years except 2009, where the data collection ended before the end of the year (see \Cref{fig:arctic_grid} for the 2006 data and Supplemental Material for the rest). The annual freezing and melting cycle is not enough to explain these features by itself since if the melting and freezing processes were symmetrical, the shape of the data would be a line not a loop. It was proposed in \cite{ilagor2023visualizing} that the degree-$1$ features result from an assymetry in the melting and freezing stages, and a single example was given to support this. In this section will use feature reweighting and birth/death cochains to refine this hypothesis.

\subsection{Feature selection experiment}

We use the year 2006 as our main example, with results for other years shown in Supplementary Material. \Cref{fig:arctic_grid} shows snapshots of the ice extent for this year. We draw weight vectors from the $2601$-simplex $\Sigma_{2601}$, so that we have one weight for each pixel. For $\mathbf{w} = (w_1, \ldots, w_{2601}) \in \Sigma_{2601}$, we let $w_X(\mathbf{w})$ be the point cloud obtained from $X$ by multiplying feature $i$ by $w_i$. We try to find the weights $\mathbf{w}$ which maximize the persistence content of the reweighted point cloud $ w_X(\mathbf{w})$. Motivated by the ability of the one-step method in \Cref{sec:features} to effectively identify topologically relevant features, we approach our problem as a feature-selection task rather than just a reweighting task. To achieve this, we compute the gradient of the degree-$1$ persistence content at the uniform weights vector $\frac{1}{2601}\mathbf{1}$ and select only those pixels whose gradient value was in the top half of positive values. This produces a binary mask which can be used to restrict the image data to only the most relevant pixels for degree-$1$ persistence. 

We can display this binary mask as an image, as shown on the left in \Cref{fig:arctic_pca}. The regions most relevant to degree-$1$ persistence are clearly visible in white. To verify that this mask helps reveal and explain the degree-$1$ feature in the data, we perform principal components analysis (PCA) on the masked dataset and compare it to PCA on the raw data $X$ in \Cref{fig:arctic_pca}. Note that we can clearly see the loop in the data in the masked version. We see that the melting (approximately days 0-250) and freezing (approximately days 250-365) parts of the year are distinguished in the scatter plot by their coordinates along the second principal component (PC2). Specifically, the left hand side (near North America) melts last but also freezes last, creating an asymmetrical cycle. Note that this is not easily observable in \Cref{fig:arctic_grid}. The degree-$1$ feature is also more easily located in time as well: the loop is located in the middle of the year, rather than early or late. 

The same analysis is shown for other years in Supplemental Material. Across all years, we see some commonalities: the mask is supported in a circle at the edge of the main core, the masked data always shows a circle in the PCA plot in contrast to the raw data, and the second principal component always distinguishes the melting and freezing phases. The positive and negative regions for the second principal component can be in different places showing some year-to-year variation. Interestingly, when using all positive values instead of the top half for our mask, one of the years (2003) no longer shows a circle in the PCA plot because the selection is not conservative enough; this is what motivated our choice earlier. 

\begin{figure}
    \centering
    \includegraphics[height=4cm]{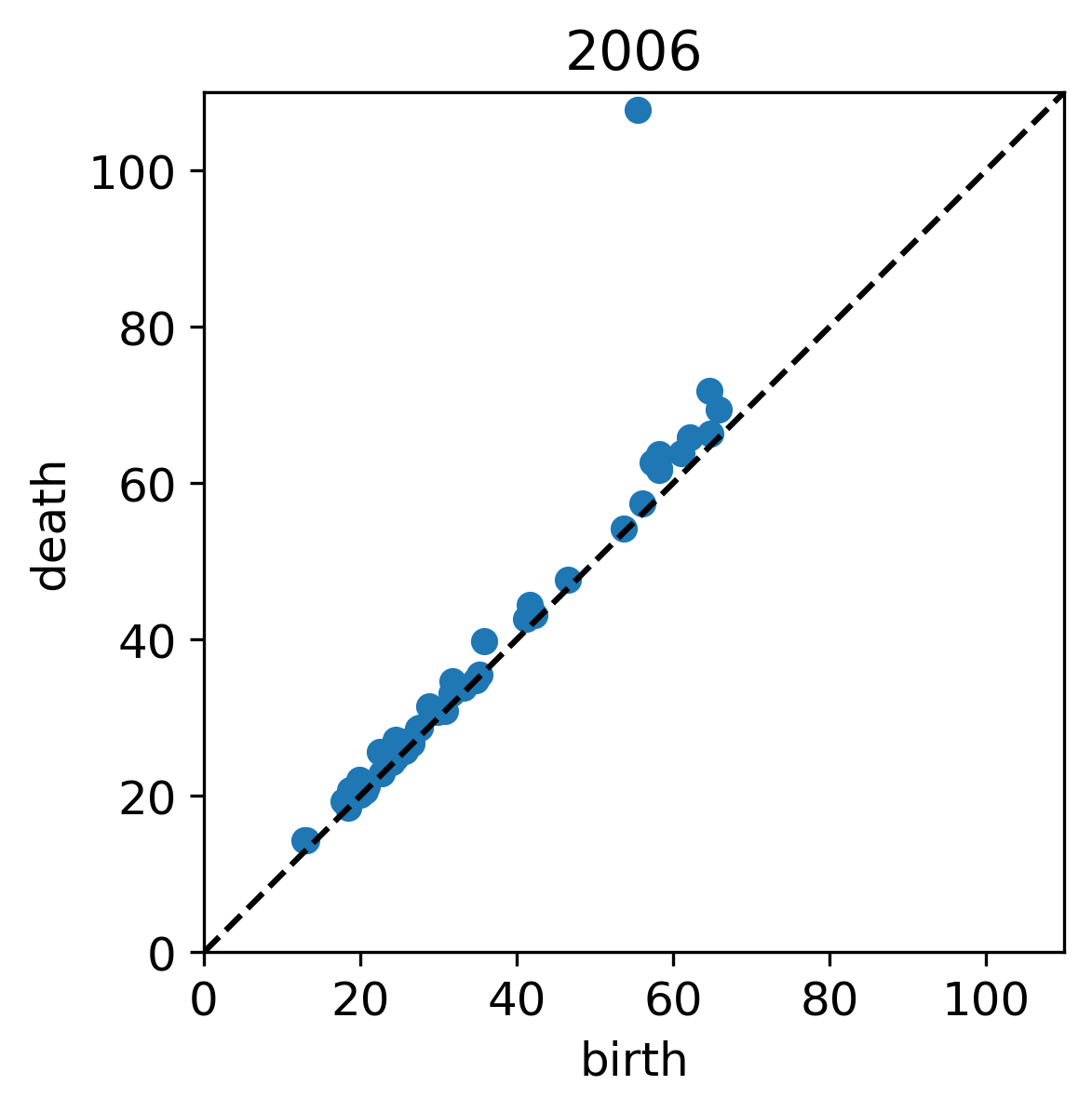}
    \includegraphics[height=4cm]{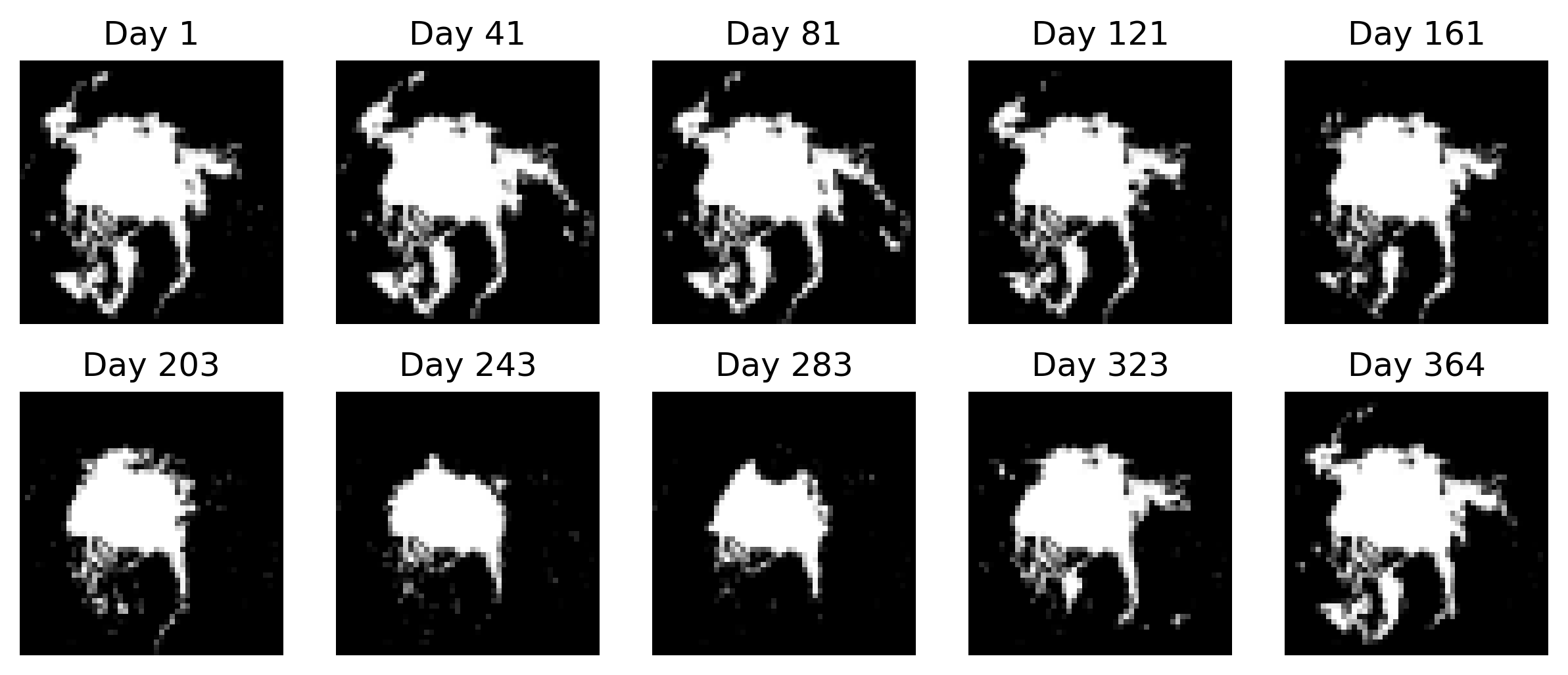}
    \caption{The degree-$1$ persistence diagram for Arctic ice data from 2006 (left) and snapshots of that data (right).}
    \label{fig:arctic_grid}
\end{figure}

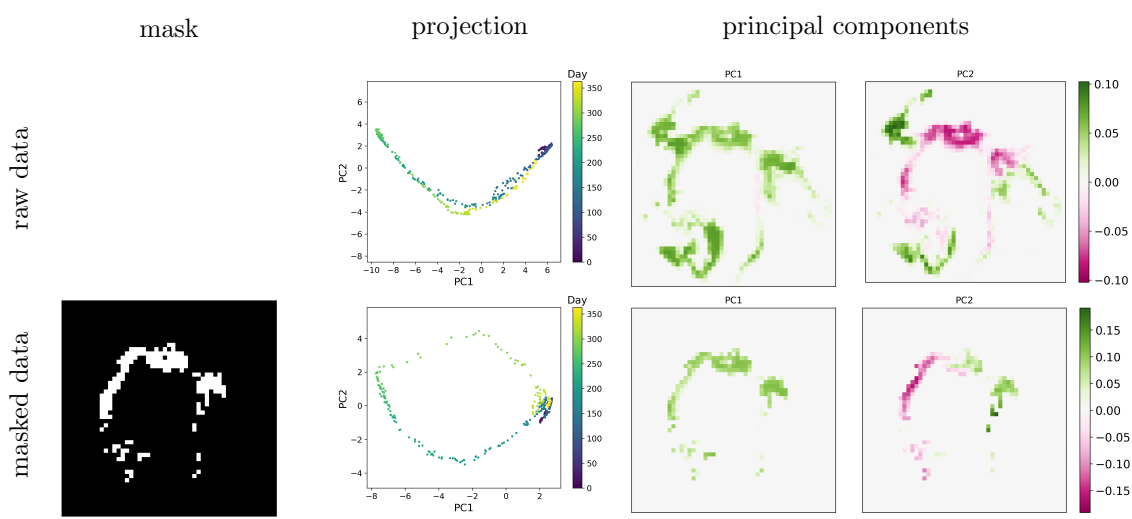
\begin{figure}
        \centering
    \begin{tikzpicture}[scale=1]
    \node at (-4,2) {\small mask};
    \node at (0,2) {\small projection};
    \node at (5, 2) {\small principal components};
    \node[rotate=90] at (-6, 0) {\small raw data};
        \node at (0,0) {\includegraphics[height=3cm]{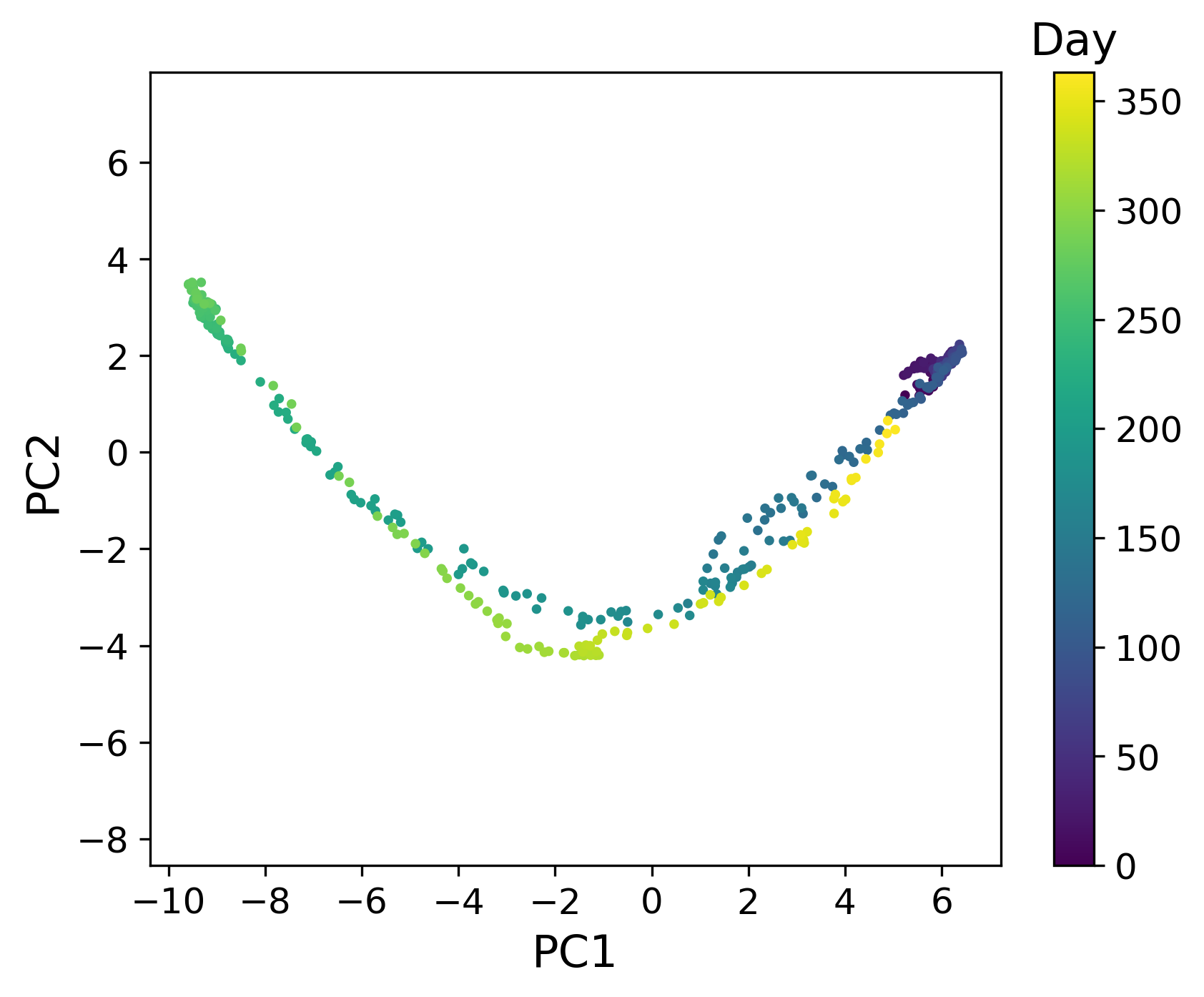}};
        \node at (3.5,0) {\includegraphics[height=3cm]{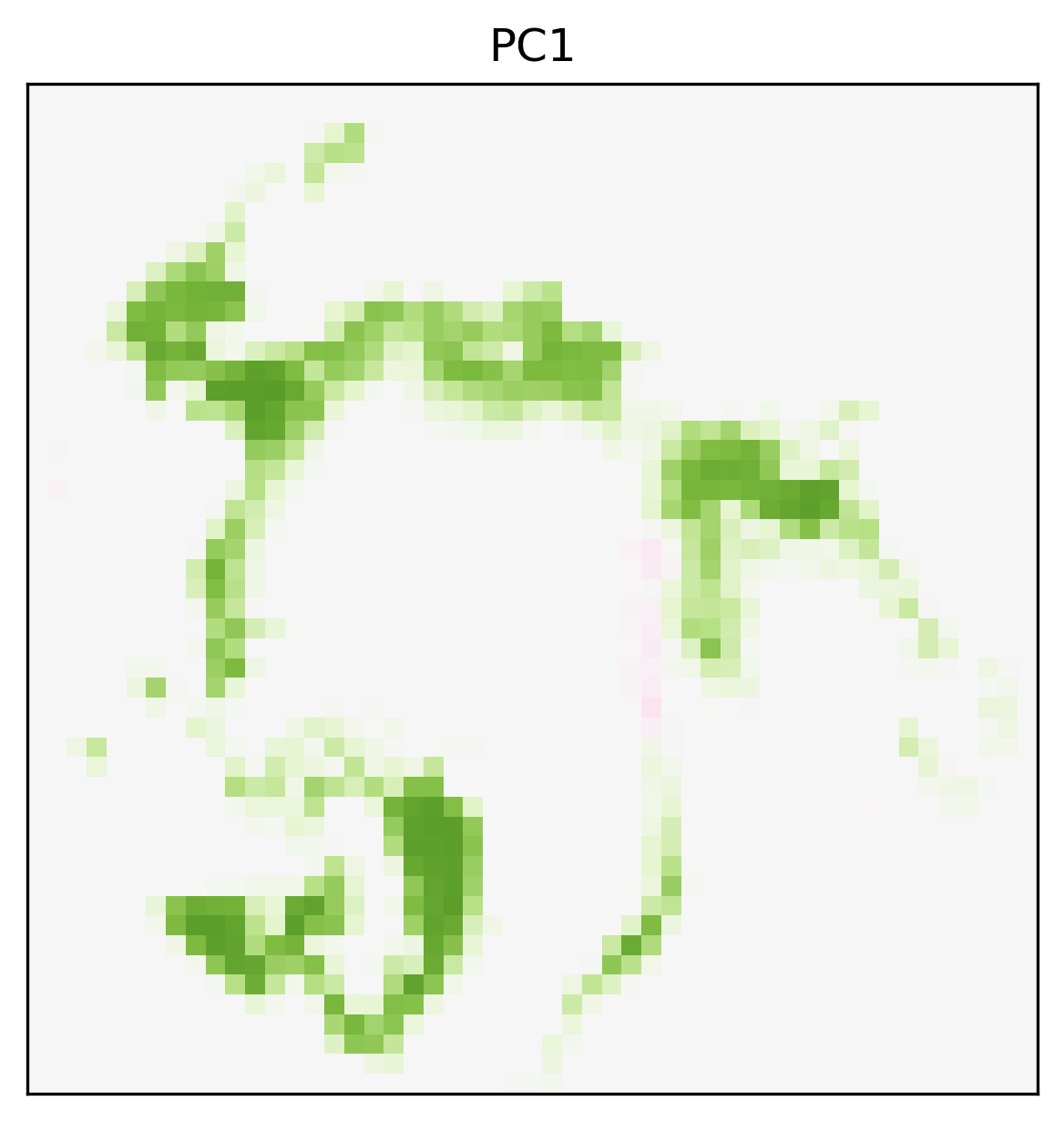}};
        \node at (7,0) {\includegraphics[height=3cm]{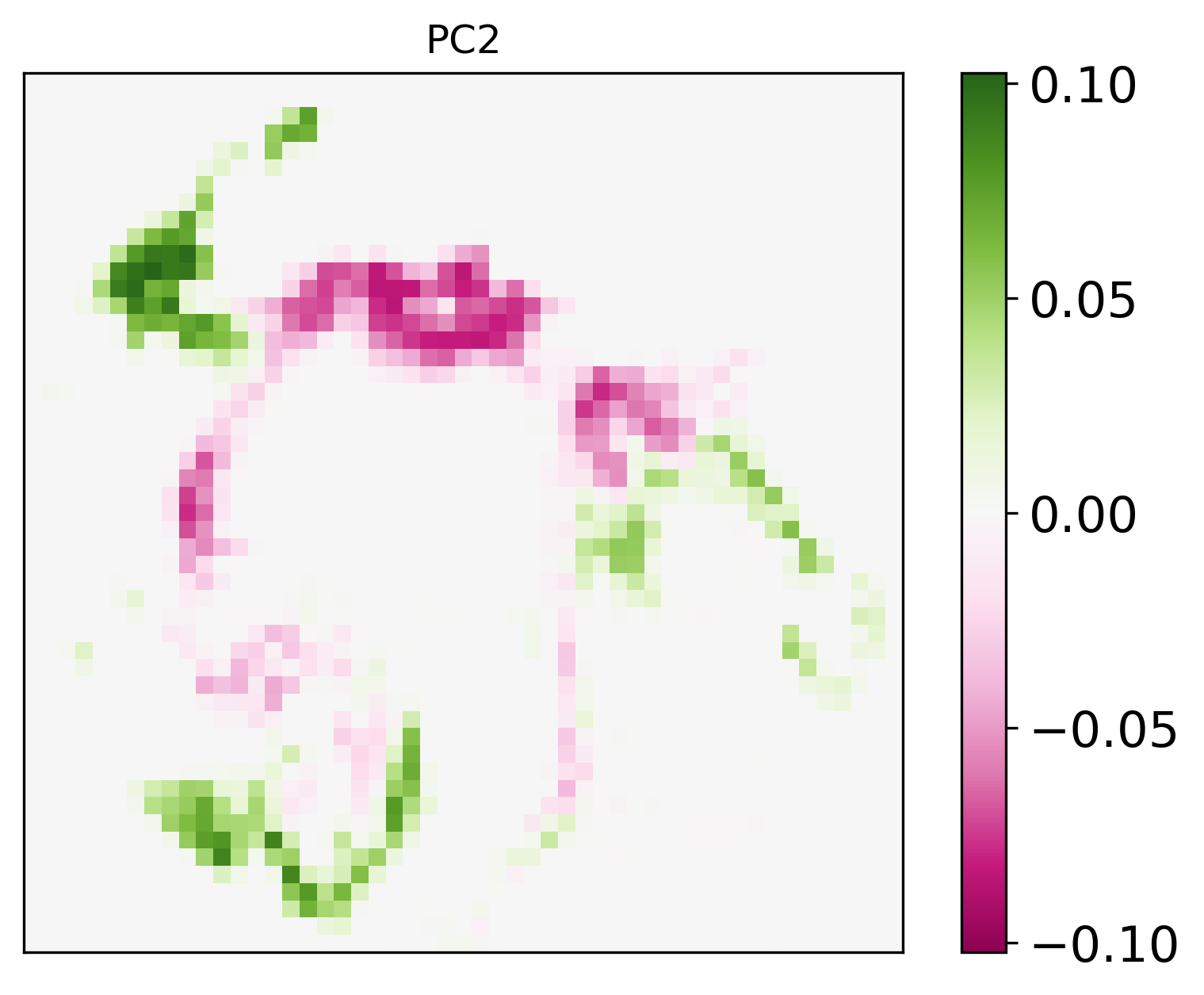}};
    \node[rotate=90] at (-6, -3) {\small masked data};
        \node at (-4,-3.05) {\includegraphics[height=3cm]{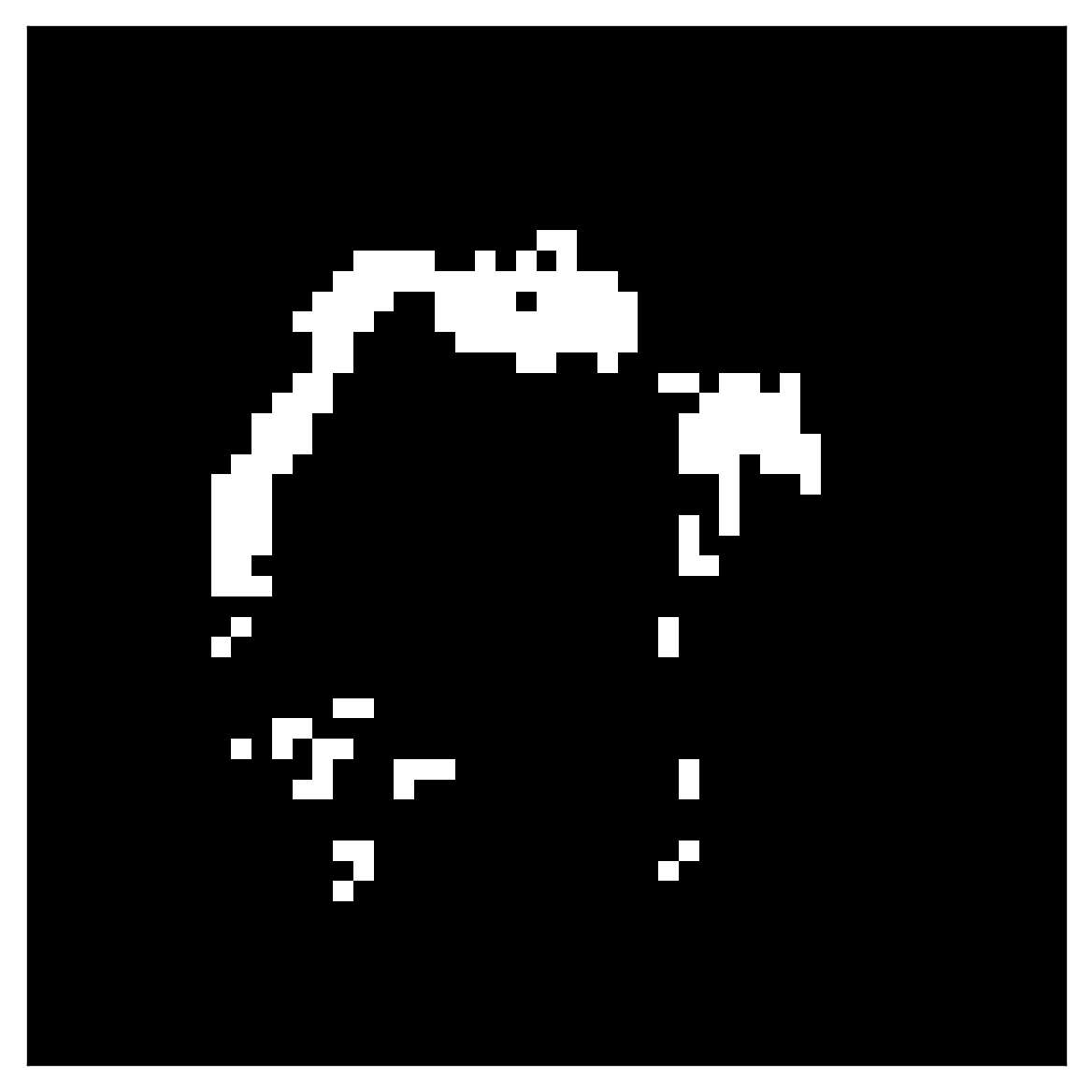}};
        \node at (0,-3) {\includegraphics[height=3cm]{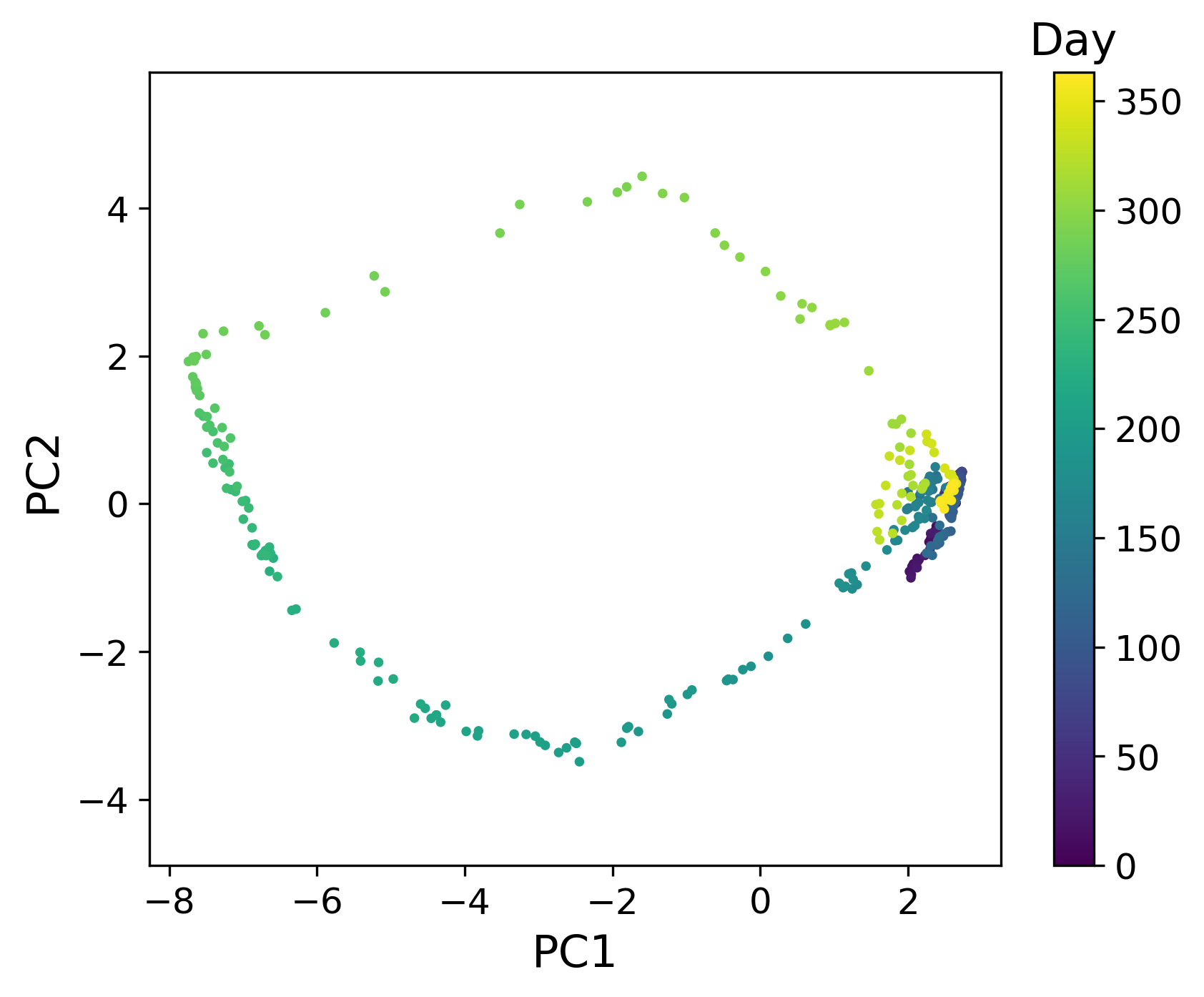}};
        \node at (3.5,-3) {\includegraphics[height=3cm]{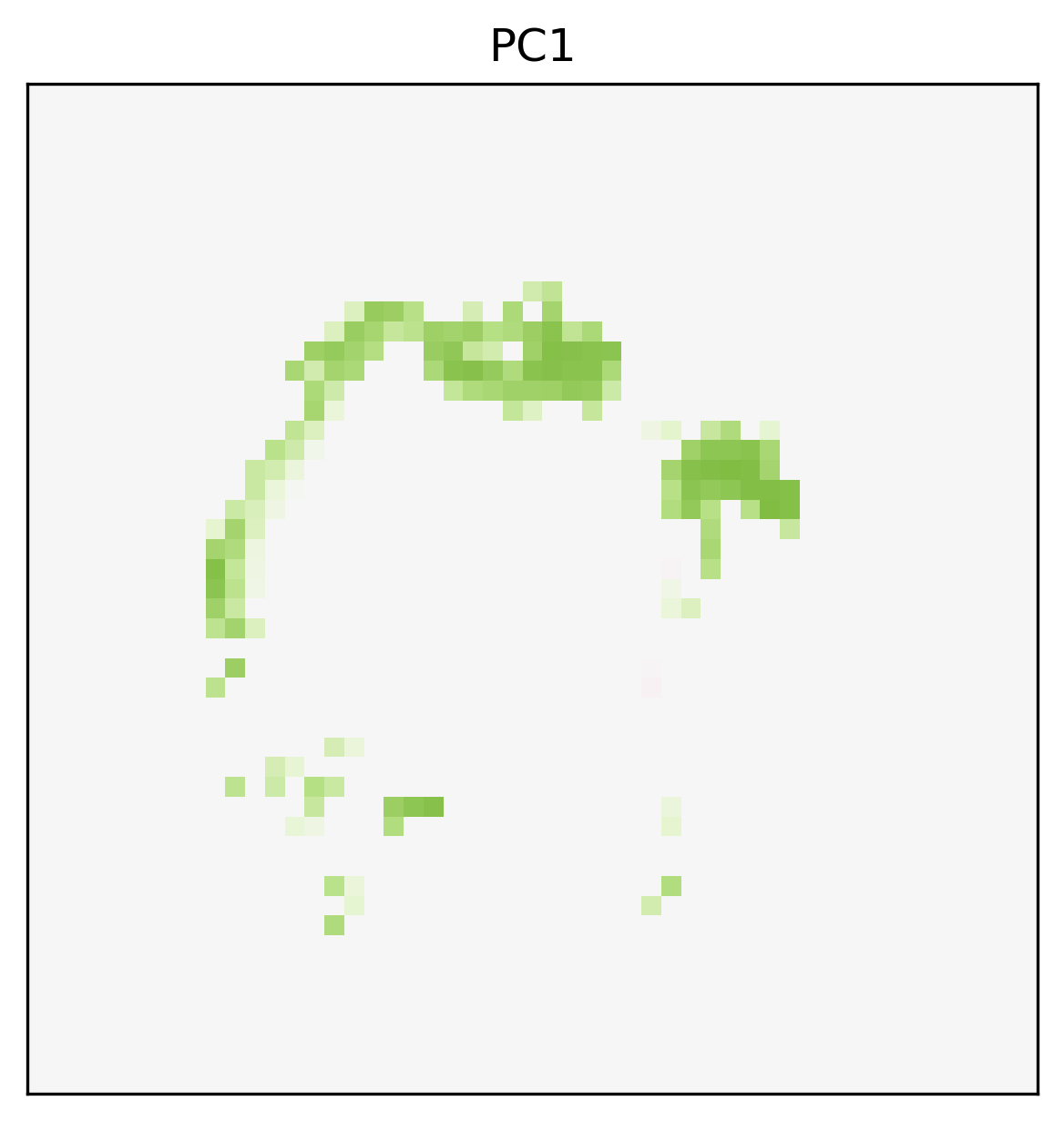}};
        \node at (7,-3) {\includegraphics[height=3cm]{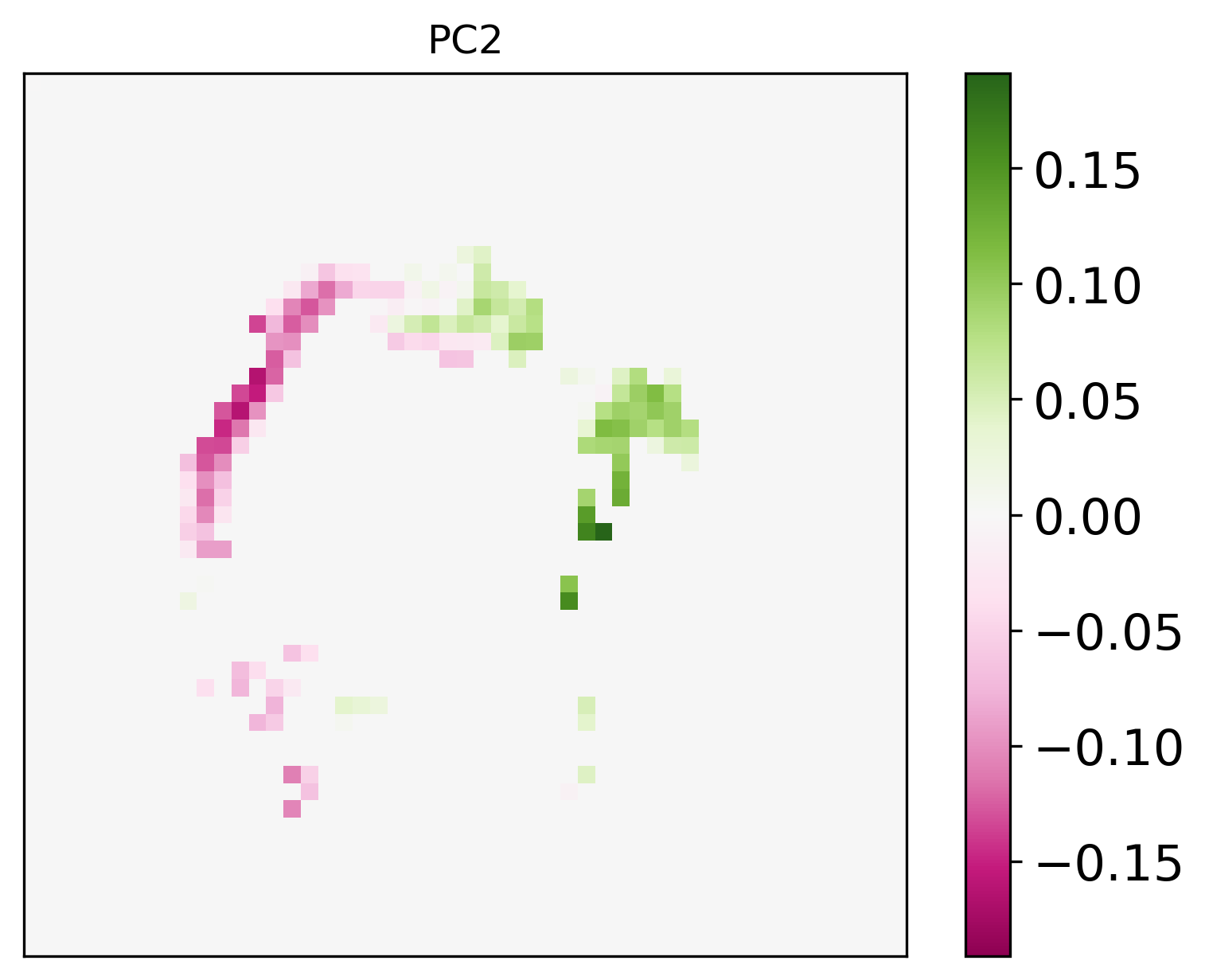}};
    \end{tikzpicture}
    
    \caption{Feature selection to reveal and explain degree-$1$ persistence in the 2006 data. We show the mask obtained by selecting pixels based on their gradients. We then show PCA plots for the raw data and masked data. The scatter plots show the projection onto the principal components. Each principal component is visualized as a heat map on the right.}
    \label{fig:arctic_pca}
\end{figure}

\section*{Acknowledgments}

We thank Ranthony A.~Clark and  Aziz Guelen for their input early on in the project. We are also grateful to Johnathan Bush for his assistance with the time series experiments, and to Zhengchao Wan for helpful discussions regarding \Cref{rmk:schur and laplacian}.

\section*{Code}

A codebase is available on github\footnote{\texttt{github.com/thomasweighill/birth-death cochains}}. Our code draws on the following TDA libraries: \texttt{gudhi}~\cite{maria2014gudhi} (including the cubical complex methods from \cite{wagner2011efficient}), \texttt{ripser}~\cite{bauer2021ripser}, \texttt{DREiMac}~\cite{perea2023dreimac}, and \texttt{persistent-cup-length}~\cite{ivshina2026persistentcuplength}. We use the MNIST loader created by Michał Dobrza\'{n}ski\footnote{\texttt{github.com/MichalDanielDobrzanski/DeepLearningPython/}}. 

\bibliographystyle{plain} 
\bibliography{main} 

\appendix

\section{Details of connection to persistent Laplacian}\label{app:psistentLaplacianproofs}

We recall the relevant notation from \Cref{sec:p laplacian} and explain \Cref{rmk:schur and laplacian}.
Let $\epsilon$ be the extension by zero map.
With respect to the standard inner product on cochains, we have the orthogonal decomposition
\[
\rmC^k(L)
=
\rmC^k(L,K)^\perp \oplus \rmC^k(L,K),
\qquad
\rmC^k(L,K)^\perp = \epsilon(\rmC^k(K)).
\]
The up Laplacian on $L$ is
\(
\Delta^{k,\mathrm{up}}_{L}
=
(\delta_L^k)^*\delta_L^k.
\)
Recall that 
\(
\partial_{k+1}^{L,K}
\)
is the restriction of $(\delta_L^k)^*$ to $\rmC^{k+1}_{L,K}$, and that the degree-$k$ persistent up Laplacian is
\(
\Delta^{k,\mathrm{up}}_{L,K}
=
\partial^{L,K}_{k+1}(\partial^{L,K}_{k+1})^*.
\)

We verify that the present setting fits the abstract framework of
\cite[Proposition~10]{gulen2023generalization}. Set
\[
\hat V:=\rmC^{k+1}(L),\qquad 
V:=\rmC^k(L),\qquad 
f:=(\delta_L^k)^*:\hat V\to V,
\qquad 
W:=\epsilon\big(\ker((\delta^{k-1}_{K})^*)\big)\subseteq V.
\]
Then \(f^*=\delta_L^k\) and
\(
f\circ f^*
= (\delta_L^k)^*\delta_L^k
= \Delta^{k,\mathrm{up}}_{L}.
\)
We verify that the preimage of \(W\) under \(f\) satisfies:
\[
f^{-1}(W)
=
\left\{\xi\in\hat V \;\middle|\; f(\xi)\in W \right\}
=
\left\{\xi\in\rmC^{k+1}(L)\;\middle|\; (\delta_L^k)^*(\xi)\in \epsilon(\ker((\delta^{k-1}_{K})^*)) \right\}
=\rmC^{k+1}_{L,K}.
\]
Let \(f_W:f^{-1}(W)\to W\) denote the restriction of \(f\) with codomain restricted
to \(W\). 
The following diagram summarizes the involved operators:
    \[
    \begin{tikzcd}[row sep=large]
        \epsilon(\rmC^k(K))
        \ar[r, two heads, "\proj_W"]
        \ar[rr, bend left=30, "(\partial^{L,K}_{k+1})^* " blue]
        &
        \epsilon\big(\ker((\delta^{k-1}_{K})^*)\big) 
        \arrow[r, " (f_W)^*"]
        \arrow[d, hook] 
        &
     \rmC^{k+1}_{L,K}  \arrow[d, hook] 
        \arrow[r, "f_W"] 
        \ar[rr, bend left=30, "\partial^{L,K}_{k+1}" blue]
      & 
      \epsilon\big(\ker((\delta^{k-1}_{K})^*)\big)
        \arrow[d, hook] 
        \arrow[r, hook, "\iota_W"]
        &  
        \epsilon(\rmC^k(K))
      \\
      &
     \rmC^k(L)
     \arrow[r, shift left=0.5ex, "\delta^{k}_{L}  = f^*"] 
     &
     \rmC^{k+1}(L) 
     \arrow[r, shift left=0.5ex, "(\delta^{k}_{L})^* = f"] 
     &\rmC^k(L). 
     &
    \end{tikzcd}
    \]
    Here $\proj_W$ denotes the orthogonal projection 
    $\epsilon(\rmC^k(K)) \to 
    W=\epsilon\big(\ker((\delta^{k-1}_{K})^*)\big)$,
    and $\iota_W$ denotes the natural inclusion 
    $W \hookrightarrow \epsilon(\rmC^k(K))$.

By \cite[Proposition~10]{gulen2023generalization}, the Schur complement of
\(f\circ f^*=\Delta^{k,\mathrm{up}}_{L}\) with respect to the orthogonal
decomposition \(\rmC^k(L)=W\oplus W^\perp\) satisfies
\begin{equation}\label{eq:sch_W}
    \mathrm{Sch}\left(\Delta^{k,\mathrm{up}}_{L},\,W\right)
=
f_W\circ (f_W)^*: W \to W.
\end{equation}

Recall that $\partial^{L,K}_{k+1}$ is defined as the restriction of $f=(\delta^{k}_{L})^*$ on $\rmC^{k+1}_{L,K}$, i.e. $\partial^{L,K}_{k+1} = \iota_W\circ f_W$. 
Taking the adjoint on both sides, we obtain $\delta^k_{L,K} =(f_W)^* \circ (\iota_W)^* =(f_W)^* \circ \proj_W $.
It follows immediately from \Cref{eq:sch_W} that
\begin{equation}\label{eq:cochain-schur}
    \Delta^{k,\mathrm{up}}_{L,K}
    =
    \partial^{L,K}_{k+1}\circ (\partial^{L,K}_{k+1})^*
    =
    \iota_W \circ f_W \circ (f_W)^* \circ \proj_W
    = \iota_W \circ \mathrm{Sch}\left(\Delta^{k,\mathrm{up}}_{L},\,W\right)\circ \proj_W.
\end{equation}

\section{Details of critical points proof}\label{app:criticalproofs}

Let $X = \{x_1, \ldots, x_n\}$ be the vertices of a regular $n$-gon, $n \geq 3$, lying in counter-clockwise order on the unit circle in $\mathbb{R}^2$. The Vietoris-Rips complexes on $X$ have already been thoroughly studied in the literature. Let $\delta = ||x_1 - x_2||_2$. For $r < \delta$, $\mathbf{VR}(X; r)$ is a set of vertices, with no homology except degree $0$. For higher $r$, the $1$-skeleton of $\mathbf{VR}(X; r)$ is the $k$-th graph power of the cycle graph with $n$ vertices for some $k$ depending on $r$, and the complex $\mathbf{VR}(X; r)$ is therefore the clique complex of this graph. 
Here, the $k$-th \emph{graph power} $G^k$ of a graph $G$ is a new graph with the same vertex set in which two vertices are connected by an edge if and only if their graph distance in $G$ is at most $k$.
Corollary 6.7 in \cite{adamaszek2013clique} contains a full classification of the clique complexes of graph powers up to homotopy equivalence. We will only need a very small piece of this result, stated in the following lemma.

\begin{lemma}[\cite{adamaszek2013clique}]\label{lem:vr}
Let $X$ be the vertices of a regular $n$-gon. Then
\begin{itemize}
    \item[(a)] there is an interval $J  = [\delta, y) $ such that $\mathbf{VR}(X; r)$ has non-trivial degree-$1$ homology if and only if $r \in J$,
    \item[(b)] for $r \in J$, $\mathbf{VR}(X; r) \simeq S^1$, so that $\rmH_1(\mathbf{VR}(X; r)) \cong \mathbb{R}$, and
    \item[(c)] for $r \in J$, the inclusion $\mathbf{VR}(X;\delta) \to \mathbf{VR}(X; r)$ is a homotopy equivalence.
\end{itemize}
\end{lemma}

Note that part (c) in \Cref{lem:vr} is not explicitly stated in \cite{adamaszek2013clique} but we can get it from, for example, Proposition 4.9~in \cite{adamaszek2017vietoris}. 

Let $D_n$ denote the \emph{dihedral group}, i.e.\ the group of $2n$ symmetries of the regular $n$-gon $X$, consisting of $n$ rotations and $n$ reflections: a rotation by $k$ sends $x_i \mapsto x_{i+k \pmod{n}}$, and a reflection sends $x_i \mapsto x_{-i \pmod{n}}$.
The dihedral group $D_n$ acts on $X$ by isometries. Therefore, it also acts on $\mathbf{VR}(X; r)$ simplicially, and hence on cohomology as well. For a group element $g \in D_n$, we abuse notation to denote the action of $g$ on cochains and on cohomology by $g^\ast$.

\begin{lemma}\label{lem:symmetry}
    Let $X$ be the vertices of a regular $n$-gon. Then
    \begin{itemize}
        \item[(a)] The Vietoris-Rips persistence diagram of $X$ has a single bar $[b,d)$ in dimension $1$. 
        \item[(b)] At every $r \in [b,d)$, we can choose a representative cocycle $\alpha $ generating $\rmH^1(\mathbf{VR}(X; r))$ such that: for any $g \in D_n$,
    \begin{equation}\label{eq:Dninvariance}
        g^\ast \alpha = (-1)^{\mathsf{sgn(g)}} \alpha,
    \end{equation}
    where $\mathsf{sgn(g)}$ is $1$ (resp. $-1$) if $g$ is orientation-preserving (resp. orientation-reversing).
    \item[(c)]  
    Let $\epsilon > 0$ be arbitrarily small.
    If $\birthcochain $ and $\omega$  are $\varepsilon$-birth and $\varepsilon$-death cochains for the persistent cohomology class generating the bar $[b,d)$, then $\birthcochain $ and $\omega$ satisfy: for any $g \in D_n$,
    \begin{align*}
        g^\ast \birthcochain  = &(-1)^{\mathsf{sgn(g)}} \birthcochain  \\
        g^\ast \omega = & (-1)^{\mathsf{sgn(g)}} \omega.
    \end{align*}
            \end{itemize}
\end{lemma}
\begin{proof}
(a): Follows from \Cref{lem:vr}.

(b): For $r \in [b,d)$, the map $\mathbf{VR}(X; b) \to \mathbf{VR}(X; r)$ is a homotopy equivalence by \Cref{lem:vr}. It is an easy exercise in cohomology that the cohomology classes for $\mathbf{VR}(X; b)$, which is cycle graph, are determined by their evaluation on the $1$-chain $\nu = (x_1, x_2) + (x_2, x_3) + \cdots + (x_n, x_1)$. Note that for any $g \in D_n$, $g_\ast \nu = (-1)^{\mathsf{sgn}(g)} \nu$ where $g_\ast$ is the action of $g$ on chains. If $\alpha$ is a degree-$1$ representative cocycle at $r$, then 
    $$
    \birthcochain  := \frac{1}{n}\sum_{g \in D_n} (-1)^{\mathsf{sgn(g)}} g^\ast \alpha
    $$
    is a cocycle such that $\birthcochain (\nu) = \alpha(\nu)$, so that $[\alpha] = [\birthcochain ]$. It also satisfies \Cref{eq:Dninvariance} by construction.

    (c): Let $\beta \in \rmC^1(\mathbf{VR}(X; d-\epsilon))$ be the representative cocycle satisfying conditions in (b), and $\alpha$ be the restriction of $\beta$ to $\mathbf{VR}(X; b+\epsilon)$.
    Then $\alpha$ is born between $X_{b-\varepsilon}$ and $X_{b+\varepsilon}$, and $\beta$ dies between $X_{d-\varepsilon}$ to $X_{d+\epsilon}$.

    Let $g \in D_n$. The map $g^\ast$ on cohomology preserves the $\ell^2$ norm on cochains since it permutes the simplices. Since $g$ is a simplicial map it is a standard fact that $g^\ast$ commutes with the coboundary operator. 
    By \cref{def:epsiloncochain}, $\eta$ is the birth cochain of $\alpha$ and $\omega$ is the death cochain of $\beta$.
    
    Thus, they are the unique minimizer of the optimization problems with feasible set
    \[
    V_{\mathrm{birth}} 
    := \left\{\, \alpha' \in \rmC^1(X_{b+\epsilon}) \;\middle|\; \alpha'|_{X_{b-\varepsilon}} = 0,\; \alpha' \in [\alpha]_{b+\epsilon} \,\right\}
    \]  
    and
    \[
    V_{\mathrm{death}} 
    := \left\{\, \omega\in \rmC^{2}(X_{s}) \;\middle|\; \omega=\delta^1_{X_{d+\varepsilon}} \beta',\; [\beta'|_{X_{d-\epsilon}}]=[\beta]_{d-\epsilon} \,\right\}
    \]
    respectively, and objective function $F(\cdot) = \|\cdot\|_2^2$. We therefore wish to understand the action of $D_n$ on these feasible sets. Let $\alpha' \in V_{\mathrm{birth}}$ and $g \in D_n$. Since $X_{b-\varepsilon}$ is closed under the action of $D_n$, $g^\ast \alpha'\vert_{X_{b-\varepsilon}} = 0$ as required. Since $g$ sends coboundaries to coboundaries, we have that 
        $$
        [g^\ast \cdot \alpha'] = [g^\ast \cdot \alpha] = (-1)^{\mathsf{sgn(g)}} [\alpha]
        $$
        Similarly, if $\omega = \delta^2_{X_{d+\varepsilon}} \beta' \in V_{\mathrm{death}}$ with $\beta'|_{X_{d-\epsilon}} \in [\beta]_{d-\epsilon}$, then $g^\ast \cdot \omega$ is the coboundary of $g^\ast \cdot \beta'$ and 
        $$
        [(g^\ast \cdot \beta')\vert_{X_{d-\varepsilon}}]
        = [g^\ast \cdot (\beta'\vert_{X_{d-\varepsilon}})]     
        = [g^\ast \cdot \beta]_{d-\varepsilon} = (-1)^{\mathsf{sgn(g)}} [\beta]_{d-\epsilon}
        $$
        
     We have thus shown that the effect of $g$ on the feasible sets is either to fix them set-wise or to multiply each vector by $-1$, depending on the sign of $g$. Therefore, if $\eta$ is the unique minimizer (\Cref{rem:uniqueness}) for the birth problem, then by uniqueness we must have that $(-1)^{\mathsf{sgn(g)}} \eta$ is also optimal, and hence must be equal to $\eta$ as required. 
     A similar discussion holds for the death cochain $\omega$.
\end{proof}

\section{Supplementary Material}

\subsection{$H_1$ optimization runs with different learning rates}

\begin{figure}
    \centering

    \begin{subfigure}{\textwidth}
    \centering
        \begin{tikzpicture}
        \tikzmath{\eps = 0.030;}
            \node at (0,0) {\includegraphics[width=5cm]{pointmover/death_minus_birth_harmonic_epsilon\eps .png}};
            \node at (4,0) {\includegraphics[width=2.25cm]{pointmover/point_movement_harmonic_gamma0.010_epsilon\eps .png}};
            \node at (6.25,0) {\includegraphics[width=2.25cm]{pointmover/point_movement_harmonic_gamma0.020_epsilon\eps .png}};
            \node at (8.5,0) {\includegraphics[width=2.25cm]{pointmover/point_movement_harmonic_gamma0.030_epsilon\eps .png}};
            \node at (10.75,0) {\includegraphics[width=2.25cm]{pointmover/point_movement_harmonic_gamma0.040_epsilon\eps .png}};
        \end{tikzpicture}
        \caption{Cochains, $\varepsilon_0 = 0.03$}
    \end{subfigure}

    \begin{subfigure}{\textwidth}
        \centering
        \begin{tikzpicture}
        \tikzmath{\eps = 0.040;}
            \node at (0,0) {\includegraphics[width=5cm]{pointmover/death_minus_birth_harmonic_epsilon\eps .png}};
            \node at (4,0) {\includegraphics[width=2.25cm]{pointmover/point_movement_harmonic_gamma0.010_epsilon\eps .png}};
            \node at (6.25,0) {\includegraphics[width=2.25cm]{pointmover/point_movement_harmonic_gamma0.020_epsilon\eps .png}};
            \node at (8.5,0) {\includegraphics[width=2.25cm]{pointmover/point_movement_harmonic_gamma0.030_epsilon\eps .png}};
            \node at (10.75,0) {\includegraphics[width=2.25cm]{pointmover/point_movement_harmonic_gamma0.040_epsilon\eps .png}};
        \end{tikzpicture}
        \caption{Cochains, $\varepsilon_0 = 0.04$}
    \end{subfigure}

    \begin{subfigure}{\textwidth}
        \centering
        \begin{tikzpicture}
        \tikzmath{\eps = 0.050;}
            \node at (0,0) {\includegraphics[width=5cm]{pointmover/death_minus_birth_harmonic_epsilon\eps .png}};
            \node at (4,0) {\includegraphics[width=2.25cm]{pointmover/point_movement_harmonic_gamma0.010_epsilon\eps .png}};
            \node at (6.25,0) {\includegraphics[width=2.25cm]{pointmover/point_movement_harmonic_gamma0.020_epsilon\eps .png}};
            \node at (8.5,0) {\includegraphics[width=2.25cm]{pointmover/point_movement_harmonic_gamma0.030_epsilon\eps .png}};
            \node at (10.75,0) {\includegraphics[width=2.25cm]{pointmover/point_movement_harmonic_gamma0.040_epsilon\eps .png}};
        \end{tikzpicture}
        \caption{Cochains, $\varepsilon_0 = 0.05$}
    \end{subfigure}

    \begin{subfigure}{\textwidth}
        \centering
        \begin{tikzpicture}
        \tikzmath{\eps = 0.060;}
            \node at (0,0) {\includegraphics[width=5cm]{pointmover/death_minus_birth_harmonic_epsilon\eps .png}};
            \node at (4,0) {\includegraphics[width=2.25cm]{pointmover/point_movement_harmonic_gamma0.010_epsilon\eps .png}};
            \node at (6.25,0) {\includegraphics[width=2.25cm]{pointmover/point_movement_harmonic_gamma0.020_epsilon\eps .png}};
            \node at (8.5,0) {\includegraphics[width=2.25cm]{pointmover/point_movement_harmonic_gamma0.030_epsilon\eps .png}};
            \node at (10.75,0) {\includegraphics[width=2.25cm]{pointmover/point_movement_harmonic_gamma0.040_epsilon\eps .png}};
        \end{tikzpicture}
        \caption{Cochains, $\varepsilon_0 = 0.06$}
    \end{subfigure}

    \begin{subfigure}{\textwidth}
        \centering
        \begin{tikzpicture}
            \node at (0,0) {\includegraphics[width=5cm]{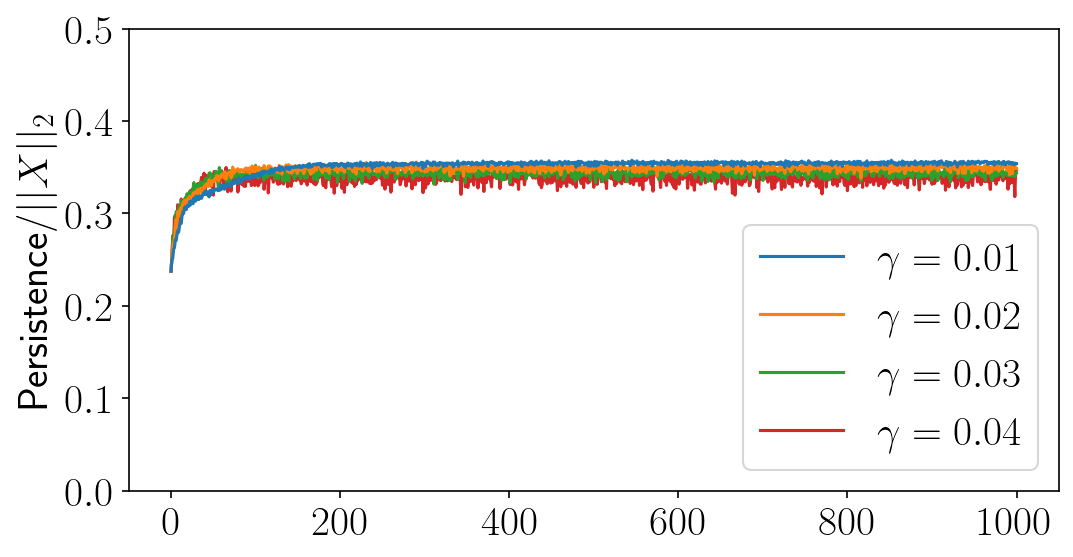}};
            \node at (4,0) {\includegraphics[width=2.25cm]{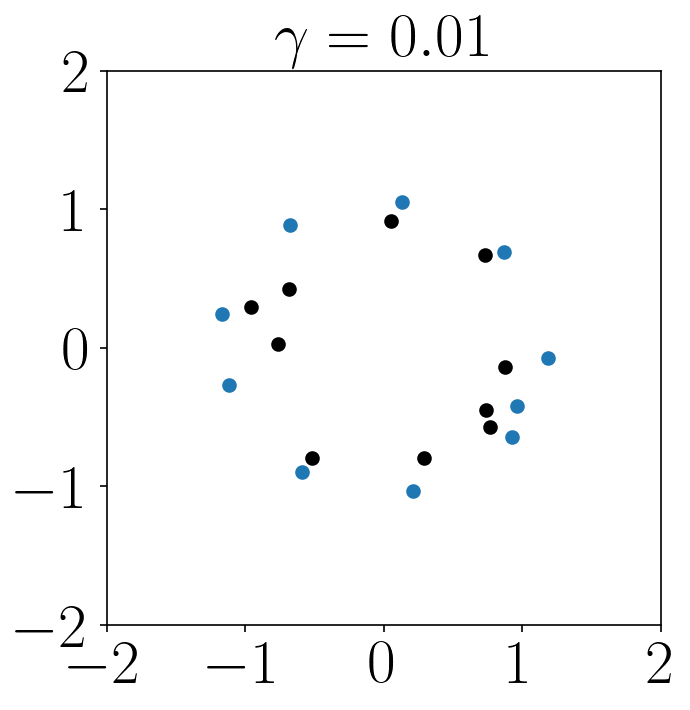}};
            \node at (6.25,0) {\includegraphics[width=2.25cm]{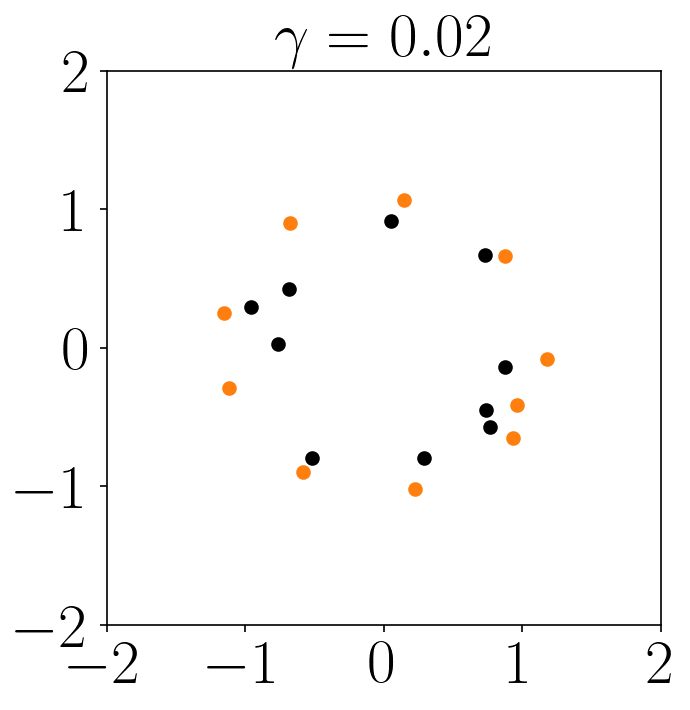}};
            \node at (8.5,0) {\includegraphics[width=2.25cm]{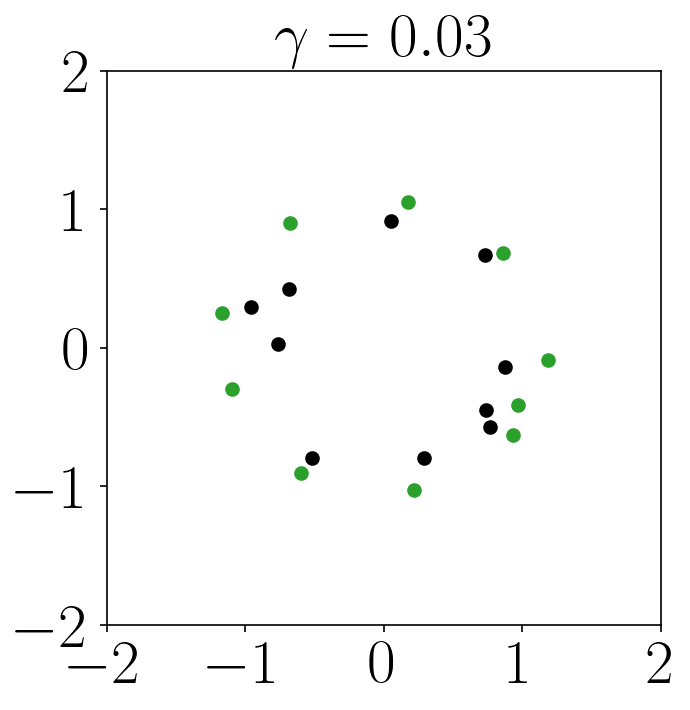}};
            \node at (10.75,0) {\includegraphics[width=2.25cm]{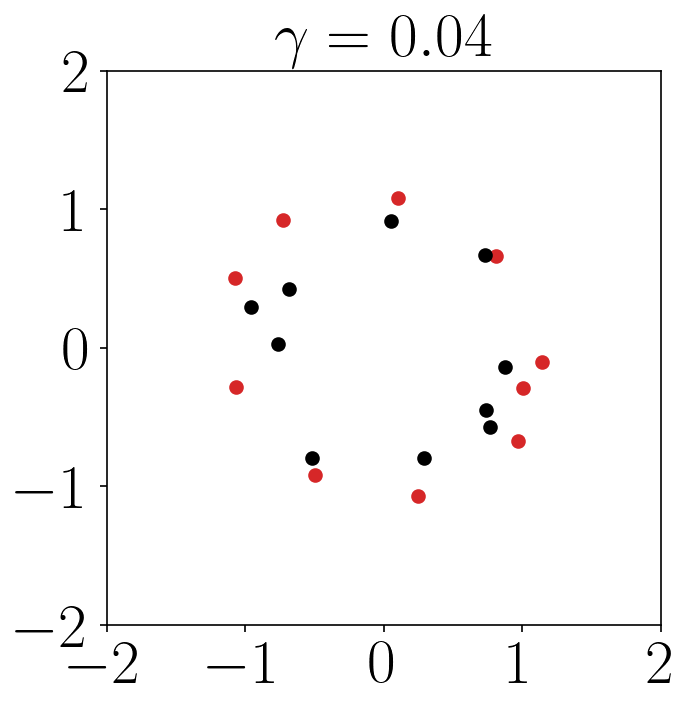}};
        \end{tikzpicture}
        \caption{Simplices}
    \end{subfigure}
    
    \caption{Maximizing an $\rmH_1$ feature using either birth and death cochains, or birth and death simplices. We show how the persistence changes over gradient descent iterations (left). We also show the initial point cloud (in black) and final point cloud for each scenario (right).}
\end{figure}

\subsection{Random trials on other MNIST digits}

\begin{figure}
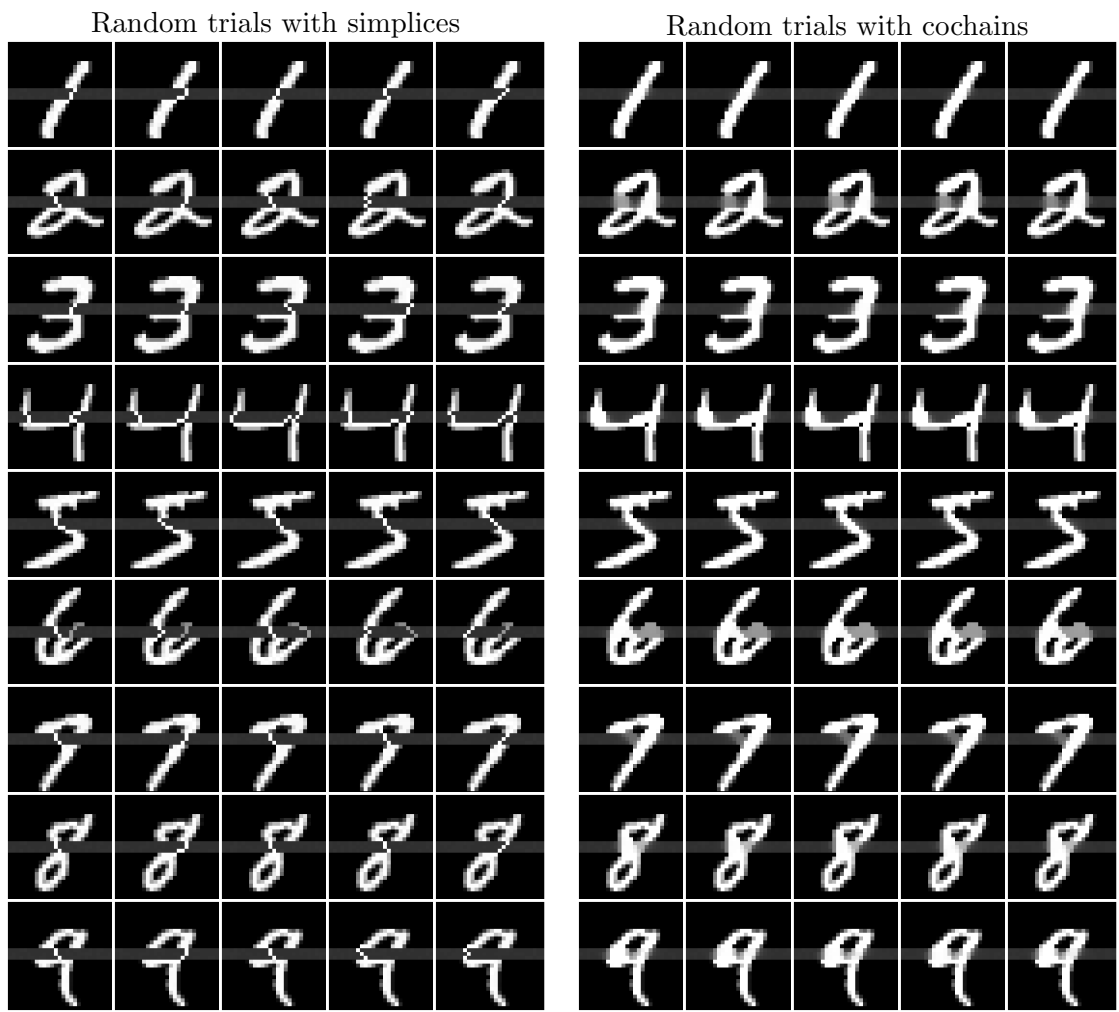

    \centering
    \begin{tikzpicture}[scale=0.95]
    \node at (4.5, -0.5) {Random trials with simplices};
        \node at (12.5, -0.5) {Random trials with cochains};
    \foreach \s in {2,3,4,5,6,7,8,9,10}{
        \foreach \i in {1,2,3,4,5} {
        \node at (1.5*\i, -1.5*\s+1.5) {
        \includegraphics[width=1.5cm]{imageshader/MNIST_simplices_sample\s_trial\i _final.png}
        };
        }
        \foreach \i in {1,2,3,4,5} {
        \node at (1.5*\i+8, -1.5*\s+1.5) {
        \includegraphics[width=1.5cm]{imageshader/MNIST_cochains_sample\s_trial\i _final.png}
        };
        }
    }
    \end{tikzpicture}
    \caption{Five random trials of the digit repair task for other digits, with slightly different noise each time.}
\end{figure}

\subsection{Arctic ice analysis for other years} We run the Arctic ice analysis feature selection method for years other than 2006. We exclude 2009, however, due to incomplete data.

\begin{figure}
    \centering
    \includegraphics[width=0.19\linewidth]{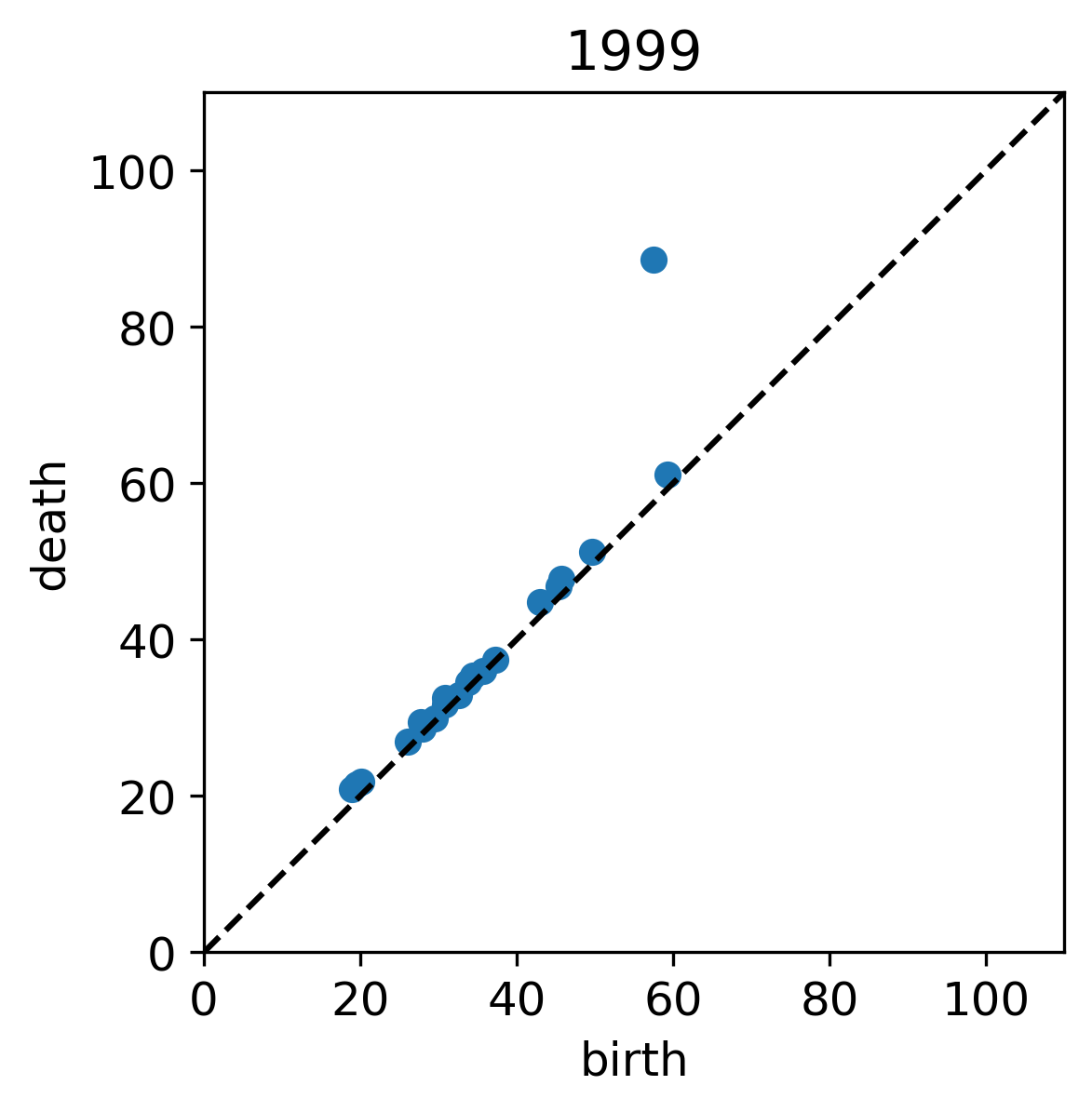}%
    \includegraphics[width=0.19\linewidth]{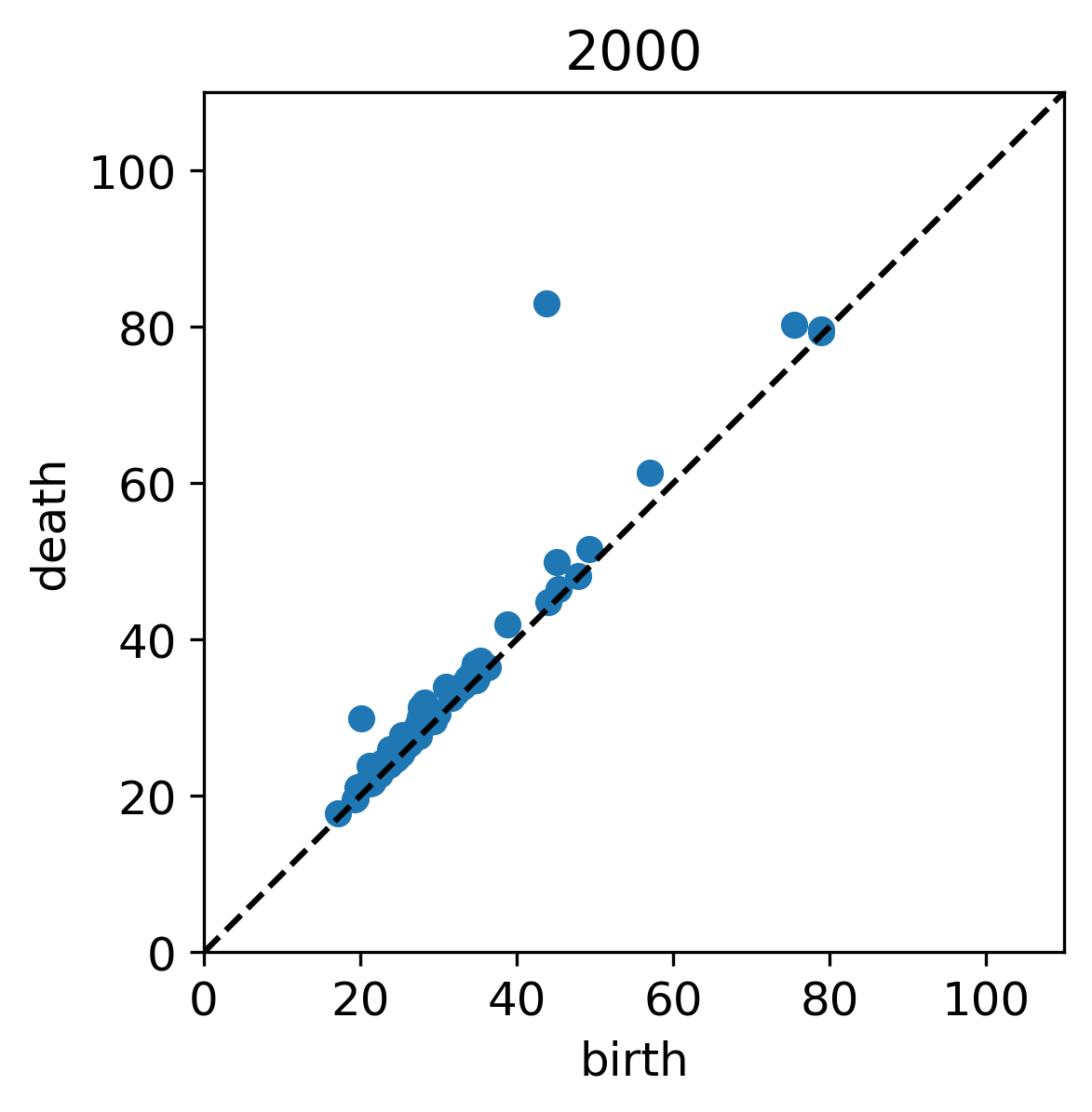}%
    \includegraphics[width=0.19\linewidth]{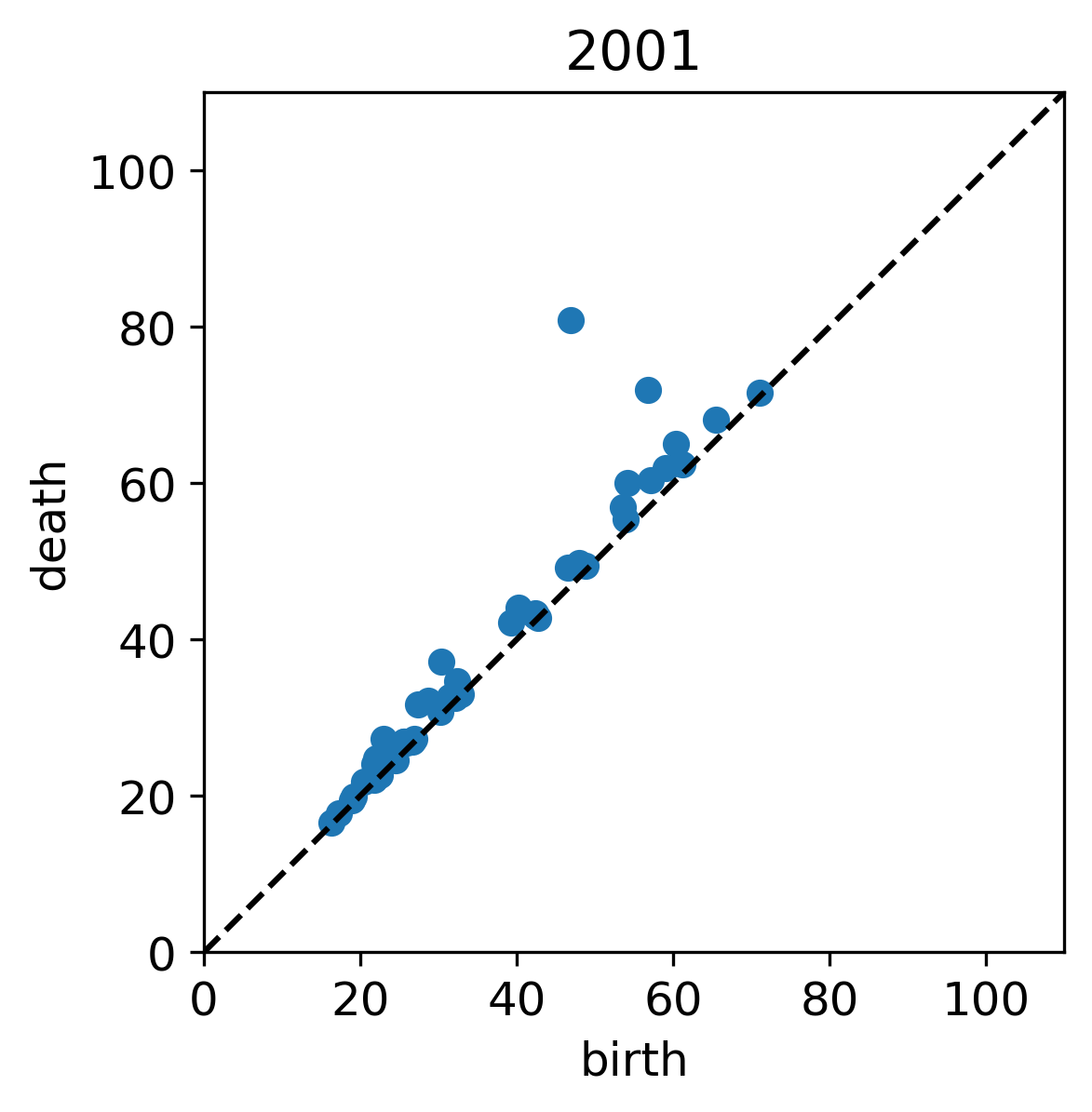}%
    \includegraphics[width=0.19\linewidth]{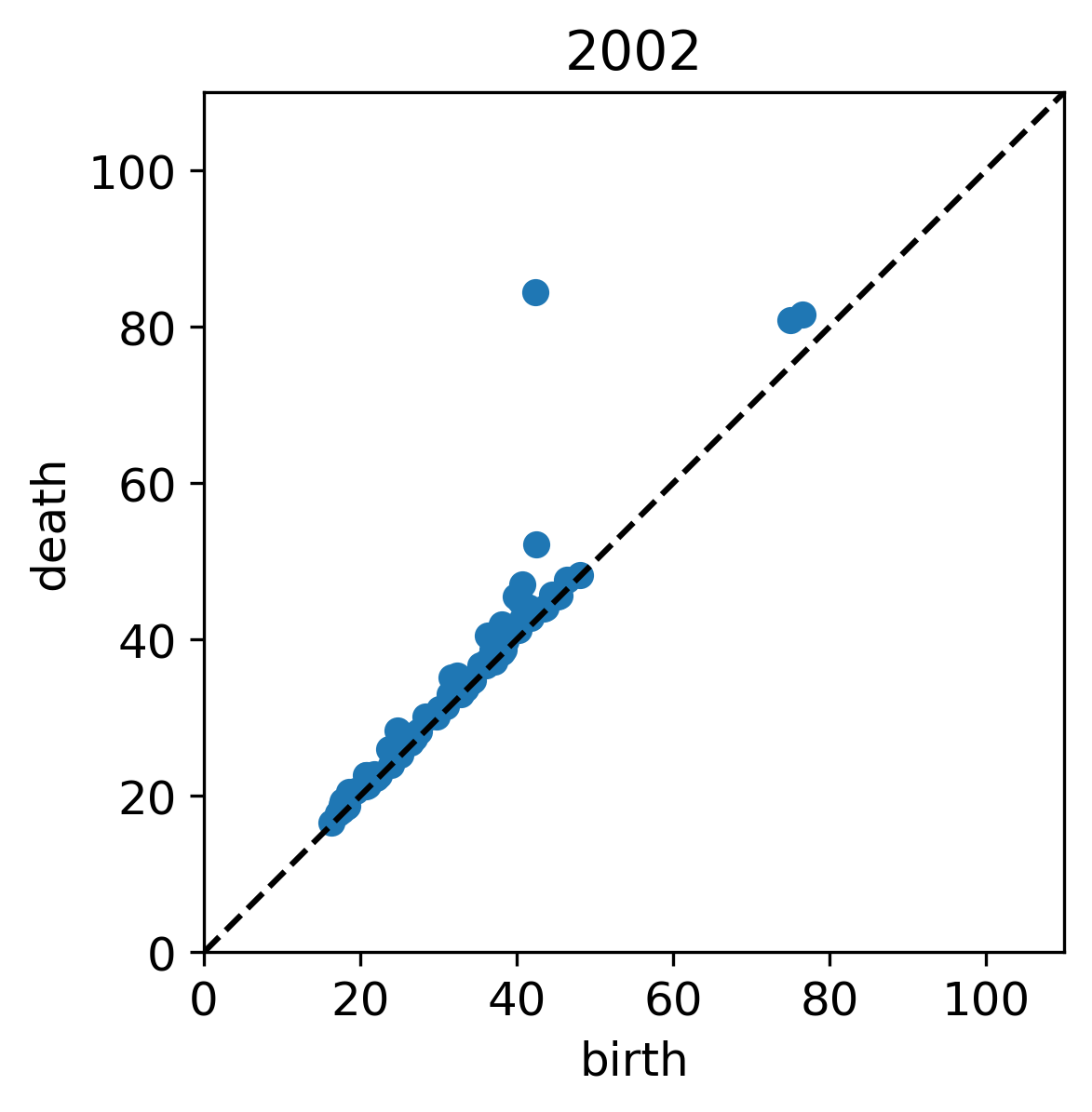}%
    \includegraphics[width=0.19\linewidth]{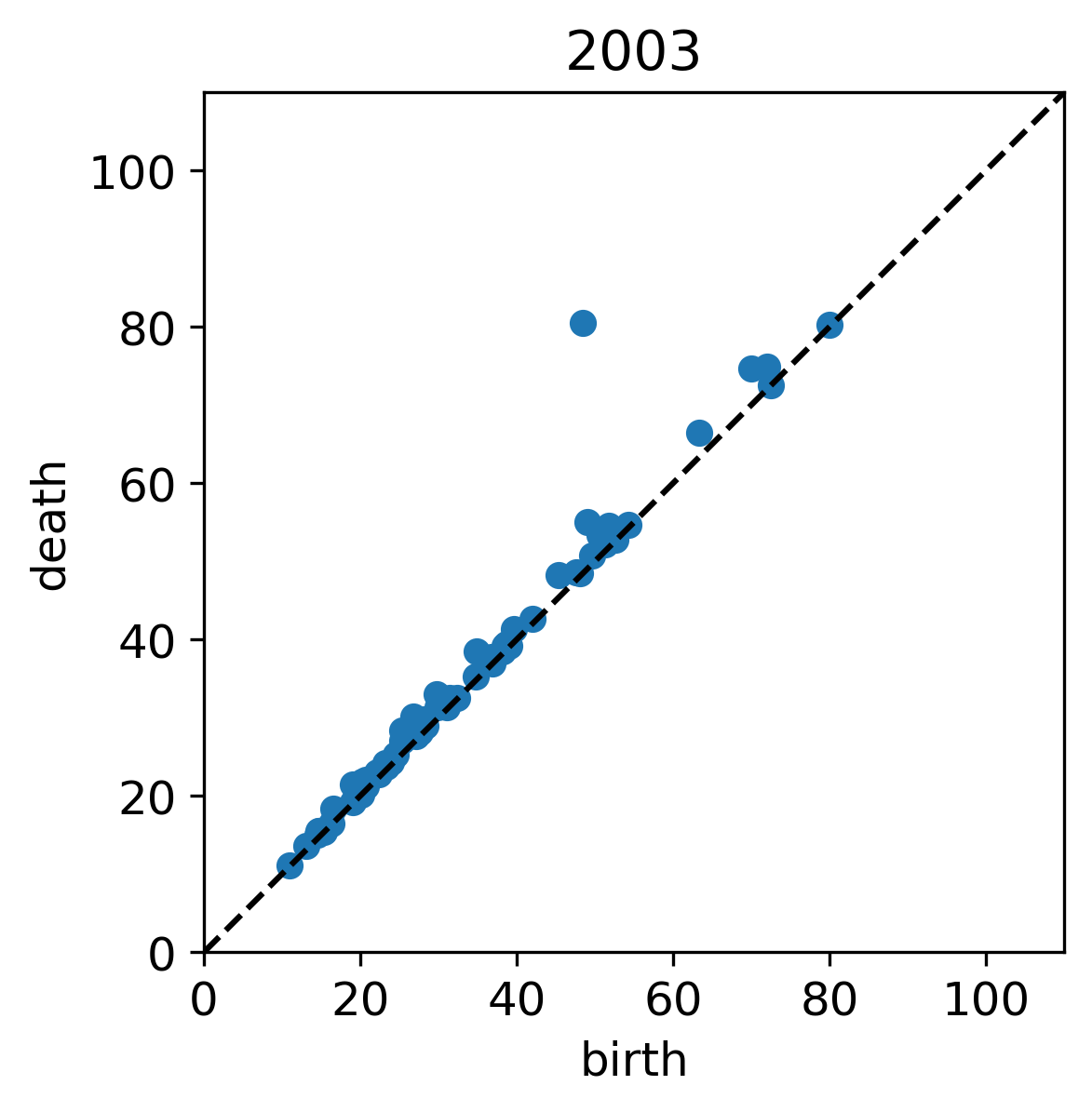}%
    
    \includegraphics[width=0.19\linewidth]{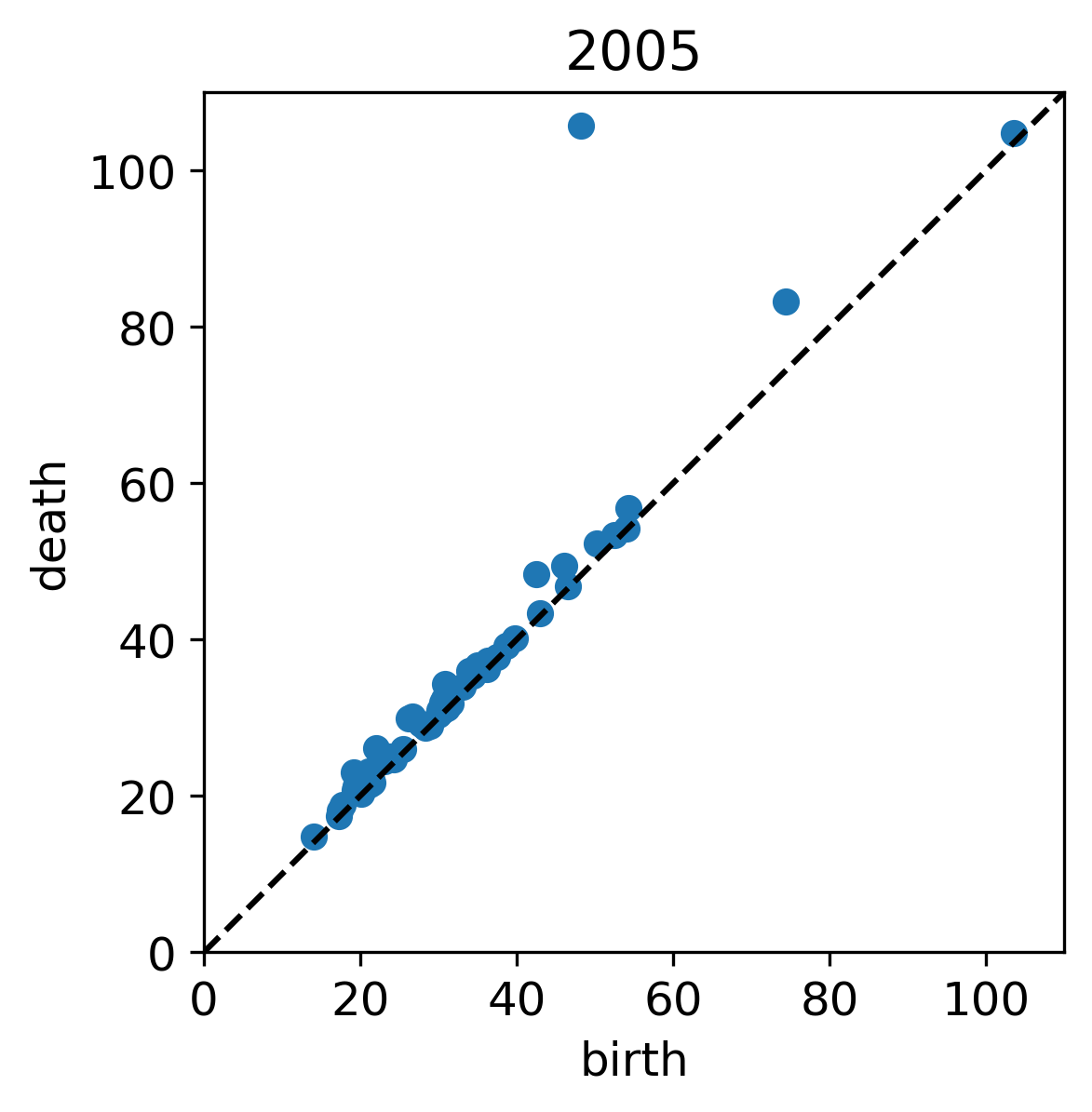}%
    \includegraphics[width=0.19\linewidth]{arctic/ice_eggs_2006_PD.png}%
    \includegraphics[width=0.19\linewidth]{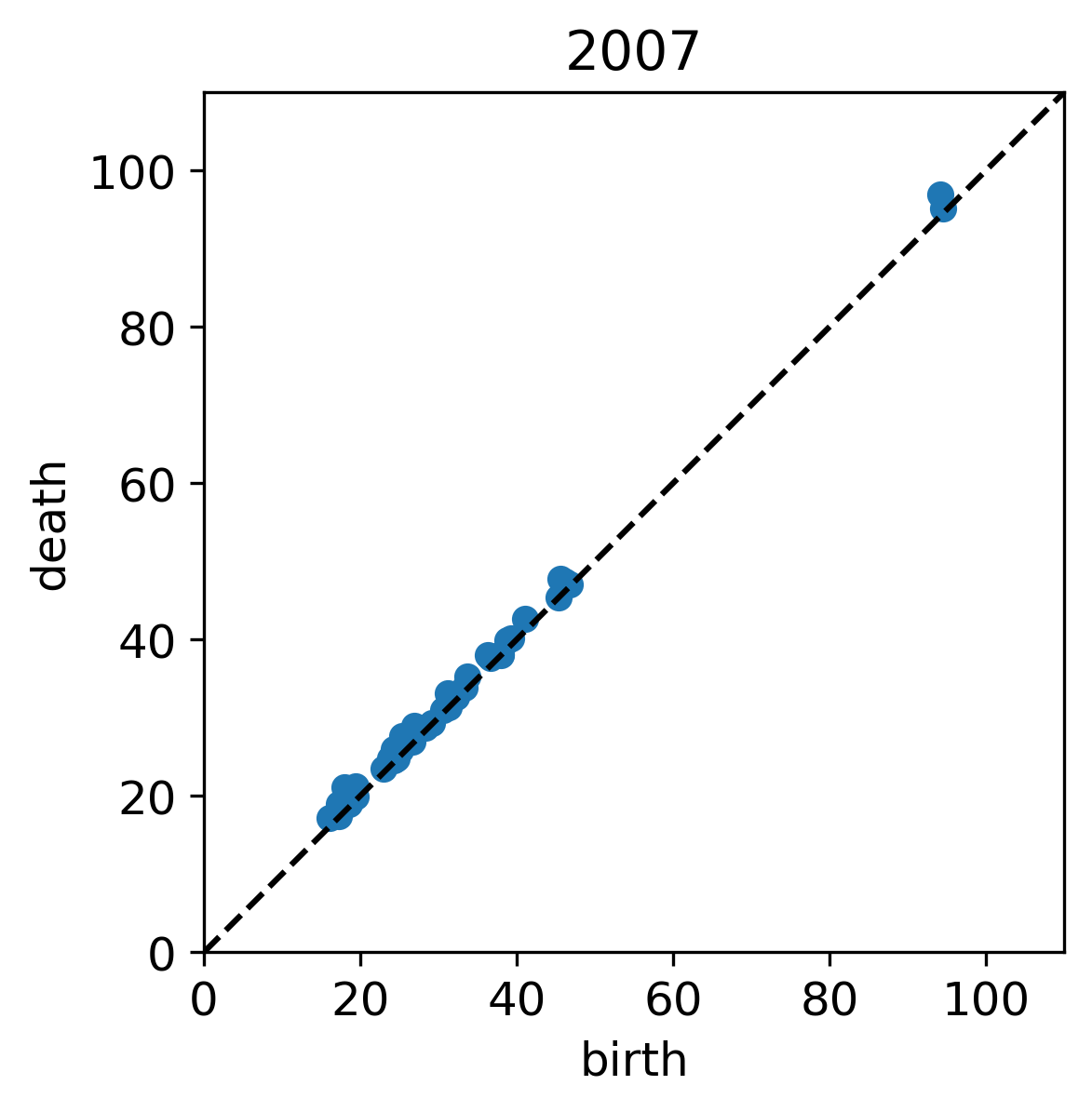}%
    \includegraphics[width=0.19\linewidth]{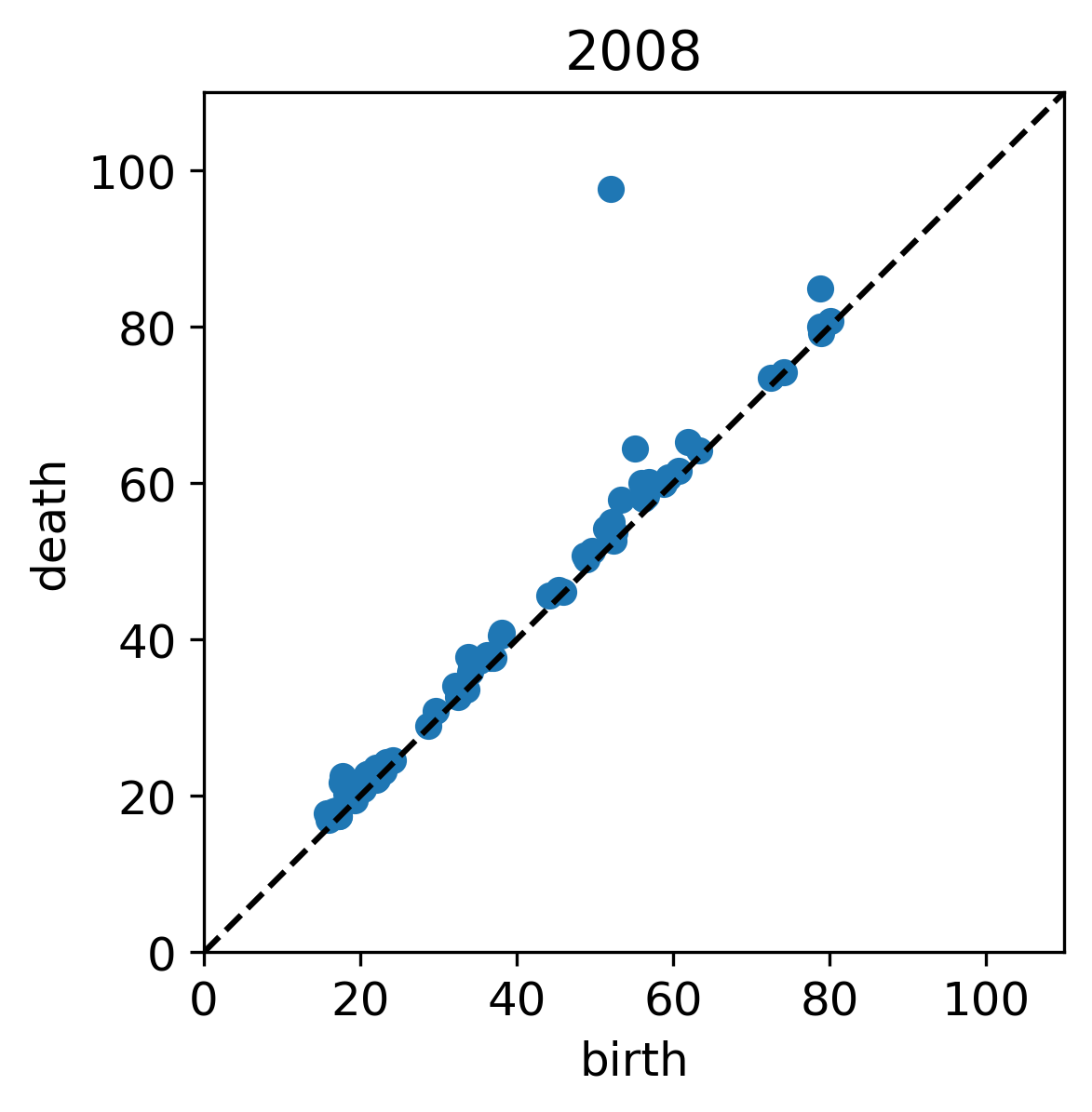}%
    \includegraphics[width=0.19\linewidth]{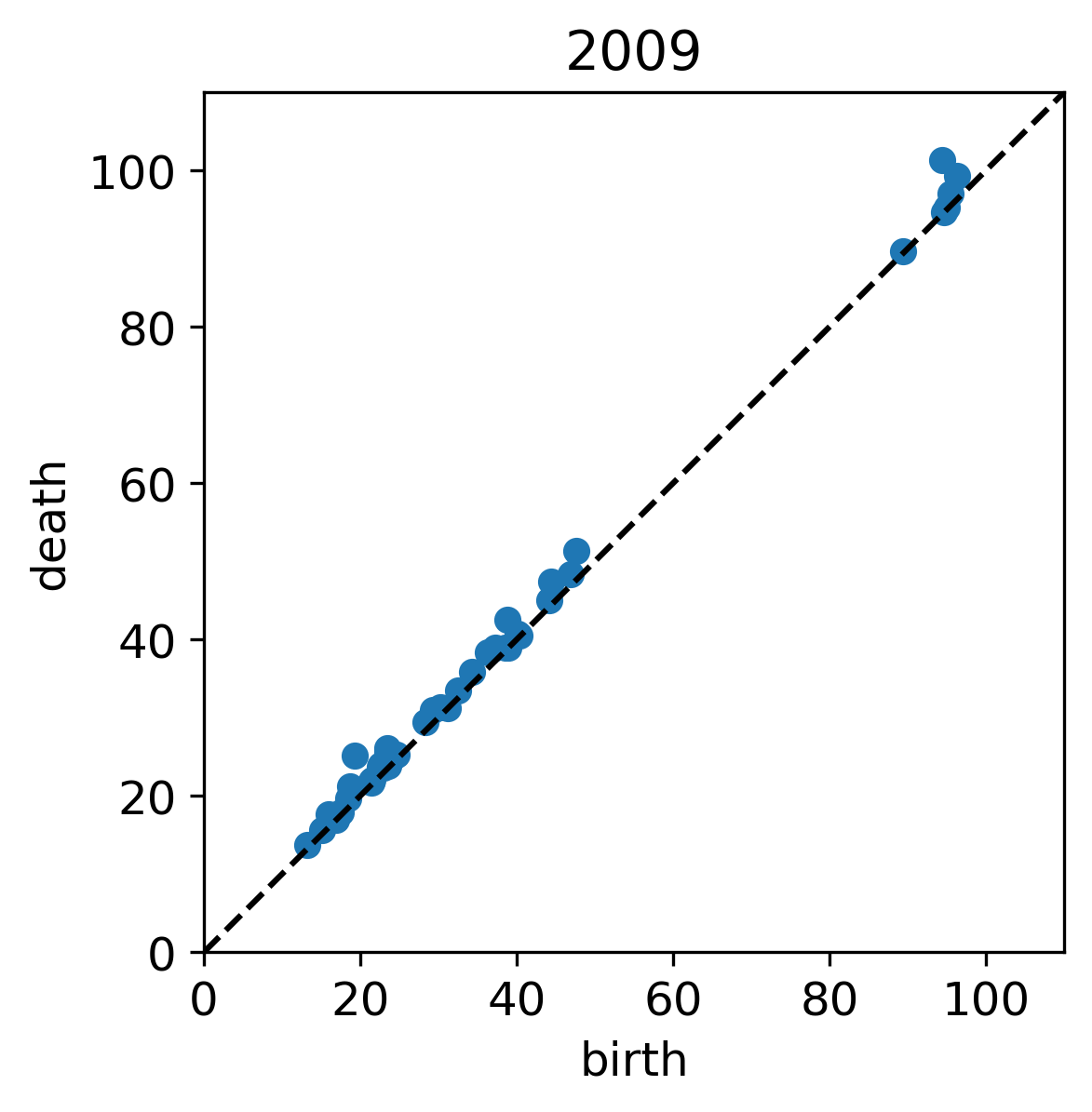}%
    
    \caption{Persistence diagrams of down-sampled images of Arctic ice extents converted into a point cloud for each year.}
\end{figure}

\foreach \i in {1999,2000,2001,2002,2003,2004,2005,2006,2007,2008}{

\begin{figure}[h]
        \centering
    \begin{tikzpicture}[scale=1.2]
    \node at (-3,2) {\small mask};
    \node at (0,2) {\small projection};
    \node at (4.5, 2) {\small principal components};
    \node[rotate=90] at (-5, 0) {\small raw data};
        \node at (0,0) {\includegraphics[height=3.4cm]{arctic/ice_eggs_\i _PCA_uniform.png}};
        \node at (3,0) {\includegraphics[height=3.2cm]{arctic/ice_eggs_\i _PC1_uniform.png}};
        \node at (6,0) {\includegraphics[height=3.2cm]{arctic/ice_eggs_\i _PC2_uniform.png}};
    \node[rotate=90] at (-5, -3) {\small masked data};
        \node at (-3,-3.05) {\includegraphics[height=3cm]{arctic/ice_eggs_\i _weights.png}};
        \node at (0,-3) {\includegraphics[height=3.4cm]{arctic/ice_eggs_\i _PCA_learned.png}};
        \node at (3,-3) {\includegraphics[height=3.2cm]{arctic/ice_eggs_\i _PC1_learned.png}};
        \node at (6,-3) {\includegraphics[height=3.2cm]{arctic/ice_eggs_\i _PC2_learned.png}};
    \end{tikzpicture}
    
    \caption{\i}
\end{figure}
}

\end{document}